\documentclass[reqno,tbtags,a4paper,12pt]{amsart}

\usepackage{amsmath,amssymb,amsthm,amsfonts,amscd,array,graphicx}
\usepackage[in]{fullpage}

\newcommand{\ba}{\mathbf{a}}
\newcommand{\bb}{\mathbf{b}}
\newcommand{\bc}{\mathbf{c}}
\newcommand{\be}{\mathbf{e}}
\newcommand{\bn}{\mathbf{n}}
\newcommand{\bm}{\mathbf{m}}

\newcommand{\bx}{\mathbf{x}}
\newcommand{\blambda}{\boldsymbol{\lambda}}
\newcommand{\bxi}{\boldsymbol{\xi}}
\newcommand{\bmu}{\boldsymbol{\mu}}

\newcommand{\btau}{\boldsymbol{\tau}}
\newcommand{\btheta}{\boldsymbol{\theta}}
\newcommand{\bsigma}{\boldsymbol{\sigma}}

\newcommand{\ah}{a}
\newcommand{\bh}{b}

\newcommand{\sn}{\mathop{\mathrm{sn}}\nolimits}
\newcommand{\dn}{\mathop{\mathrm{dn}}\nolimits}
\newcommand{\cn}{\mathop{\mathrm{cn}}\nolimits}

\newcommand{\sign}{\mathop{\mathrm{sign}}\nolimits}
\newcommand{\dist}{\mathop{\mathrm{dist}}\nolimits}

\newcommand{\spa}{\mathop{\mathrm{span}}\nolimits}
\renewcommand{\Re}{\mathop{\mathrm{Re}}\nolimits}
\renewcommand{\Im}{\mathop{\mathrm{Im}}\nolimits}

\newcommand{\R}{\mathbb{R}}

\newcommand{\RP}{\mathbb{RP}}
\newcommand{\CP}{\mathbb{CP}}

\newcommand{\E}{\mathbb{E}}
\newcommand{\C}{\mathbb{C}}
\newcommand{\Z}{\mathbb{Z}}

\newcommand{\X}{\mathbb{X}}
\newcommand{\V}{\mathbb{V}}
\newcommand{\F}{\mathbb{F}}

\newcommand{\bS}{\mathbb{S}}
\newcommand{\B}{\mathbb{B}}
\newcommand{\CA}{\mathcal{A}}
\newcommand{\CB}{\mathcal{B}}
\newcommand{\CE}{\mathcal{E}}
\newcommand{\I}{\mathcal{I}}

\newtheorem{theorem}{Theorem}[section]
\newtheorem{propos}[theorem]{Proposition}
\newtheorem{cor}[theorem]{Corollary}
\newtheorem{lem}[theorem]{Lemma}

\newtheorem{assump}[theorem]{Assumption}
\newtheorem{problem}[theorem]{Problem}

\theoremstyle{definition}
\newtheorem{constr}[theorem]{Construction}
\newtheorem{defin}[theorem]{Definition}
\newtheorem{remark}[theorem]{Remark}

\author{Alexander A. Gaifullin}

\thanks{The work was partially supported by the Russian Foundation for Basic Research (projects 13-01-12469 and 13-01-91151), by a grant of the President of the Russian Federation (project MD-2969.2014.1), and by a grant from Dmitry Zimin's ``Dynasty'' foundation.}

\title{Flexible cross-polytopes in spaces of constant curvature}

\date{}

\address{Steklov Mathematical Institute, Moscow, Russia\newline
${}$\hspace{4.3mm}Moscow State University, Moscow, Russia\newline 
${}$\hspace{4.3mm}Kharkevich Institute for Information Transmission Problems, Moscow, Russia}

\email{agaif@mi.ras.ru}

\begin{document}

\begin{abstract}
We construct self-intersected flexible cross-polytopes in the spaces of constant curvature, that is, Euclidean spaces~$\E^n$, spheres~$\bS^n$, and Lobachevsky spaces~$\Lambda^n$ of all dimensions~$n$. In dimensions $n\ge 5$, these are the first examples of flexible polyhedra. Moreover, we classify all flexible cross-polytopes in each of the spaces~$\E^n$, $\bS^n$, and~$\Lambda^n$. For each type of flexible cross-polytopes, we provide an explicit parametrization of the flexion by either rational or elliptic functions.
\end{abstract}

\maketitle

\begin{flushright}
\textit{To my Advisor Victor Matveevich Buchstaber\\ on the occasion of his seventieth birthday}
\end{flushright}

\section{Introduction}

Theory of flexible polyhedra started from very classical results of 19th century. In 1813 Cauchy~\cite{Cau13} proved that any convex polytope is rigid, that is, admits no flexion. 
A \textit{flexion\/} of a polyhedron in three-space is defined in the following way. Suppose that faces of a polyhedron are made of some rigid material, and we have hinges at edges of the polyhedron. So we allow deformations of the polyhedron such that each face of it remains congruent to itself during the deformation, and the dihedral angles vary continuously. If a polyhedron admits a non-trivial deformation of such kind, then we say that this polyhedron is \textit{flexible\/}. (A deformation is called \textit{non-trivial\/} if it is not induced by an ambient rotation of the space.) Otherwise, the polyhedron is called \textit{rigid\/}.  

In 1897 Bricard~\cite{Bri97} constructed examples of self-intersected flexible octahedra in the $3$-dimensional Euclidean space~$\E^3$ (cf.~\cite{Ben12}). Let us explain in more detail what we mean under a flexible self-intersected octahedron. Under an \textit{octahedron\/} we mean a polyhedron that has the  combinatorial type of the standard regular octahedron. An octahedron is uniquely determined by positions of its vertices, which we denote by $\ba_1$, $\ba_2$, $\ba_3$, $\bb_1$, $\bb_2$, and~$\bb_3$. A pair of vertices of the octahedron forms an edge if and only if it is not a pair $(\ba_p,\bb_p)$. A triple of vertices of the octahedron forms a face if and only if it  does not contain a pair $(\ba_p,\bb_p)$. The figure~$P^3$ consisting of all these edges and faces  is called an \textit{octahedron\/} with the vertices $\ba_1$, $\ba_2$, $\ba_3$, $\bb_1$, $\bb_2$, and~$\bb_3$ if every face is non-degenerate, i.\,e., its vertices do not belong to a line. An octahedron is called \textit{embedded\/} if the relative interiors of its edges and faces are disjoint, and is called \textit{self-intersected\/} otherwise.
 Now, suppose that the points $\ba_1$, $\ba_2$, $\ba_3$, $\bb_1$, $\bb_2$, and~$\bb_3$ vary continuously depending on a parameter~$u$ such that each face of the octahedron~$P^3(u)$ remains congruent to itself. The family of octahedra~$P^3(u)$ is called a \textit{flex\/} if the octahedra $P^3(u_1)$ and~$P^3(u_2)$ are not congruent to each other whenever $u_1$ and~$u_2$ are close to each other and $u_1\ne u_2$. Bricard classified all flexible octahedra in~$\E^3$. He constructed three types of flexible octahedra, which are now called \textit{line-symmetric flexible octahedra\/}, \textit{plane-symmetric flexible octahedra\/}, and \textit{skew flexible octahedra\/} respectively, and proved that any flexible octahedron is of one of these types. All flexible octahedra in~$\E^3$ are self-intersected. 

For a long time a question on the existence of embedded flexible polyhedra in~$\E^3$ remained open. Only in 1977 Connelly~\cite{Con77} constructed the first example  of an embedded flexible polyhedron in~$\E^3$. The simplest known example of an embedded flexible polyhedron due to Steffen has $9$ vertices.  Let us also mention a result by Shtogrin~\cite{Sto13} on the existence of an embedded flexible polyhedron of arbitrary genus such that no its handle remains rigid under the flexion. This result uses an earlier example of a self-intersected flexible polyhedron with the topology of a torus due to Alexandrov~\cite{Ale95}.

There are two main directions for generalizing flexible polyhedra in~$\E^3$. First, one can replace~$\E^3$ with the spaces of constant positive or negative curvature, i.\,e., the sphere~$\bS^3$ or the Lobachevsky space~$\Lambda^3$. Stachel~\cite{Sta06} showed that all three types of Bricard's flexible octahedra exist both in~$\bS^3$ and in~$\Lambda^3$. However, the complete classification of flexible octahedra in the sphere and in the Lobachevsky space has not been known. The second direction is to generalize theory of flexible polyhedra to higher dimensions. Certainly, the first question is the existence of flexible polyhedra in spaces~$\E^n$, $\bS^n$, and~$\Lambda^n$ for $n\ge 4$. The first example of a flexible self-intersected polyhedron in~$\E^4$ is due to Walz. More examples were obtained by Stachel~\cite{Sta00}. All these examples are flexible $4$-dimensional cross-polytopes, that is, polyhedra with the combinatorial type of the regular $4$-dimensional cross-polytope. (Recall that the regular cross-polytope is the regular polytope dual to the cube.) Nevertheless, until now there have been no examples of flexible polyhedra in dimensions $5$ and higher. Moreover, the usual expectation is that flexible polyhedra do not exist in dimensions greater than $4$.  The intuition behind this expectation is that simplicial polyhedra in spaces of high dimensions must have ``too many edges''. Hence the system of equations describing flexions becomes highly overdetermined.

The aim of this paper is to prove that this intuition is wrong, and flexible polyhedra exist in all three spaces~$\E^n$, $\bS^n$, and~$\Lambda^n$ of arbitrary dimension~$n$. In each of the spaces~$\E^n$, $\bS^n$, and~$\Lambda^n$ we construct examples of (self-intersected) flexible cross-polytopes. Moreover, we classify all flexible cross-polytopes in all spaces~$\E^n$, $\bS^n$, and~$\Lambda^n$. 
In particular, it will follow that Stachel's examples cover only very small part of the set of all flexible cross-polytopes in~$\E^4$. Another interesting corollary is that even in~$\bS^3$, alongside with the known three types of  flexible octahedra that generalize Bricard's three types of flexible octahedra in~$\E^3$ (cf.~\cite{Sta06}), there exists a new type  of flexible octahedra that does not have analogues in~$\E^3$ and~$\Lambda^3$. We call this new type of flexible octahedra~\textit{exotic}.   

A very important motivation for the problem on existence of flexible polyhedra in higher dimensions is the so-called \textit{Bellows Conjecture\/}. Connelly~\cite{Con78} conjectured that the volume of any flexible polyhedron in~$\E^3$ does not change under the flexion. (For self-intersected polyhedra one should replace the volume with a naturally defined generalized volume.) This conjecture was proved by Sabitov in 1996 \cite{Sab96}, \cite{Sab98a},~\cite{Sab98b}. Another proof was given in~\cite{CSW97}. Recently the author has generalized Sabitov's theorem to flexible polyhedra in  Euclidean spaces of arbitrary dimensions $n\ge 4$~\cite{Gai11}, \cite{Gai12}. Nevertheless, it  remained unknown whether this result is non-empty, i.\,e., whether there exists at least one flexible polyhedron in~$\E^n$ for $n\ge 5$. The same concerns a theorem due to Alexander~\cite{Ale85} that claims that the total mean curvature of any flexible polyhedron in~$\E^n$ does not change under the flexion.  
The results of the present paper show that flexible polyhedra actually exist in all spaces~$\E^n$. 

Our method is algebraic. Following the original approach due to Bricard~\cite{Bri97}, we study relations of the form
\begin{equation}\label{eq_main_intro}
At^2t'^2+Bt^2+2Ctt'+Dt'^2+E=0,
\end{equation}
where~$t$, $t'$ are the tangents of the halves of dihedral angles of the cross-polytope adjacent to a facet~$\Delta$ of it. We obtain a system of ${n\choose 2}$ algebraic equations in $n$ variables, and we study the conditions for this system to have a one-parametric family of solutions. The idea is to interpret equation~\eqref{eq_main_intro} as the addition law for Jacobi's elliptic functions. This idea is very natural from the viewpoint of elliptic functions. In theory of flexible polyhedra it was first introduced by Izmestiev~\cite{Izm} who studied  parametrizations for deformations of quadrilaterals in~$\bS^2$, and applied the obtained results to the study of so-called Kokotsakis meshes. The results of the present paper can be considered as a multidimensional generalization of Izmestiev's results. Indeed, for $n=2$, our parametrization of flexible cross-polytopes turns exactly into Izmestiev's parametrization of flexible quadrilaterals. (Recall that a two-dimensional cross-polytope is a quadrilateral.) Elliptic functions (namely, the Weierstrass  $\wp$-function) were first used to parametrize flexible polyhedra (in the three-dimensional space) by Connelly~\cite{Con80}. However, his method for introducing the elliptic parametrization differs drastically from the method used by Izmestiev and in the present paper.

This paper is organized as follows. In Section~\ref{section_but} we introduce some definitions and notation. In Sections~\ref{section_Bri} and~\ref{section_EPBQ} we establish a connection between flexible cross-polytopes, and certain algebraic objects, which we call \textit{even poly-biquadratic curves\/} (or \textit{EPBQ-curves\/}). In Section~\ref{section_simplest} we construct the simplest family of flexible cross-polytopes. Notice that this result already yields the existence of flexible cross-polytopes in all spaces~$\E^n$, $\bS^n$, and~$\Lambda^n$. Sections~\ref{section_ap_C} and~\ref{section_ap_R} are devoted to the classification of EPBQ-curves. Based on these results, we obtain in Section~\ref{section_classify} the classification of all flexible cross-polytopes (Theorem~\ref{theorem_classify}). In Section~\ref{section_exist} we prove that all constructed families of flexible cross-polytopes are actually non-empty. Section~\ref{section_concl} contains some conclusive remarks and two open problems.

The author is grateful to V.\,M.\,Buchstaber, A.\,V.\,Penskoi, I.\,Kh.\,Sabitov, M.\,B.\,Skopenkov, A.\,P.\,Veselov, and E.\,S.\,Zaputryaeva for useful comments. The author also wishes to thank I.\,V.\,Izmestiev for sending him an unpublished version of manuscript~\cite{Izm}, and for a fruitful discussion. 
{\sloppy

}

\section{Cross-polytopes and butterflies}\label{section_but}

Let $\X^n$ be one of the three spaces~$\E^n$, $\bS^n$, and~$\Lambda^n$.

\begin{defin}
Let $\ba_1,\ldots,\ba_n$, $\bb_1,\ldots,\bb_n$ be points in~$\X^n$ such that any $n$ points  
\begin{equation}\label{eq_points}
\ba_{p_1},\ldots,\ba_{p_{n'}}, \bb_{q_1},\ldots,\bb_{q_{n''}},
\end{equation}
where $n'+n''=n$ and $p_1,\ldots,p_{n'}$, $q_1,\ldots,q_{n''}$ is a permutation of $1,\ldots,n$,
do not belong to a hyperplane in~$\X^n$. Consider all $(n-1)$-simplices spanned by the sets of points of the form~\eqref{eq_points}, and all their subsimplices. The figure consisting of all these simplices is called the \textit{cross-polytope\/} with vertices $\ba_1,\ldots,\ba_n$, $\bb_1,\ldots,\bb_n$, and the simplices themselves are called \textit{faces\/} of the cross-polytope. Faces of dimension~$n-1$ are called \textit{facets\/}. A continuous deformation $P^n(u)$ of a cross-polytope $P^n$ is called a \textit{flex\/} if all faces of~$P^n(u)$ remain congruent to themselves during the deformation, but the cross-polytopes~$P^n(u_1)$ and~$P^n(u_2)$ are not congruent to each other whenever $u_1$ and~$u_2$ are close to each other and $u_1\ne u_2$. A cross-polytope~$P^n$ is called \textit{essential\/} if any dihedral angle of it is neither zero nor straight. A flex~$P^n(u)$ is called \textit{essential\/} if the cross-polytope~$P^n(u)$ is essential for all but finitely many of~$u$.   
\end{defin}

The concept of an essential cross-polytope is very important to avoid degenerate situations. Indeed, let us consider a ``twice covered square'', i.\,e., the octahedron in~$\E^3$ with vertices $\ba_1=\bb_1=(0,0,0)$,
$\ba_2=(1,0,0)$, $\bb_2=(-1,0,0)$, $\ba_3=(0,1,0)$, $\bb_3=(0,-1,0)$. This octahedron admits a flex, which is just the bending of the square along the diagonal, such that $\ba_2(u)=(\cos u,0,\sin u)$ and all other vertices are fixed. Similar examples can be easily constructed in the spaces $\E^n$, $\bS^n$, and~$\Lambda^n$ of all dimensions. Certainly, they are not interesting. Hence our goal is to construct and classify essential flexes of cross-polytopes. Further we shall see that for any essential flex~$P^n(u)$ none of the dihedral angles of~$P^n(u)$ is constant under the flex.

Obviously, by rotations of~$\X^n$, we may achieve that the facet $\Delta=[\ba_1\ldots\ba_n]$ is fixed under the flexion.  For  $p=1,\ldots,n$, let~$F_p$ be the $(n-2)$-face of~$\Delta$ opposite to~$\ba_p$. For $p=1,\ldots,n$, let~$\Delta_p$ be the facet spanned by the vertex~$\bb_p$ and all vertices~$\ba_q$ such that $q\ne p$. Then $F_p$ is the common face of~$\Delta$ and~$\Delta_p$. The mechanism consisting of the $n+1$ non-degenerate $(n-1)$-dimensional simplices $\Delta$, $\Delta_1,\ldots,\Delta_n$ connected by the $n$ hinges at faces~$F_p$ will be called a \textit{butterfly\/}, and will be denoted by~$\B$. We shall always assume that the simplex~$\Delta$ is fixed. This does not restrict the flexions of~$\B$. 

The position of every wing~$\Delta_p$ of the butterfly~$\B$ is determined by the oriented dihedral angle~$\varphi_p$ between~$\Delta$ and~$\Delta_p$. (The sign of~$\varphi_p$ will be specified later.) Following Bricard~\cite{Bri97}, we introduce the parameters 
$$t_p=\tan(\varphi_p/2).$$
These parameters will always be considered as elements of $\R\cup\{\infty\}=\RP^1$.

Now, let us fix the set $\ell$ of edge lengths of a cross-polytope, and consider the set of all cross-polytopes with the prescribed set of edge lengths. The butterfly~$\B$ is uniquely determined by the lengths of all edges of the cross-polytope~$P^n$ except for the edges $[\bb_p\bb_q]$. The lengths of the edges~$[\bb_p\bb_q]$ changes as the wings of~$\B$ rotate. Let $\Xi(\ell)\subset (\RP^1)^n$ be the subset consisting of all points $(t_1,\ldots,t_n)$ such that the position of~$\B$ corresponding to parameters $t_1,\ldots,t_n$ yields the prescribed lengths~$\ell_{\bb_p\bb_q}$ of all edges~$[\bb_p\bb_q]$. The set~$\Xi(\ell)$ will be called the \textit{configuration space\/} of the cross-polytopes with the prescribed set of edge lengths~$\ell$. Obviously,  a flexible cross-polytope with the set of edge lengths~$\ell$ exists if and only if $\Xi(\ell)$ contains a continuous curve. Bellow we shall prove that $\Xi(\ell)$ is an algebraic variety, and write the equations for it.

Let us introduce some notation concerning butterflies. Denote by~$\Pi$ the hyperplane in~$\X^n$ containing~$\Delta$.

Suppose that $\X^n=\E^n$.  Let $\ah _1,\ldots,\ah _n$ be the lengths of the altitudes of the simplex~$\Delta$ drawn from the vertices $\ba_1,\ldots,\ba_n$ respectively, and let $\bh_1,\ldots,\bh_n$ be the lengths of  the altitudes of the simplices $\Delta_1,\ldots,\Delta_n$ respectively drawn from the vertices $\bb_1,\ldots,\bb_n$ respectively. Let $\bm$ be one of the two unit normal vectors to~$\Pi$. For $p=1,\ldots,n$, let~$\bn_p$ be the unit inner normal vector to the facet~$F_p$ of the simplex~$\Delta$. It is easy to see that 
\begin{gather}
%(\ba_p-\ba_q,\bm)=0,\qquad (\bn_p,\bm)=0,\\
\label{eq_an_E}
(\ba_p-\ba_q,\bn_r)=\ah _r\left(\delta_{pr}-\delta_{qr}\right),\qquad (\ba_p-\ba_q,\bm)=0,\\
\label{eq_n_sum}
\sum_{p=1}^n\ah _p^{-1}\bn_p=0,
\end{gather} 
where  $\delta_{pq}$ is the Kronecker symbol. 

Now, suppose that $\X^n$ is either~$\bS^n$ or $\Lambda^n$. We always realise~$\bS^n$ as the standard unit sphere in the Euclidean space~$\E^{n+1}$, and we always realise~$\Lambda^n$ as a sheet of the hyperboloid given by $(\bx,\bx)=-1$ in the pseudo-Euclidean space~$\E^{n,1}$. To unify the notation, we put~$\V=\E^{n+1}$ in the first case, and $\V=\E^{n,1}$ in the second case.
 
In the spherical case, we denote by $\ah _1,\ldots,\ah _n$ the \textit{sines\/} of the lengths of the altitudes of~$\Delta$ drawn from the vertices $\ba_1,\ldots,\ba_n$ respectively, and we denote by $\bh_1,\ldots,\bh_n$ the \textit{sines\/} of the lengths of the the altitudes of the simplices $\Delta_1,\ldots,\Delta_n$ respectively drawn from the vertices $\bb_1,\ldots,\bb_n$ respectively.  In the Lobachevsky case, we denote by $\ah _1,\ldots,\ah _n$, $\bh_1,\ldots,\bh_n$ the \textit{hyperbolic sines\/} of the lengths of the corresponding altitudes. 

The normal vectors~$\bm$ and $\bn_1,\ldots,\bn_n$ are defined as follows. For each point $\bx\in \X^n$, the tangent space $T_{\bx}\X^n$ can be naturally identified with a subspace of~$\V$. Notice that in the Lobachevsky case the tangent space $T_{\bx}\Lambda^n$ consists of space-like vectors. Then, for $\bx\in \Delta$, $\bm$ is  one of the two unit vectors in~$T_{\bx}\X^n$ orthogonal to~$T_{\bx}\Delta$, and, for $\bx\in F_p$,  $\bn_p$ is  the unit vector in~$T_{\bx}\Delta$ orthogonal to~$T_{\bx}F_p$ and pointing inside~$\Delta$. It is easy to see that~$\bm$ and $\bn_p$, $p=1,\ldots,n$, considered as  vectors in~$\V$ are independent of~$\bx$. Then
\begin{equation}\label{eq_an_SL}
(\ba_p,\bn_q)=\ah _q\delta_{pq}, \qquad (\ba_p,\bm)=0.
\end{equation}

Now, let $\X^n$ be any of the spaces~$\E^n$, $\bS^n$, and~$\Lambda^n$.  We agree to choose the sign of the oriented angle $\varphi_p$ (defined modulo $2\pi m$, $m\in\Z$) between the facets~$\Delta$ and~$\Delta_p$   so that $\sin\varphi_p>0$ if $\Delta_p$ lies in the semi-space bounded by~$\Pi$ to which the vector~$\bm$ points, and  $\sin\varphi_p<0$ if $\Delta_p$ lies in the other semi-space bounded by~$\Pi$. We denote by~$\bb_p^0$ the position of the vertex~$\bb_p$ corresponding to $\varphi_p=0$. Then $\bb_p^0\in\Pi$, and $\bb_p^0$ and~$\ba_p$ lie on the same side from the plane spanned by~$F_p$.  We have
\begin{align*}
(\bb^0_p-\ba_q,\bn_p)&=\bh_p,\quad q\ne p,&&\text{if $\X^n=\E^n$,}\\
(\bb^0_p,\bn_p)&=\bh_p&&\text{if $\X^n=\bS^n$ or~$\Lambda^n$.}
\end{align*} 
 The rotation of the wing~$\Delta_p$ around~$F_p$ is described by
\begin{equation*}
\bb_p(\varphi_p)=\bb_p^0+\bh_p(\cos\varphi_p-1)\,\bn_p+\bh_p\sin\varphi_p\cdot\bm.
\end{equation*}
Equivalently, 
\begin{equation}\label{eq_t}
\bb_p(t_p)=\bb_p^0-\frac{2\bh_pt_p^2}{t_p^2+1}\,\bn_p+\frac{2\bh_pt_p}{t_p^2+1}\,\bm.
\end{equation}

Let $G=(g_{pq})$ be the Gram matrix of the vectors~$\bn_1,\ldots,\bn_n$. Then $G$ is a symmetric matrix with units on the diagonal. If $G$ is non-degenerate, then we denote the entries of the inverse matrix~$G^{-1}$ by~$g^{pq}$.

We also define the matrix~$H=(h_{pq})$ of size $n\times n$ by
\begin{equation}\label{eq_h}
h_{pq}=\left\{
\begin{aligned}
&\bh_p^{-1}(\bb_p^0-\ba_r,\bn_q),\quad r\ne q,&&\text{if $\X^n=\E^n$,}\\
&\bh_p^{-1}(\bb_p^0,\bn_q)&&\text{if $\X^n=\bS^n$ or~$\Lambda^n$.}
\end{aligned}
\right.
\end{equation}
(In the Euclidean case, by~\eqref{eq_an_E}, $h_{pq}$ is independent of~$r$.) Then $h_{pp}=1$ and
\begin{equation}\label{eq_bb0}
\bb_p^0=\bh_p\sum_{q=1}^n h_{pq}\ah _q^{-1}\ba_q.
\end{equation}

It is convenient to put $\bc_p=a_p^{-1}\ba_p$. Then the formulae~\eqref{eq_t} and~\eqref{eq_bb0} imply that 
\begin{equation}\label{eq_t_c}
\bb_p(t_p)=\bh_p\left(\sum_{q=1}^n h_{pq}\bc_q-\frac{2t_p^2}{t_p^2+1}\,\bn_p+\frac{2t_p}{t_p^2+1}\,\bm\right).
\end{equation}
If $\X^n$ is either $\bS^n$ or~$\Lambda^n$, then $\bc_1,\ldots,\bc_n$ is the basis of~$\spa(\bn_1,\ldots,\bn_n)$ dual to the basis $\bn_1,\ldots,\bn_n$. 

\section{Biquadratic relations among tangents of half dihedral angles}\label{section_Bri} 

Let us find equations describing the configuration space~$\Xi(\ell)$. The matrices~$G$ and~$H$ are determined by the butterfly~$\B$, that is, by all edge lengths of a cross-polytope except for the lengths of the edges~$[\bb_p\bb_q]$. The condition that the length of~$[\bb_p\bb_q]$ is equal to the prescribed number~$\ell_{\bb_p\bb_q}$ has the form 
\begin{align*}
(\bb_p-\ba_r,\bb_q-\ba_r)&=\frac12(\ell_{\ba_r\bb_p}^2+\ell_{\ba_r\bb_q}^2-\ell_{\bb_p\bb_q}^2),\ \ r\ne p,q,&&\text{if $\X^n=\E^n$,}\\
(\bb_p,\bb_q)&=\cos\ell_{\bb_p\bb_q}&&\text{if $\X^n=\bS^n$,}\\
(\bb_p,\bb_q)&=-\cosh\ell_{\bb_p\bb_q}&&\text{if $\X^n=\Lambda^n$.}
\end{align*}
Using~\eqref{eq_t}, this can be rewritten as
\begin{equation}\label{eq_cij}
-\frac{t_p^2}{t_p^2+1}\,h_{qp}-\frac{t_q^2}{t_q^2+1}\,h_{pq}+\frac{2t_p^2t_q^2}{(t_p^2+1)(t_q^2+1)}\,g_{pq}+\frac{2t_pt_q}{(t_p^2+1)(t_q^2+1)}=E_{pq}\,,
\end{equation}
where
\begin{equation}
E_{pq}=\frac{1}{2\bh_p\bh_q}\cdot\left\{
\begin{aligned}
&\textstyle-\frac12\ell^2_{\bb_p\bb_q}+\frac12\dist^2(\bb_p^0,\bb_q^0)&&\text{if\/ $\X^n=\E^n$,}\\
&\cos\ell_{\bb_p\bb_q}-\cos\dist(\bb_p^0,\bb_q^0)&&\text{if\/ $\X^n=\bS^n$,}\\
&-\cosh\ell_{\bb_p\bb_q}+\cosh\dist(\bb_p^0,\bb_q^0)&&\text{if\/ $\X^n=\Lambda^n$.}
\end{aligned}
\right.\label{eq_E}
\end{equation}
Here we denote by $\dist(\bb_p^0,\bb_q^0)$ the distance between~$\bb_p^0$ and~$\bb_q^0$ in the metric of~$\X^n$.
Finally, we rewrite~\eqref{eq_cij} as
\begin{equation}\label{eq_main_rel}
A_{pq}t_p^2t_q^2+B_{pq}t_p^2-2t_pt_q+D_{pq}t_q^2+E_{pq}=0,
\end{equation}
where
\begin{align}\label{eq_coeff}
A_{pq}&=h_{qp}+h_{pq}-2g_{pq}+E_{pq},&
B_{pq}&=h_{qp}+E_{pq},&
D_{pq}&=h_{pq}+E_{pq}.
\end{align}
We shall say that $\CA=(A_{pq})$, $\CB=(B_{pq})$, $\mathcal{D}=(D_{pq})$, and $\CE=(E_{pq})$ are \textit{matrices with undefined diagonals\/}, since their diagonal entries are not defined. The matrices $\CA$ and~$\CE$ are symmetric, and~$\CB^T=\mathcal{D}$. Further, we usually write $B_{qp}$ instead of~$D_{pq}$. Formulae~\eqref{eq_coeff} yield
\begin{equation}\label{eq_gh_from_coeff}
g_{pq}=\frac12(-A_{pq}+B_{pq}+B_{qp}-E_{pq}),\qquad
h_{pq}=B_{qp}-E_{pq}
\end{equation}
whenever $p\ne q$. (Recall that $g_{pp}=h_{pp}=1$ for all~$p$.)
Thus we obtain the following important proposition.
\begin{propos}
The configuration space $\Xi(\ell)\subset(\RP^1)^n$ is the algebraic variety given by the ${n\choose 2}$ equations~\eqref{eq_main_rel} with coefficients given by~\eqref{eq_coeff}.
\end{propos}

Notice that the matrices~$G$, $H$, and~$\CE$ from which the coefficients of equations~\eqref{eq_main_rel} are computed are determined solely by the set of edge lengths~$\ell$.

Each irreducible component of $\Xi(\ell)$ is either zero-dimensional or one-dimensional, since any pair of coordinates~$t_p$ and~$t_q$ satisfy a non-trivial algebraic relation. One-dimensional irreducible components of~$\Xi(\ell)$ correspond to flexes of cross-polytopes.  We shall say that a one-dimensional irreducible component $\Sigma$ of~$\Xi(\ell)$ is \textit{inessential\/} if  one of the coordinates~$t_p$ is either identically~$0$ or identically~$\infty$ in~$\Sigma$. Inessential components of~$\Xi(\ell)$ correspond to inessential flexes, and we are not interested in them. All other one-dimensional irreducible components of~$\Xi(\ell)$ will be called \textit{essential\/}. Let $\Xi^{ess}(\ell)$ be the union of all essential one-dimensional irreducible components of~$\Xi(\ell)$. Our aim is to classify \textit{irreducible\/} essential flexes of cross-polytopes, i.\,e., pairs $(\ell,\Sigma)$ such that $\Sigma$ is an irreducible component of~$\Xi^{ess}(\ell)$. The author does not know whether there exists a set of edge lengths~$\ell$ such that $\Xi^{ess}(\ell)$ is reducible, i.\,e., such that there exist two different flexible cross-polytopes with the same set of edge lengths~$\ell$.

Relations of the form~\eqref{eq_main_rel}  first appeared in the $3$-dimensional case in the paper by Bricard~\cite{Bri97}, and were the main tool for his classification of flexible octahedra in~$\E^3$. Bricard obtained these relations  using a more geometric approach, and gave explicit formulae for their coefficients from the plane angles of the faces of the octahedron. His result is as follows.

\begin{propos}\label{propos_Bricard}
Let $SABNM$ be a tetrahedral angle in~$\E^3$ with vertex~$S$ and faces $SAB$, $SBN$, $SNM$, and $SMA$. Let $\varphi$ and~$\psi$ be the dihedral angles of this tetrahedral angle at the edges $SA$ and~$SB$ respectively, and let
 $\alpha$, $\beta$, $\gamma$, and~$\delta$ be the plane angles $ASB$, $ASM$, $MSN$, and $NSB$ respectively. Then the tangents $t=\tan(\varphi/2)$ and $t'=\tan(\psi/2)$ satisfy the relation
\begin{equation*}
At^2t'^2+Bt^2+2Ctt'+Dt'^2+E=0,
\end{equation*}
where
\begin{gather*}
\begin{aligned}
A&=\cos\gamma-\cos(\alpha+\beta+\delta),&&&&&
B&=\cos\gamma-\cos(\alpha+\beta-\delta),\\
D&=\cos\gamma-\cos(\alpha-\beta+\delta),&&&&&
E&=\cos\gamma-\cos(\alpha-\beta-\delta),
\end{aligned}\\
C=-2\sin\beta\sin\delta.
\end{gather*}
Besides,
$$
(C^2-AE-BD)^2-4ABDE=16\sin^2\!\alpha\sin^2\!\beta\sin^2\!\gamma\sin^2\!\delta>0.
$$
\end{propos}

Let $\Delta_{pq}$ be the facet of~$P^n$ with vertices $$\bb_p,\bb_q,\ba_1,\ldots,\hat\ba_p,\ldots,\hat\ba_q,\ldots\ba_n,$$ and let $F_{pq}$ be the $(n-3)$-dimensional face of~$P^n$ with vertices  $$\ba_1,\ldots,\hat\ba_p,\ldots,\hat\ba_q,\ldots\ba_n.$$ Then $F_{pq}$ is the common face of~$\Delta$ and~$\Delta_{pq}$.
For each $p\ne q$, we denote by~$\alpha_{pq}$, $\beta_{pq}$, and~$\gamma_{pq}$ the dihedral angles of the $(n-1)$-simplices~$\Delta$, $\Delta_p$, and~$\Delta_{pq}$ respectively at their common codimension~2 face~$F_{pq}$. Obviously, $\alpha_{qp}=\alpha_{pq}$ and $\gamma_{qp}=\gamma_{pq}$.

Let $\bx$ be an arbitrary point in the interior of~$F_{pq}$. Let $L$ be the three-dimensional subspace of~$T_{\bx}\X^n$ orthogonal to~$T_{\bx}F_{pq}$. The space~$L$ contains the tetrahedral angle formed by the two-dimensional tangent cones  to the facets $\Delta$, $\Delta_p$, $\Delta_{pq}$, and~$\Delta_q$ respectively. The plane angles of this tetrahedral angle are equal to~$\alpha_{pq}$, $\beta_{pq}$, $\gamma_{pq}$, and $\beta_{qp}$ respectively, and the dihedral angles of this tetrahedral angle at the two edges adjacent to the plane angle~$\alpha_{pq}$ are equal to~$\varphi_p$ and~$\varphi_q$ respectively. Hence Proposition~\ref{propos_Bricard} implies that the tangents $t_p=\tan(\varphi_p/2)$ satisfy the relations
\begin{equation}\label{eq_main_Bri_rel}
A'_{pq}t_p^2t_q^2+B'_{pq}t_p^2+2C'_{pq}t_pt_q+B'_{qp}t_q^2+E'_{pq}=0,
\end{equation}
where 
\begin{equation}\label{eq_A'B'C'E'}
\begin{aligned}
A'_{pq}&=\cos\gamma_{pq}-\cos(\alpha_{pq}+\beta_{pq}+\beta_{qp}),&&&
B'_{pq}&=\cos\gamma_{pq}-\cos(\alpha_{pq}+\beta_{pq}-\beta_{qp}),\\
C'_{pq}&=-2\sin\beta_{pq}\sin\beta_{qp},&&&
E'_{pq}&=\cos\gamma_{pq}-\cos(\alpha_{pq}-\beta_{pq}-\beta_{qp}).
\end{aligned}
\end{equation}

\begin{propos}
For any $p\ne q$, relation~\eqref{eq_main_Bri_rel} is proportional to relation~\eqref{eq_main_rel}.
\end{propos}
\begin{proof}
To prove this proposition, we need to compute the coefficients $A_{pq}$, $B_{pq}$, $B_{qp}$, and~$E_{pq}$ from the angles $\alpha_{pq}$, $\beta_{pq}$, $\beta_{qp}$, and~$\gamma_{pq}$. Let $\sigma_{pq}$ be the common $(n-2)$-dimensional face of the simplices~$\Delta_p$ and~$\Delta_{pq}$, and let $\sigma_{qp}$ be the common $(n-2)$-dimensional face of the simplices~$\Delta_q$ and~$\Delta_{pq}$. Let $r_1$ and $r_2$ be the lengths of the altitudes of the simplices~$\sigma_{pq}$ and~$\sigma_{qp}$ respectively drawn from the vertices~$\bb_p$ and~$\bb_q$ respectively to the common face~$F_{pq}$ of~$\sigma_{pq}$ and~$\sigma_{qp}$, and let $\rho$ be the distance between the bases of these altitudes. Since the dihedral angle between the simplices~$\sigma_{pq}$ and~$\sigma_{qp}$ is equal to~$\gamma_{pq}$, it is not hard to check that 
\begin{align*}
\ell^2_{\bb_p\bb_q}&=\rho^2+r_1^2+r_2^2-2r_1r_2\cos\gamma_{pq}&&\text{if\/ $\X^n=\E^n$,}\\
\cos\ell_{\bb_p\bb_q}&=\cos\rho\cos r_1\cos r_2+\sin r_1\sin r_2\cos\gamma_{pq}&&\text{if\/ $\X^n=\bS^n$,}\\
\cosh\ell_{\bb_p\bb_q}&=\cosh\rho\cosh r_1\cosh r_2-\sinh r_1\sinh r_2\cos\gamma_{pq}&&\text{if\/ $\X^n=\Lambda^n$.}
\end{align*}
If we rotate simplices~$\Delta_p$ and~$\Delta_q$ around their facets~$F_p$ and~$F_q$ respectively so that their vertices opposite to these facets coincide with~$\bb^0_p$ and~$\bb^0_q$ respectively, then the angle between their faces~$\sigma_{pq}$ and~$\sigma_{qp}$ becomes equal to~$|\alpha_{pq}-\beta_{pq}-\beta_{qp}|$. Hence the above formulae remain correct if we replace everywhere $\bb_p$ with $\bb_p^0$, $\bb_q$ with $\bb_q^0$, and $\cos\gamma_{pq}$ with $\cos(\alpha_{pq}-\beta_{pq}-\beta_{qp})$. Besides, notice that $r_1=\bh_p\sin^{-1}\beta_{pq}$ if $\X^n=\E^n$, $\sin r_1=\bh_p\sin^{-1}\beta_{pq}$ if $\X^n=\bS^n$,  $\sinh r_1=\bh_p\sin^{-1}\beta_{pq}$ if $\X^n=\Lambda^n$, and similarly for~$r_2$. Using~\eqref{eq_E}, we finally obtain that in all three cases 
$$
E_{pq}=\frac{\cos\gamma_{pq}-\cos(\alpha_{pq}-\beta_{pq}-\beta_{qp})}{2\sin\beta_{pq}\sin\beta_{qp}}=\frac{E'_{pq}}{2\sin\beta_{pq}\sin\beta_{qp}}\,.
$$
Now, the definitions of~$G$ and~$H$ yield that $g_{pq}=-\cos\alpha_{pq}$ and $h_{pq}=\frac{\sin(\alpha_{pq}-\beta_{pq})}{\sin\beta_{pq}}$ for all $p\ne q$. Combining this with~\eqref{eq_coeff}, we obtain that
$$
A_{pq}=\frac{A'_{pq}}{2\sin\beta_{pq}\sin\beta_{qp}}\,,\qquad
B_{pq}=\frac{B'_{pq}}{2\sin\beta_{pq}\sin\beta_{qp}}\,. 
$$
Hence relation~\eqref{eq_main_Bri_rel} is proportional to~\eqref{eq_main_rel} with the coefficient $2\sin\beta_{pq}\sin\beta_{qp}$.
\end{proof}

\begin{cor}\label{cor_ineq_main}
For any $p\ne q$, we have 
\begin{equation}\label{eq_ineq_main}
(1-A_{pq}E_{pq}-B_{pq}B_{qp})^2-4A_{pq}B_{pq}B_{qp}E_{pq}>0.
\end{equation}
\end{cor}

\section{Recovering a flexible cross-polytope from the curve~$\Sigma$}\label{section_EPBQ}

Let $(\ell,\Sigma)$ be an irreducible essential flex of a cross-polytope, i.\,e., a pair such that $\Sigma$ is an irreducible component of~$\Xi^{ess}(\ell)$. Let us study the problem of recovering the set of edge lengths~$\ell$ from the given curve~$\Sigma$. This problem consists of two parts. First, we need to recover coefficients~$A_{pq}$, $B_{pq}$, and~$E_{pq}$ (or, equivalently, matrices~$G$, $H$, and~$\CE$) from~$\Sigma$. Second, we need to recover the set of edge lengths~$\ell$ from the matrices~$G$, $H$, and~$\CE$. We start from the second problem. For a given butterfly, edge lengths~$\ell_{\bb_p\bb_q}$ are recovered uniquely from~$\CE$ by~\eqref{eq_E}. Hence we need to recover a butterfly from~$G$ and~$H$.  

\subsection{Recovering a butterfly from the matrices~$G$ and~$H$} 
Let us describe all pairs of matrices $(G,H)$ that correspond to butterflies in~$\E^n$, $\bS^n$, and~$\Lambda^n$. The main difficulty in this description consists in the fact that we need to take care of the positivity of~$\ah _p$ and~$\bh_p$. The easiest way to overcome this difficulty is to turn down the positivity of these parameters. 

Indeed, $\ah _p$ and~$\bh_p$ are the lengths of altitudes of certain simplices (or their sines, or their hyperbolic sines). Nevertheless, from the algebraic viewpoint, we can change the sign of the length of an altitude of an $(n-1)$-simplex, and simultaneously replace the inner normal vector to the opposite facet of this simplex by the outer normal vector, and the normal vector to the plane spanned by this simplex by the opposite vector. Then all algebraic formulae remain correct. This leads us to two types of algebraic transformations corresponding to changing sings of the altitudes~$\ah _p$ and~$\bh_p$ respectively. These transformations, which will be called \textit{elementary reversions\/}, do not change the butterfly~$\B$, and all above formulae are invariant under them. The action of \textit{$\ah_p$-reversions\/} and \textit{$\bh_p$-reversions\/}  on the objects introduced above is given by the following tables. In the second table $q\ne p$. Certainly, both the \textit{$\ah_p$-reversion\/} and the \textit{$\bh_p$-reversion\/} do not change $\ah _q$, $\bh_q$,  $\bn_q$, $\bb^0_q$, $\varphi_q$, $t_q$ for $q\ne p$, and the elements of the matrices~$\CA$, $\CB$, $\CE$, $G$, and~$H$ that belong neither to the $p$th row nor to the $p$th column.  
\begin{center}
\smallskip
\begin{tabular}{|c|c|c|c|c|c|c|c|}
\hline
$\vphantom{\Bigl(}$&$a_p$&$b_p$&$\bm$&$\bn_p$&$\bb^0_p$&$\varphi_p$&$t_p$\\
\hline
$\vphantom{\Bigl(}a_p$-reversion&$-a_p$&$b_p$&$-\bm$&$-\bn_p$&$\bb^0_p-2b_p\bn_p$&$\pi-\varphi_p$&$t_p^{-1}$\\
$\vphantom{\Bigl(}b_p$-reversion&$a_p$&$-b_p$&$\bm$&$\bn_p$&$\bb^0_p-2b_p\bn_p$&$\pi+\varphi_p$&$-t_p^{-1}$\\ \hline
\end{tabular}

\medskip

\begin{tabular}{|c|c|c|c|c|c|c|c|}
\hline
$\vphantom{\Bigl(}$&$A_{pq}$&$B_{pq}$&$B_{qp}$&$E_{pq}$&$g_{pq}$&$h_{pq}$&$h_{qp}$\\
\hline
$\vphantom{\Bigl(}a_p$-reversion&$B_{qp}$&$E_{pq}$&$A_{pq}$&$B_{pq}$&$-g_{pq}$&$h_{pq}-2g_{pq}$&$-h_{qp}$\\
$\vphantom{\Bigl(}b_p$-reversion&$-B_{qp}$&$-E_{pq}$&$-A_{pq}$&$-B_{pq}$&$g_{pq}$&$2g_{pq}-h_{pq}$&$h_{qp}$\\ \hline
\end{tabular}
\smallskip
\end{center}

In the sequel, we assume that all introduced objects  corresponding to a butterfly~$\B$, in particular, the matrices~$G$ and~$H$ are defined up to elementary reversions. It is easy to see that elementary reversions commute to each other. Hence we may agree that any of the $2^{2n}$ pairs of matrices~$(G,H)$ that can be obtained by elementary reversions from the pair~$(G,H)$ introduced in section~\ref{section_but} corresponds to the butterfly~$\B$. 
This convention extends the class of pairs of matrices~$(G,H)$ that can correspond to butterflies, and allows us to describe this class effectively.

Let~$W_n$ be the space of all pairs~$(G,H)$ of real matrices of sizes $n\times n$ such that $G$ is symmetric and $g_{pp}=h_{pp}=1$ for all~$p$. Then~$W_n$ is a vector space of dimension~$\frac{3}{2}n(n-1)$. We denote by~$\Psi(\E^n)$, $\Psi(\bS^n)$, and $\Psi(\Lambda^n)$ the subsets of~$W_n$ consisting of all pairs $(G,H)$ that correspond to butterflies in~$\E^n$, $\bS^n$ and~$\Lambda^n$ respectively. These subsets are invariant under elementary reversions.

\begin{propos}\label{propos_but_S}
A pair of matrices $(G,H)\in W_n$ belongs to~$\Psi(\bS^n)$ if and only if 
 $G$ is positive definite. A butterfly in~$\bS^n$ is determined by a pair~$(G,H)\in\Psi(\bS^n)$ uniquely up to isometry, and up to replacing some of its vertices with their antipodes.
\end{propos}  

\begin{propos}\label{propos_but_E}
A pair of matrices $(G,H)\in W_n$ belongs to~$\Psi(\E^n)$ if and only if the following two conditions hold:
\begin{itemize}
\item[($\E1$)] $G$ is degenerate positive semidefinite, and all principal minors of~$G$ of sizes $2\times2,$ $3\times 3, \ldots,$ $(n-1)\times(n-1)$ are strictly positive. 
\item[($\E2$)] no row of~$H$ is a linear combination of rows of~$G$.
\end{itemize}
A butterfly in~$\E^n$ is determined by a pair  $(G,H)\in\Psi(\E^n)$ uniquely up to similarity.  
\end{propos}

\begin{propos}\label{propos_but_L}
A pair of matrices $(G,H)\in W_n$ belongs to~$\Psi(\Lambda^n)$ if and only if the following two conditions hold:
\begin{itemize}
\item[($\Lambda1$)] $G$ is non-degenerate indefinite with negative index of inertia~$1$, and all principal minors of~$G$ of sizes $2\times2,$ $3\times 3, \ldots,$ $(n-1)\times(n-1)$ are strictly positive. \item[($\Lambda2$)] $\sum_{q=1}^n\sum_{r=1}^ng^{qr}h_{pq}h_{pr}<0$ for $p=1,\ldots,n$.
\end{itemize}
A butterfly in~$\Lambda^n$ is determined by a pair  $(G,H)\in\Psi(\Lambda^n)$ uniquely up to isometry.  
\end{propos}

\begin{proof}[Proofs of Propositions~\ref{propos_but_S}, \ref{propos_but_E}, and~\ref{propos_but_L}]
For $\X^n=\bS^n$, $\bn_1,\ldots,\bn_n$ must be linearly independent vectors of~$\E^{n+1}$. For $\X^n=\Lambda^n$, $\bn_1,\ldots,\bn_n$ must be linearly independent space-like vectors of $\langle\bm\rangle^{\bot}\cong\E^{n-1,1}$ such that for any $n-1$ of them the normal vector to their span is time-like. For $\X^n=\E^n$, $\bn_1,\ldots,\bn_n$ must be linearly dependent vectors of~$\E^{n}$ such that any $n-1$ of them are linearly independent. A symmetric matrix~$G$ can be the Gram matrix of vectors $\bn_1,\ldots,\bn_n$ satisfying these conditions if and only if all proper principal minors of~$G$ are positive, and $\det G>0$, $\det G<0$, and $\det G=0$ in the cases $\X^n=\bS^n$, $\X^n=\Lambda^n$, and~$\X^n=\E^n$ respectively. If $G$ satisfies these properties, then  $\bn_1,\ldots,\bn_n$ can be recovered from~$G$ uniquely up to (pseudo-)orthogonal transformations of~$\E^{n+1}$, $\E^{n,1}$, and~$\E^n$ respectively.

If $\X^n=\E^n$, then the numbers $\ah_1,\ldots,\ah_n$  are determined uniquely up to proportionality by formula~\eqref{eq_n_sum}. Moreover, \eqref{eq_an_E} and~\eqref{eq_h} imply that 
\begin{equation}\label{eq_beta_E}
\bh_p=\left(\sum_{q=1}^n \ah _q^{-1}h_{pq}\right)^{-1}.
\end{equation}
Hence $\bh_p$ are well defined if and only if $\sum_{q=1}^n \ah _q^{-1}h_{pq}\ne 0$ for all~$p$. Let us show that this condition is equivalent to~$(\E2)$.
Indeed, formula~\eqref{eq_n_sum} implies that $\sum_{q=1}^n \ah _q^{-1}g_{pq}=0$ for $p=1,\ldots,n$. Since the rank of~$G$ is $n-1$, it follows that the $p$th row of~$H$ is a linear combination of the rows of~$G$ if and only if $\sum_{q=1}^n \ah _q^{-1}h_{pq}=0$. 

If either $\X^n=\bS^n$ or $\X^n=\Lambda^n$, then the numbers $\ah_p$ and~$\bh_p$ are determined by the conditions $\ba_p\in\X^n$ and $\bb^0_p\in\X^n$ respectively. Let us introduce the parameter~$\varepsilon$ that is equal to~$1$ for $\bS^n$, and to~$-1$ for~$\Lambda^n$. Then $(\ba_p,\ba_p)=(\bb_p^0,\bb_p^0)=\varepsilon$ for all $p$. Using~\eqref{eq_an_SL} and~\eqref{eq_bb0}, we easily rewrite this as 
\begin{align}
\label{eq_alpha_SL}\ah _p&=\pm(\varepsilon g^{pp})^{-\frac12},
\\
\label{eq_beta_SL}\bh_p&=\pm\left(\varepsilon  \sum_{q=1}^n\sum_{r=1}^n g^{qr}h_{pq}h_{pr}\right)^{-\frac12}.
\end{align}
The expressions under the square roots in these formulae are always positive if $\X^n=\bS^n$. If $\X^n=\Lambda^n$, then we always have $-g^{pp}>0$, and the positivity of the expression under the square root in~\eqref{eq_beta_SL} is exactly condition~$(\Lambda2)$.

The vectors $\bc_1,\ldots,\bc_n$ are determined by the vectors $\bn_1,\ldots,\bn_n$ uniquely in the cases~$\bS^n$ and~$\Lambda^n$. Now, the formulae $\ba_p=a_p\bc_p$ and formulae~\eqref{eq_t_c} determine the butterfly~$\B$. If $\X^n=\bS^n$, then the signs~$\pm$ in~\eqref{eq_alpha_SL}, \eqref{eq_beta_SL} may be chosen arbitrarily. Hence the butterfly~$\B$ is determined up to isometry, and up to replacing some of its vertices by their antipodes. 
If $\X^n=\Lambda^n$, then the signs~$\pm$ in~\eqref{eq_alpha_SL}, \eqref{eq_beta_SL} are uniquely determined by the requirement that the points~$\ba_p$ and~$\bb^0_p$ belong to the proper sheet~$\Lambda^n$ of the hyperboloid $(\bx,\bx)=-1$. Hence the butterfly~$\B$ is determined up to isometry.

If $\X^n=\E^n$, then the numbers $\ah_1,\ldots,\ah_n$, $\bh_1,\ldots,\bh_n$ are determined up to proportionality. For given $\ah_1,\ldots,\ah_n$, formulae~\eqref{eq_an_E} determine the points $\ba_1,\ldots,\ba_n$ uniquely up to parallel translation, and then formulae~\eqref{eq_t_c} determine the butterfly~$\B$. Hence the butterfly~$\B$ is determined up to similarity. 

\end{proof}

\begin{cor}
$\Psi(\bS^n)$ and~$\Psi(\Lambda^n)$ are open semi-algebraic subsets of\/~$W_n$,
$\Psi(\E^n)$ is an open semi-algebraic subset of the hypersurface in~$W_n$ given by $\det G=0$.
\end{cor}

\begin{cor}\label{cor_but}
The set\/ $\Psi(n)=\Psi(\bS^n)\sqcup\Psi(\E^n)\sqcup\Psi(\Lambda^n)$ is the open semi-algebraic subset of\/~$W_n$ given by the following inequalities:
\begin{enumerate}
\item All principal minors of the matrix~$G$ of sizes $2\times2,$ $3\times 3, \ldots,$ $(n-1)\times(n-1)$ are positive.
\item $\sum_{q=1}^n\sum_{r=1}^nG_{qr}h_{pq}h_{pr}>0$ for $p=1,\ldots,n$, where $G_{qr}$ is the $(q,r)$ cofactor of~$G$.
\end{enumerate} 
\end{cor}

\subsection{Recovering the coefficients from the curve~$\Sigma$}
\begin{lem}\label{lem_nonproportional}
Suppose that coordinates $t_p$ and~$t_q$ restricted to~$\Sigma$ are neither directly nor inversely proportional to each other. Then the $4$ coefficients $A_{pq}$, $B_{pq}$, $B_{qp}$, $E_{pq}$ such that relation~\eqref{eq_main_rel} holds in~$\Sigma$ are determined uniquely.  
\end{lem}

\begin{proof}
Assume that a relation
\begin{equation}\label{eq_rel_unique}
R(t_p,t_q)=At_p^2t_q^2+Bt_p^2-2t_pt_q+Dt_q^2+E=0
\end{equation} 
holds in~$\Sigma$. Since $\Sigma$ is irreducible, there is a minimal relation $R^{\min}(t_p,t_q)=0$ satisfied in~$\Sigma$, which divides all polynomial relations between~$t_p$ and~$t_q$ in~$\Sigma$. Suppose that $R=QR^{\min}$. Since $\Sigma$ is essential, neither of the coordinates~$t_p$  and~$t_q$ is identically zero in~$\Sigma$, and neither of the coordinates~$t_p$  and~$t_q$ is identically infinity in~$\Sigma$. Then the existence of a relation of the form~\eqref{eq_rel_unique} easily implies, that neither of the coordinates~$t_p$ and~$t_q$ is constant in~$\Sigma$. Hence both $t_p$ and~$t_q$ enter non-trivially the polynomial~$R^{\min}(t_p,t_q)$. It is easy to check that a polynomial~$P$ that has degree $1$ with respect to either of the variables~$t_p$ and~$t_q$ can divide the polynomial~$R$ only if $P$ has either the form $c_1t_p+c_2t_q$ or the form $c_1t_pt_q+c_2$. The polynomial~$R^{\min}$ has neither of these forms, since $t_p$ and~$t_q$ are neither directly nor inversely proportional to each other in~$\Sigma$. Hence the degree of  one of the variables, say~$t_p$, in~$R^{\min}$ is equal to~$2$. Therefore $t_p$ does not enter~$Q$. Besides, since $R$ is an even polynomial, we obtain that either both~$R^{\min}$ and~$Q$ are even, or both~$R^{\min}$ and~$Q$ are odd. If $R^{\min}$ is even, then $Q$ must be a constant. If $R^{\min}$ is odd, then $Q$ must have the form $ct_q$. In both cases, the polynomial $R$ is determined uniquely up to proportionality. But the coefficient of~$t_pt_q$ in $R$ is equal to~$-2$. Hence $R$ is determined uniquely.
\end{proof}

\begin{cor}
If $t_p$ and~$t_q$ are neither directly nor inversely proportional to each other in~$\Sigma$, then the entries $g_{pq}$, $h_{pq}$, $h_{qp}$, and $E_{pq}$ of matrices $G$, $H$, and~$\CE$ such that $\Sigma\subset\Xi(G,H,\CE)$ are determined uniquely. 
\end{cor}

Now, suppose that $t_p$ and~$t_q$ are either directly or inversely proportional to each other in~$\Sigma$. The second case can be reduced to the first case by an elementary reversion, hence, we assume that $t_p$ and~$t_q$ are directly proportional to each other. The proof of the following lemma is straightforward.

\begin{lem}\label{lem_proportional}
Suppose that $t_q=\lambda t_p$ in $\Sigma$ for a non-zero constant~$\lambda$. If $\lambda\ne\pm 1$, then relation~\eqref{eq_main_rel} with coefficients given by~\eqref{eq_coeff} is satisfied in~$\Sigma$ if and only if $E_{pq}=0$ and
\begin{equation*}
h_{pq}=\frac{2(g_{pq}-\lambda)}{1-\lambda^2}\,,\qquad h_{qp}=\frac{2(g_{pq}-\lambda^{-1})}{1-\lambda^{-2}}\,.
\end{equation*}
In the exceptional cases $\lambda=\pm 1$, relation~\eqref{eq_main_rel} with coefficients given by~\eqref{eq_coeff} is satisfied in~$\Sigma$ if and only if $E_{pq}=0$, $g_{pq}=\lambda$, and $h_{pq}+h_{qp}=2\lambda$.
\end{lem} 

\begin{remark}\label{rem_except}
The exceptional cases $\lambda=\pm 1$ are not interesting to us, since a pair~$(G,H)$ with a non-diagonal entry $g_{pq}=\pm 1$ never belongs to any of the sets~$\Psi(\X^n)$ for $n\ge 3$. (If $n=2$, then this case is possible and corresponds to the flexion of a parallelogram in~$\E^2$.) 
\end{remark}

We see that the cases of  proportional and non-proportional pairs $(t_p,t_q)$ differ drastically.   
Let $[n]$ be the set $\{1,\ldots,n\}$. Let $\I$ be the decomposition of~$[n]$ into pairwise disjoint subsets $I_1,\ldots,I_m$ such that $t_p$ and $t_q$ are either directly or inversely proportional in~$\Sigma$ if and only if $p$ and~$q$ belong to the same subset~$I_j$. Then we say that the curve~$\Sigma$ and any  irreducible flex of a cross-polytope $(\ell,\Sigma)$ corresponding to it are \textit{subject\/} to the decomposition~$\I$. Let $n_j$ be the cardinality of~$I_j$, $j=1,\ldots,m$. Then the partition $(n_1,\ldots,n_m)$ of~$n$ is called the \textit{type\/} of the curve~$\Sigma$ and of any  irreducible flex of a cross-polytope $(\ell,\Sigma)$ corresponding to it. Recall that the word \textit{partition\/} means that we do  not differ between the sequences  $(n_1,\ldots,n_m)$ that can be obtained from each other by permutations. 

\subsection{EPBQ-curves}  Consider the space $(\RP^1)^m$ with the standard coordinates $z_1,\ldots,z_m$.  

\begin{defin}
An irreducible algebraic curve $\Gamma\subset (\RP^1)^m$ will be called an \textit{even poly-biquadratic curve\/} (or, simply an \textit{EPBQ-curve\/}) if  
\begin{enumerate}
\item each pair of coordinates $z_j$ and~$z_l$, $j\ne l$, satisfies in~$\Gamma$ a relation of the form
\begin{equation}\label{eq_main_rel_abe}
a_{jl}z_j^2z_l^2+b_{jl}z_j^2-2z_jz_l+b_{lj}z_l^2+e_{jl}=0,
\end{equation}
\item each coordinate~$z_j$ is neither identically~$0$ nor identically~$\infty$ in~$\Gamma$,
\item each pair of coordinates $z_j$ and~$z_l$, $j\ne l$, is neither directly nor inversely proportional to each other in~$\Gamma$.
\end{enumerate}
Two EPBQ-curves $\Gamma_1,\Gamma_2\subset (\RP^1)^m$ are said to be \textit{equivalent\/} if and only if the first can be taken to the second by the composition of a transformation of the form $$(z_1,\ldots,z_m)\mapsto(\nu_1 z_1^{\pm 1},\ldots,\nu_m z_m^{\pm 1}),\qquad\nu_j\ne 0, \ j=1,\ldots,m,$$ and a permutation of the factors~$\RP^1$.
\end{defin}

To each irreducible flex of a cross-polytope~$(\ell,\Sigma)$, we assign a EPBQ-curve in the following way. Suppose that $\Sigma$ is subject to the decomposition $[n]=I_1\sqcup\cdots\sqcup I_m$.
For each $p\in[n]$, we denote by~$j(p)$ the index such that $p\in I_{j(p)}$. Let $p_1,\ldots,p_m$ be any representatives of the subsets $I_1,\ldots,I_m$ respectively. Let $\Gamma$ be the coordinate projection of~$\Sigma\subset(\RP^1)^n$ to the product of $m$ factors~$\RP^1$ with indices $p_1,\ldots,p_m$. 
Then $\Gamma$ is an EPBQ-curve. Besides, up to equivalence $\Gamma$ is independent of the choice of $p_1,\ldots,p_m$, and does not change under elementary reversions.

The problem of classification of flexible cross-polytopes now splits into two parts. First, we need to classify all EPBQ-curves $\Gamma\subset(\RP^1)^m$ up to equivalence. Second, for each EPBQ-curve  $\Gamma\subset(\RP^1)^m$ and each decomposition $\I\colon[n]=I_1\sqcup\cdots\sqcup I_m$, we need to describe all irreducible flexes~$(\ell,\Sigma)$ such that $\Sigma$ is subject to~$\I$, and the EPBQ-curve corresponding to~$\Sigma$ is equivalent to~$\Gamma$. The latter problem was actually solved in the previous section. Indeed, the results of the previous section immediately imply the following description of \textit{all\/} flexible cross-polytopes corresponding to a pair~$(\Gamma,\I)$:

\begin{constr}\label{constr}
(1) Let $z_1(u),\ldots,z_m(u)$ be a parametrization of~$\Gamma$. By Lemma~\ref{lem_nonproportional}, the coefficients $a_{jl}$, $b_{jl}$, $e_{jl}$ are determined uniquely by~$\Gamma$. Obviously, $a_{jl}=a_{lj}$ and $e_{jl}=e_{lj}$.

(2) Let $\Sigma\subset(\RP^1)^n$ be the curve parametrized by
\begin{equation}\label{eq_t_z}
t_p(u)=\lambda_pz_{j(p)}(u),\qquad p=1,\ldots,n,
\end{equation}
where $\blambda=(\lambda_1,\ldots,\lambda_n)$ is a set of non-zero real coefficients such that $\lambda_p\ne \pm\lambda_q$ whenever $j(p)=j(q)$ and $p\ne q$.

(3) If $j(p)\ne j(q)$, then we put
\begin{equation}\label{eq_ABE_abe}
A_{pq}=\frac{a_{j(p)j(q)}}{\lambda_p\lambda_q}\,,\qquad B_{pq}=\frac{\lambda_qb_{j(p)j(q)}}{\lambda_p}\,,\qquad E_{pq}=\lambda_p\lambda_qe_{j(p)j(q)}\,.
\end{equation}
Then the corresponding entries of the matrices $G$ and~$H$ are determined by
\begin{align}\label{eq_g_ne}
g_{pq}&=\frac12\left(-\frac{a_{j(p)j(q)}}{\lambda_p\lambda_q}
+\frac{\lambda_qb_{j(p)j(q)}}{\lambda_p}
+\frac{\lambda_pb_{j(q)j(p)}}{\lambda_q}
-\lambda_p\lambda_qe_{j(p)j(q)}\right),\\
\label{eq_h_ne}
h_{pq}&=\frac{\lambda_pb_{j(q)j(p)}}{\lambda_q}
-\lambda_p\lambda_qe_{j(p)j(q)}\,.
\end{align}

(4) If $j(p)=j(q)$ and $p\ne q$, then the entry $g_{pq}=g_{qp}$  can be chosen arbitrarily, $E_{pq}=0$,  and 
\begin{equation}\label{eq_h_e}
h_{pq}=\frac{2\lambda_p(\lambda_pg_{pq}-\lambda_q)}{\lambda_p^2-\lambda_q^2}\,.
\end{equation}
Let~$\mathbf{g}$ be the set of $\sum_{j=1}^m{n_j\choose 2}$ numbers~$g_{pq}=g_{qp}$ such that $j(p)=j(q)$ and $p\ne q$. As usually, we put $g_{pp}=h_{pp}=1$ for all~$p$.

\begin{assump}\label{assump}
The pair $(G,H)$ belongs to $\Psi(\X^n)$, where $\X^n$ is $\E^n$, $\bS^n$, or $\Lambda^n$.  
\end{assump}

If this assumption holds, we may proceed with the construction of a flexible cross-polytope:

(5) Let $\bn_1,\ldots,\bn_n$ be vectors in~$\V$ with the Gram matrix~$G$, and
let $\bm\in \V$ be a unit vector orthogonal to $\bn_1,\ldots,\bn_n$, where $\V$ is $\E^n$, $\E^{n+1}$, and $\E^{n,1}$ in the cases $\X^n=\E^n$, $\bS^n$, and~$\Lambda^n$ respectively.

(6) If $\X^n=\E^n$, then let $\ah_1,\ldots,\ah_n$ be non-zero real numbers such that the linear combination of the rows of~$G$ with coefficients $\ah_1^{-1},\ldots,\ah_n^{-1}$ vanishes, and let $\bh_p$ be given by~\eqref{eq_beta_E}. Now, choose points $\ba_1,\ldots,\ba_n$ in~$\E^n$ such that equalities~\eqref{eq_an_E} hold, and put $\bc_p=a_p^{-1}\ba_p$.

If $\X^n$ is either $\bS^n$ or $\Lambda^n$, then let $\bc_1,\ldots,\bc_n$ be the basis of $\spa(\bn_1,\ldots,\bn_n)$ dual to the basis $\bn_1,\ldots,\bn_n$, and let the numbers~$\ah_p$ and $\bh_p$ be given by~\eqref{eq_alpha_SL},~\eqref{eq_beta_SL}. (The signs are chosen arbitrarily if $\X^n=\bS^n$, and are chosen so that the vectors $\ba_p=a_p\bc_p$ and the vectors $\bb_p^0$ given by~\eqref{eq_bb0} belong to~$\Lambda^n$ if $\X^n=\Lambda^n$.)

(7) Now, we obtain the flexible cross-polytope parametrized by the formulae
\begin{align*}%\label{eq_t}
\ba_p(u)&=\ah_p\bc_p,\\
\bb_p(u)&=\bh_p\left(\sum_{q=1}^nh_{pq}\bc_q-\frac{2t_p^2(u)}{t_p^2(u)+1}\,\bn_p+\frac{2t_p(u)}{t_p^2(u)+1}\,\bm\right).
\end{align*}
Using~\eqref{eq_t_z}, the latter formula can be rewritten in the form
\begin{equation*}
\bb_p(u)=\bh_p\left(\sum_{q=1}^nh_{pq}\bc_q-\frac{2\lambda_p^2z_{j(p)}^2(u)}{\lambda_p^2z_{j(p)}^2(u)+1}\,\bn_p+\frac{2\lambda_pz_{j(p)}(u)}{\lambda_p^2z_{j(p)}^2(u)+1}\,\bm\right).
\end{equation*}
\end{constr}

\begin{remark}  It follows from the construction of the EPBQ-curve corresponding to a flexible cross-polytope that each function~$t_p(u)$ should be either directly or inversely proportional to $z_{j(p)}(u)$. However, by elementary reversions we can achieve that  every~$t_p(u)$ is \textit{directly\/}  proportional to $z_{j(p)}(u)$, which is asserted in Step~(2) of the construction.  
\end{remark}

\begin{remark} For $j(p)\ne j(q)$, relation~\eqref{eq_main_rel} with the coefficients given by~\eqref{eq_ABE_abe} follows immediately from relation~\eqref{eq_main_rel_abe}. However, by Lemma~\ref{lem_nonproportional}, the relation of such form is unique. Hence we \textit{must} determine~$A_{pq}$, $B_{pq}$, and~$E_{pq}$ by~\eqref{eq_ABE_abe}. 
\end{remark}

\begin{remark}
We need to explain why the construction described always yields an \textit{essential\/} flexible cross-polytope. Assumption~\ref{assump} implies immediately that the facet $\Delta=[\ba_1\ldots\ba_n]$ and the facets $\Delta_1,\ldots,\Delta_n$ adjacent to it of the obtained cross-polytope are non-degenerate, and, for each~$p$, the dihedral angle between~$\Delta$ and~$\Delta_p$ is not constant under the flexion. Nevertheless, we need to check that all other facets are also non-degenerate, and all other dihedral angles are non-constant. Let us prove this. Since the facets~$\Delta$ and~$\Delta_p$ are non-degenerate and the dihedral angle between them is non-constant, we obtain that the length of the diagonal~$[\ba_p\bb_p]$ is also non-constant under the flexion, $p=1,\ldots,n$. If some facet of the cross-polytope was degenerate, there would exist two facets with a common $(n-2)$-dimesional face one of which would be degenerate and the other would not. Then the length of the diagonal between the vertices of these facets opposite to their common face would be constant under the flexion, which is not true. Hence all facets are non-degenerate. Similarly, if a dihedral angle between some two facets was constant under the flexion, then the length of the corresponding diagonal would also be constant, which again yields a contradiction. Therefore, all flexible cross-polytopes given by Construction~\ref{constr} are essential.
\end{remark}

Thus, we obtain the following proposition.

\begin{propos}\label{propos_constr}
Suppose, $\X^n=\E^n$, $\bS^n$, or~$\Lambda^n$, $n\ge 3$. Then for each $4$-tuple $(\Gamma,\I,\blambda,\mathbf{g})$ satisfying Assumption~\ref{assump}, Construction~\ref{constr} yields an essential flexible cross-polytope in~$\X^n$ (more precisely, an irreducible essential flex of a cross-polytope). A $4$-tuple $(\Gamma,\I,\blambda,\mathbf{g})$ determines a flexible cross-polytope uniquely up to similarity if\/ $\X^n=\E^n$, uniquely up to isometry if\/ $\X^n=\Lambda^n$, and uniquely up to isometry and up to replacing some of its vertices with their antipodes if\/ $\X^n=\bS^n$. Vice versa, any essential flexible cross-polytope in each of the spaces~$\X^n$ can be obtained by Construction~\ref{constr}.
\end{propos}

Thus, to obtain a classification of flexible cross-polytopes in~$\X^n$, we need: 
\begin{enumerate}
\item To classify all EPBQ-curves $\Gamma\subset(\RP^1)^m$ up to equivalence.
\item For each~$\I$, to describe all triples $(\Gamma,\blambda,\mathbf{g})$ such that the corresponding pairs of matrices $(G,H)$ belong to~$\Psi(\X^n)$.
\end{enumerate}

A classification of EPBQ-curves is a purely algebraic problem. Hence it is easier to solve it over~$\C$ than over~$\R$. The definition of a complex EPBQ-curve $\Gamma\subset(\CP^1)^m$ repeats literally the definition of a real EPBQ-curve.  In Section~\ref{section_ap_C} we classify complex  EPBQ-curves, and  in Section~\ref{section_ap_R} we use this result to classify real EPBQ-curves.

Since the classification of EPBQ-curves itself is not our goal, during the classification we shall throw away certain families of EPBQ-curves that for sure cannot correspond to flexible cross-polytopes. We say that an EPBQ-curve~$\Gamma$ is \textit{realisable} if it corresponds to at least one flexible cross-polytope in at least one of the spaces~$\E^n$, $\bS^n$, and~$\Lambda^n$,  $n\ge 3$. In other words, $\Gamma$ is realisable if there exists a number~$n\ge 3$, a decomposition~$\I$ of~$[n]$, and sets of parameters~$\blambda$ and~$\mathbf{g}$ such that the corresponding pair~$(G,H)$ belongs to~$\Psi(n)$. Our result in Sections~\ref{section_ap_C} and~\ref{section_ap_R} will be the following: We shall describe several families $\Gamma(\bxi)$ of EPBQ-curves, each depending on certain set of parameters~$\bxi$, and prove that any \textit{realisable} EPBQ-curve belongs to one of these families. 

Now, let us obtain two results that will allow us to understand that certain EPBQ-curves are not realisable and, hence, can be thrown away. First, Corrollary~\ref{cor_ineq_main} and formulae~\eqref{eq_ABE_abe} immediately yield the following.

\begin{cor}\label{cor_ineq_main_abe}
Let $\Gamma$ be a realisable real EPBQ-curve, and let $a_{jl}$, $b_{jl}$, and~$e_{jl}$ be the coefficients of the corresponding relations~\eqref{eq_main_rel_abe}. Then for any $j\ne l$,  
\begin{equation}\label{eq_ineq_main_abe}
(1-a_{jl}e_{jl}-b_{jl}b_{lj})^2-4a_{jl}b_{jl}b_{lj}e_{jl}>0.
\end{equation}
\end{cor}

Second, let us prove the following lemma.

\begin{lem}\label{lem_realise}
Let $\Gamma$ be a  real EPBQ-curve, and let $a_{jl}$, $b_{jl}$, and~$e_{jl}$ be the coefficients of the corresponding relations~\eqref{eq_main_rel_abe}. Suppose that there exist indices $j\ne l$ such that the numbers $a_{jl}$, $-b_{jl}$, $-b_{lj}$, $e_{jl}$ are either all positive or all negative, and either $a_{jl}e_{jl}\ge1$ or $b_{jl}b_{lj}\ge1$. Then $\Gamma$ is not realisable.
\end{lem}

\begin{proof}
Assume that $\Gamma$ admits a realisation subject to a decomposition~$\I$ of $[n]$, $n\ge 3$. Choose any representatives $p\in I_j$ and $q\in I_l$.  Then $g_{pq}$ is given by~\eqref{eq_g_ne}. The Cauchy inequality easily implies that $|g_{pq}|\ge 1$. Therefore the $2\times 2$ principal minor of~$G$ 
corresponding to the rows and the columns of numbers~$p$ and~$q$ is non-positive, which yields a contradiction.
\end{proof}

\section{The simplest type of flexible cross-polytopes}\label{section_simplest}

According to Bricard's classification~\cite{Bri97}, there are three types of flexible octahedra in~$\E^3$. The octahedra of the first type are line-symmetric, the octahedra of the second type are plane-symmetric, and the octahedra of the third type do not possess any symmetry, and are called \textit{skew flexible octahedra\/}. From the geometric viewpoint  the third case is the most complicated one. However, from the algebraic viewpoint, it is the simplest case, since, for skew flexible octahedra, the dihedral angles vary so that the tangents of their halves remain proportional to each other. In this section, we construct multidimensional analogues of skew flexible octahedra. This means that we construct and classify flexible cross-polytopes of type~$(n)$, that is, flexible cross-polytopes corresponding to $m=1$, and to the simplest possible EPBQ-curve $\Gamma=\RP^1$. 
Then we have  $t_p(x)=\lambda_px$, $p=1,\ldots,n$, for some non-zero constants $\lambda_1,\ldots,\lambda_n$ such that $\lambda_p\ne \pm\lambda_q$ unless $p= q$. 

All $g_{pq}=g_{qp}$, $p\ne q$, are free parameters. So let~$G=(g_{pq})$ be any symmetric matrix of size $n\times n$ with units on the diagonal such that all its principal minors of sizes $2\times 2,\ldots,$ $(n-1)\times (n-1)$ are positive. Then 
\begin{equation}\label{eq_hpq_simple}
h_{pq}=\frac{2\lambda_p(\lambda_pg_{pq}-\lambda_q)}{\lambda_p^2-\lambda_q^2}
\end{equation}
if $p\ne q$, and $h_{pp}=1$.

In the Euclidean case, i.\,e., for $\det G=0$, $\ah_1,\ldots,\ah_n$ are the numbers such that the linear combination of the rows of~$G$ with coefficients  $\ah_1^{-1},\ldots,\ah_n^{-1}$ is equal to zero, and the numbers~$\bh_p$ are given by formula~\eqref{eq_beta_E}, which takes the form
\begin{equation*}%\label{eq_beta_E}
\bh_p=\left(\frac{1}{\ah _p}+2\lambda_p\sum_{q\ne p}\frac{\lambda_pg_{pq}-\lambda_q}{\ah _q(\lambda_p^2-\lambda_q^2)}\right)^{-1}.
\end{equation*}
It is easy to see that the expression in brackets does not vanish identically for all~$\blambda$, since, for instance, it is non-zero whenever $\lambda_p=0$. Hence, for each degenerate positive semidefinite matrix~$G$ with units on the diagonal and positive principal minors of sizes $2\times 2,\ldots,$ $(n-1)\times (n-1)$, there is a non-empty Zariski open subset of~$\R^n$ such that, for all~$\blambda$ in this subset, the numbers $a_1,\ldots,a_n$, $b_1,\ldots,b_n$ are well defined up to proportionality. Then  we choose points $\ba_1,\ldots,\ba_n\in\E^n$ satisfying~\eqref{eq_an_E}, put $\bc_p=a_p^{-1}\ba_p$, and obtain the flexible cross-polytope in~$\E^n$ with the parametrization 
\begin{align}\label{eq_simplest_param_a}
\ba_p(x)&=\ah_p\bc_p,\\
\bb_p(x)&=\bh_p\left(\sum_{q=1}^nh_{pq}\bc_q-\frac{2\lambda_p^2x^2}{\lambda_p^2x^2+1}\,\bn_p+\frac{2\lambda_px}{\lambda_p^2x^2+1}\,\bm\right).\label{eq_simplest_param_b}
\end{align}

In the spherical and the Lobachevsky cases, i.\,e., for $\det G\ne 0$, the numbers~$\ah_p$ and~$\bh_p$ are given by~\eqref{eq_alpha_SL} and~\eqref{eq_beta_SL}. If $\X^n=\bS^n$, then the numbers~$\ah_p$ and~$\bh_p$ are always well defined (up to signs), hence, for each positive definite~$G$ with units on the diagonal  and each $\blambda$ such that $\lambda_p\ne\pm\lambda_q$ unless $p=q$, we obtain the corresponding flexible cross-polytope in~$\bS^n$ with the  parametrization given by the same formulae~\eqref{eq_simplest_param_a} and~\eqref{eq_simplest_param_b}, where $\bc_1,\ldots,\bc_n$ is the basis in $\spa(\bn_1,\ldots,\bn_n)$ dual to the basis $\bn_1,\ldots,\bn_n$. 

Since in both cases $\X^n=\E^n$ and $\X^n=\bS^n$ the matrix~$G$ can be chosen arbitrarily, we obtain the following result.

\begin{theorem}
Let $\X^n$ be either~$\E^n$ or~$\bS^n$. Let $\Delta$ be an arbitrary non-degenerate $(n-1)$-dimensional simplex in~$\X^n$. Then there exists an essential flexible cross-polytope in~$\X^n$ with a facet congruent to~$\Delta$.
\end{theorem}

The case $\X^n=\Lambda^n$ is more difficult. We do not know whether an analogue of the latter theorem holds in this case. The conditions $\lambda_p\ne\pm\lambda_q$ unless $p=q$, and conditions~$(\Lambda1)$ and~$(\Lambda2)$ give an open semi-algebraic subset $\Theta\subset\R^{n+{n\choose 2}}$ consisting of all pairs~$(\blambda,G)$ such that the corresponding flexible cross-polytopes given by~\eqref{eq_simplest_param_a},~\eqref{eq_simplest_param_b} are well defined. The inequalities for $\Theta$ can be easily written down explicitly. However, we cannot extract from them a reasonable description of the topology and the geometry of~$\Theta$.

Nevertheless, let us  at least  show that  flexible cross-polytopes in~$\Lambda^n$ of type~$(n)$ actually exist for any~$n$.

\begin{propos}
The set $\Theta$ is non-empty.
\end{propos}

\begin{proof}
Take any degenerate symmetric matrix $G^0=(g_{pq}^0)$ of size $n\times n$ with units on the diagonal such that all its principal minors of sizes $2\times 2,\ldots,$ $(n-1)\times (n-1)$ are positive. Above we have shown that, for all~$\blambda$ in some Zariski open subset of~$\R^n$, the pair $(G^0,H^0)$ belongs to~$\Psi(\E^n)$. Here the entries of the matrix~$H^0$ are computed by formulae~\eqref{eq_hpq_simple} with all~$g_{pq}$ replaced by~$g_{pq}^0$. Choose any such vector~$\blambda$. By Corollary~\ref{cor_but}, the set $\Psi(n)=\Psi(\bS^n)\cup\Psi(\E^n)\cup\Psi(\Lambda^n)$ is open. Hence, for all symmetric matrices~$G$ with units on the diagonals sufficiently close to~$G^0$, the corresponding pairs~$(G,H)$ belong to~$\Psi(n)$. However, it is easy to see that any neighborhood of~$G^0$ contains matrices with negative determinants. Taking such matrix for $G$, we obtain that $(G,H)\in\Psi(\Lambda^n)$. Therefore $(\blambda,G)\in\Theta$.
\end{proof}

\section{Classification of complex EPBQ-curves}\label{section_ap_C} 

In notation concerning Jacobi's elliptic functions we follow~\cite{BaEr55}. In particular, we denote by~$k'$ the complementary elliptic modulus, $k'=\sqrt{1-k^2}$, and by~$K$ and~$iK'$ the real and the imaginary quarter-periods respectively. We freely use the formulae for the change of the variable by quarter- and half-periods~\cite[p.~350]{BaEr55}, and the transformation formulae for Jacobi's elliptic functions~\cite[p.~369]{BaEr55}.

\subsection{Elliptic EPBQ-curves.}\label{subsec_elliptic}
Let $k\in\C$ be an arbitrary elliptic modulus. With some abuse of notation, we omit~$k$ in notation for Jacobi's elliptic functions. It is well known (cf.~\cite[p.~529]{WhWa27}) that for $u+v+w=0$, we have
$$
1-\dn^2\! u-\dn^2\! v -\dn^2\!w+2\dn u\dn v\dn w=k^4\sn^2\!u\sn^2\!v\sn^2\!w.
$$ 
Using the evenness of the function $\dn$ and the oddness of the function~$\sn$, we easily conclude that the functions $d_1=\dn u$, and $d_2=\dn(u-\sigma)$ satisfy the relation
\begin{equation}\label{eq_rel_d}
\sn^2\!\sigma\cdot d_1^2d_2^2+\cn^2\!\sigma\cdot(d_1^2+d_2^2)-2\dn\sigma\cdot d_1d_2+k'^2\sn^2\!\sigma=0.
\end{equation}

Let $\Gamma$ be the irreducible algebraic curve in~$(\CP^1)^m$ parametrized by
\begin{equation}\label{eq_ell1}
z_j(u)=\dn(u-\sigma_j), \qquad j=1,\ldots,m,
\end{equation}
where the phases $\sigma_1,\ldots,\sigma_m$ are pairwise distinct modulo the lattice generated by the quarter-periods~$K$ and~$iK'$. 
Formula~\eqref{eq_rel_d} implies that the functions~$z_j(u)$ satisfy relations~\eqref{eq_main_rel_abe} with coefficients
\begin{align}\label{eq_A-E_ell}
a_{jl}&=\frac{\sn^2(\sigma_j-\sigma_l)}{\dn(\sigma_j-\sigma_l)}\,,&
b_{jl}&=\frac{\cn^2(\sigma_j-\sigma_l)}{\dn(\sigma_j-\sigma_l)}\,,&
e_{jl}&=\frac{k'^2\sn^2(\sigma_j-\sigma_l)}{\dn(\sigma_j-\sigma_l)}\,.
\end{align}
Since the phases   $\sigma_1,\ldots,\sigma_m$ are pairwise distinct modulo the lattice generated by~$K$ and~$iK'$, we see that any two functions~$z_j(u)$ and $z_l(u)$, $j\ne l$, are neither directly nor inversely proportional to each other, and the coefficients $a_{jl}$, $b_{jl}$, and~$e_{jl}$ are defined.
Then $\Gamma$ is a complex EPBQ-curve.
%Obviously, if we add the same constant to all~$\sigma_p$, the curve~$\Gamma$ will not change. Hence the obtained solution depends on $2n$ parameters, namely, the elliptic modulus~$k$, and the $n-1$ differences $\sigma_2-\sigma_1,\ldots,\sigma_n-\sigma_1$.

\subsection{Rational EPBQ-curves}\label{subsec_rational} Now, let us consider in more detail the degenerate case $k=1$. Then $\sn (u,1)=\tanh u$, $\cn(u,1)=\dn(u,1)=\cosh^{-1}\!u$, $K=\infty$, and $K'=\pi/2$.
After the transformation $(z_1,\ldots,z_m)\mapsto (2e^{\sigma_1}z_1^{-1},\ldots,2e^{\sigma_m}z_m^{-1})$ and the reparametrization $x=e^u$, we obtain the curve $\Gamma\subset(\CP^1)^m$ parametrized by
\begin{equation*}
z_j(x)=x+\frac{\mu_{j}}{x},
\end{equation*} 
where $\mu_j=e^{2\sigma_j}$. 
The phases $\sigma_1,\ldots,\sigma_m$ are pairwise distinct modulo $\frac{1}{2}i\pi\Z$. Therefore $\mu_j\ne\pm\mu_l$ unless $j=l$. Formulae~\eqref{eq_A-E_ell}  for the coefficients~$a_{jl}$, $b_{jl}$, and~$e_{jl}$ take the following form: 
\begin{align}\label{eq_rat_ABE}
a_{jl}&=0,&
b_{jl}&=\frac{2\mu_l}{\mu_{j}+\mu_{l}}\,,&
e_{jl}&=\frac{2(\mu_{j}-\mu_{l})^2}{\mu_{j}+\mu_{l}}\,.
\end{align}

\begin{remark}
The coefficients~$\mu_j$ were initially defined by $\mu_j=e^{2\sigma_j}$.  Nevertheless, it is easy to check that if one of the coefficients~$\mu_j$ is equal to~$0$, then the above formulae still give a EPBQ-curve. This EPBQ-curve cannot be obtained as a partial case of the elliptic EPBQ-curves constructed in the previous section. However, it can be obtained as a limit of such EPBQ-curves. Hence we do not require that all $\mu_j$ are non-zero.
\end{remark}

\subsection{Exotic EPBQ-curves}\label{subsec_exotic} 
Identity~\eqref{eq_rel_d} allows us to construct one more family of EPBQ-curves in $(\CP^1)^3$. These EPBQ-curves will be called \textit{exotic\/}. Take $k\ne 1$ and $\sigma=\frac{K}{2}$ in~\eqref{eq_rel_d}. It is well known (cf.~\cite[pp.~346--348]{BaEr55}) that
\begin{equation*}
\sn\frac{K}{2}=\frac{1}{\sqrt{1+k'}},\qquad
\cn\frac{K}{2}=\sqrt{\frac{k'}{1+k'}},\qquad
\dn\frac{K}{2}=\sqrt{k'}.
\end{equation*}
Hence~\eqref{eq_rel_d} takes the form
\begin{equation}\label{eq_rel_d_special}
\left(d_1+\frac{k'}{d_1}\right)\left(d_2+\frac{k'}{d_2}\right)=2\sqrt{k'}(1+k').
\end{equation}
We put
\begin{align}\label{eq_par_C_ex}
z_1(u)&=\dn u,&
z_2(u)&=\dn\left(u-\frac{K}{2}\right),&
z_3(u)&=\dn u+\frac{k'}{\dn u}\,.
\end{align}
Formula~\eqref{eq_rel_d_special} easily implies that these three functions satisfy relations~\eqref{eq_main_rel_abe} with the matrices of coefficients $\CA=(a_{jl})$, $\CB=(b_{jl})$, and $\CE=(e_{jl})$ given by
\begin{gather}\label{eq_C_ABE1}
\CA=\frac{1}{\sqrt{k'}(1+k')}
\begin{pmatrix}
*&1&0\\1&*&1\\0&1&*
\end{pmatrix}\qquad\quad
\CB=
\begin{pmatrix}
*&\frac{\sqrt{k'}}{1+k'}&2\\\frac{\sqrt{k'}}{1+k'}&*&0\\0&\frac{\sqrt{k'}}{1+k'}&*
\end{pmatrix}\\\label{eq_C_ABE2}
\CE=
\begin{pmatrix}
*&\frac{k'\sqrt{k'}}{1+k'}&2k'\\ \frac{k'\sqrt{k'}}{1+k'}&*&0\\2k'\vphantom{\frac{k'\sqrt{k'}}{1+k'}}&0&*
\end{pmatrix}
\end{gather}
(Recall that diagonal entries of~$\CA$, $\CB$, and~$\CE$ are undefined.) Hence the curve~$\Gamma\subset(\CP^1)^3$ parametrized by~\eqref{eq_par_C_ex} is an EPBQ-curve.

Instead of $\sigma=\frac{K}{2}$, we may take $\sigma=\frac{K}{2}+iK'$. We have
\begin{equation*}
\sn\sigma=\frac{1}{\sqrt{1-k'}},\qquad
\cn\sigma=-i\sqrt{\frac{k'}{1-k'}},\qquad
\dn\sigma=-i\sqrt{k'}.
\end{equation*}
The corresponding EPBQ-curve is constructed in the same way as in the previous case. Using the identity $\dn (v-iK')=\frac{i\cn v}{\sn v}$, we can write it in the following form:
\begin{align}
z_1(u)&=\dn u,&
z_2(u)&=\frac{i\cn\left(u-\frac{K}{2}\right)}{\sn\left(u-\frac{K}{2}\right)},&
z_3(u)&=\dn u-\frac{k'}{\dn u}\,.
\end{align}
The formulae for~$\CA$, $\CB$, and~$\CE$ are obtained from~\eqref{eq_C_ABE1}, \eqref{eq_C_ABE2} by replacing $\sqrt{k'}$ with $-i\sqrt{k'}$ everywhere.

\begin{remark}
The above two families of exotic EPBQ-curves in fact form one family, since they can be obtained from each other by analytic continuation. Indeed, the complementary modulus~$k'$ and the quarter-period~$K$ considered as functions of~$k$ have a branch point at $k=1$. The monodromy along a circuit around this point takes the first exotic EPBQ-curve to the second one.   
\end{remark}

\subsection{Classification}

\begin{theorem}\label{theorem_sol_C}
Let $\Gamma\subset(\CP^1)^m,$ $m\ge 2,$ be a complex EPBQ-curve.  Assume that the coefficients of the corresponding relations~\eqref{eq_main_rel_abe} satisfy 
\begin{equation}
(1-a_{jl}e_{jl}-b_{jl}b_{lj})^2-4a_{jl}b_{jl}b_{lj}e_{jl}\ne 0\label{eq_ineq_cond4}
\end{equation}
for all $j\ne l$.
Then  $\Gamma$ is equivalent to one of the curves constructed in Sections~\ref{subsec_elliptic}, \ref{subsec_rational}, and~\ref{subsec_exotic}. 
\end{theorem}
\begin{remark} 
By Corrollary~\ref{cor_ineq_main},  inequalities~\eqref{eq_ineq_cond4} always hold for EPBQ-curves corresponding to flexible cross-polytopes. By Theorem~\ref{theorem_sol_C}, any such EPBQ-curve is equivalent to an EPBQ-curve in our list. Possibly, there exist additional EPBQ-curves that are not in our list. We are not interested in them, since they do not correspond to flexible cross-polytopes. 
\end{remark}

\begin{proof}
Let $\C(\Gamma)$ be the field of rational functions on the curve~$\Gamma$. In the rest of this section we shall always denote by~$z_j$ the restriction of the coordinate~$z_j$ to~$\Gamma$; then $z_j\in\C(\Gamma)$. For any $j_1,\ldots,j_q$, we denote by $\C(z_{j_1},\ldots,z_{j_q})$ the subfield of~$\C(\Gamma)$ generated over~$\C$ by the elements~$z_{j_1},\ldots,z_{j_q}$. If $F$ is a subfield of a field~$E$, then we denote by~$|E/F|$ the degree of the extension $E\supset F$, that is, the dimension of~$E$ as a vector space over~$F$.

Let $R^{\min}_{jl}(z_j,z_l)=0$ be the minimal polynomial relation between $z_j$ and~$z_l$. Then this relation divides the relation 
$$
R_{jl}(z_j,z_l)=a_{jl}z_j^2z_l^2+b_{jl}z_j^2-2z_jz_l+b_{lj}z_l^2+e_{jl}=0
$$
in the definition of an EPBQ-curve. There are two possibilities (compare with the proof of Lemma~\ref{lem_nonproportional}):

(1) The polynomial $R^{\min}_{jl}$ has degree~$2$ with respect to either of the variables~$z_j$ and~$z_l$. Then $R^{\min}_{jl}=R_{jl}$ and neither $z_j$ belongs to~$\C(z_l)$ nor $z_l$ belongs to~$\C(z_j)$. Therefore, $|\C(z_j,z_l)/\C(z_j)|=|\C(z_j,z_l)/\C(z_l)|=2$. Such pair $(j,l)$ will be called \textit{non-special\/}. It is easy to check that this possibility occurs if and only if at most one of the four coefficients~$a_{pq}$, $b_{pq}$, $b_{qp}$, and~$e_{pq}$ vanishes (cf.~\cite{Bri97}).

(2) The relation $R^{\min}_{jl}(z_j,z_l)=0$ has degree~$2$ with respect to one of the variables, say~$z_j$, and degree~$1$ with respect to the other variable~$z_l$. Then
\begin{equation}\label{eq_special}
z_l=\left(c_1z_j+\frac{c_2}{z_j}\right)^{\pm 1},\qquad c_1,c_2\ne0.
\end{equation}
We obtain that $z_l\in\C(z_j)$, but $z_j\notin\C(z_l)$. Therefore, $\C(z_j,z_l)=\C(z_j)$ and $|\C(z_j)/\C(z_l)|=2$. Such pair $(j,l)$ will be called \textit{special\/}. It is easy to check that this possibility occurs if and only if either $a_{jl}=b_{lj}=0$ and~$b_{jl},e_{jl}\ne 0$ or  $b_{jl}=e_{jl}=0$ and~$a_{jl},b_{lj}\ne 0$. %(Similarly, we have relation~\eqref{eq_special} with $t_p$ and~$t_q$ interchanged if and only if either $A_{pq}=B_{pq}=0$ and~$D_{pq},E_{pq}\ne 0$ or  $D_{pq}=E_{pq}=0$ and~$A_{pq},B_{pq}\ne 0$.)
{\sloppy

}

\begin{lem}\label{lem_tr_rat}
If a pair $(j,l)$ is non-special, then any $z_r$ is a rational function in~$z_j$ and~$z_l$. Thus
$\C(z_j,z_l)=\C(\Gamma)$ and $|\C(\Gamma)/\C(z_j)|=|\C(\Gamma)/\C(z_l)|=2$.   
\end{lem}
\begin{proof}
We have
\begin{align}\label{eq_rel_pr}
(a_{jr}z_j^2+b_{rj})z_r^2-2z_jz_r+(b_{jr}z_j^2+e_{jr})&=0,\\
(a_{lr}z_l^2+b_{rl})z_r^2-2z_lz_r+(b_{lr}z_l^2+e_{lr})&=0.
\end{align}
Hence, $P(z_j,z_l)z_r=Q(z_j,z_l)$, where
\begin{align*}
P(z_j,z_l)&=2((a_{lr}z_l^2+b_{rl})z_j-(a_{jr}z_j^2+b_{rj})z_l),\\
Q(z_j,z_l)&=(a_{lr}z_l^2+b_{rl})(b_{jr}z_j^2+e_{jr})-(a_{jr}z_j^2+b_{rj})(b_{lr}z_l^2+e_{lr}).
\end{align*}
If $P(z_j,z_l)\ne 0$, we immediately obtain that $z_r$ is a rational function in~$z_j$ and~$z_l$.
Assume that $P(z_j,z_l)=0$. Since the pair~$(j,l)$ is non-special,  the polynomial $P$ must be divisible by $R_{jl}$. The degree of $R_{jl}$ with respect to either of the variables~$z_j$ and~$z_l$ is equal to~$2$, and the degree of $P$ with respect to either of the variables~$z_j$ and~$z_l$ does not exceed~$2$. Hence $P=cR_{jl}$  for some constant~$c$. Since the polynomial~$P$ does not contain the monomial~$z_jz_l$, we obtain that $c=0$.  Consequently, all coefficients of~$P$ vanish, in particular, $a_{jr}=b_{rj}=0$. Then~\eqref{eq_rel_pr} implies that $z_r$ is a rational function in~$z_j$.
\end{proof}

\begin{lem}\label{lem_3_cases}
Exactly one of the following three assertions holds:
\begin{enumerate}
\item There exists $j\in[m]$ such that $\C(\Gamma)=\C(z_j)$, i.\,e., all~$z_l$ are rational functions in~$z_j$. 
\item All pairs $(j,l)$, $j,l\in[m]$, $j\ne l$, are non-special.
\item $m=3$ and, after renumbering the functions $z_1$, $z_2$, and~$z_3$, we have
$$
z_3=c_1z_1+\frac{c_1'}{z_1}=\left(c_2z_2+\frac{c_2'}{z_2}\right)^{-1}
$$
for some non-zero complex numbers $c_1$, $c_1'$, $c_2$, and~$c_2'$, and the pair~$(2,3)$ is non-special. Then $|\C(\Gamma)/\C(z_1)|=|\C(\Gamma)/\C(z_2)|=2$ and $|\C(\Gamma)/\C(z_3)|=4$.
\end{enumerate}
\end{lem}

\begin{proof}
Assume that assertions~(1) and~(2) do not hold. Since~(2) does not hold, there exists a special pair $(j,l)$, i.\,e., such that $z_l$ is a rational function in~$z_j$. Then $|\C(z_j)/\C(z_l)|=2$. Since~(1) does not hold, we obtain that $\C(\Gamma)\ne\C(z_j)$, hence, $|\C(\Gamma)/\C(z_l)|>2$. Besides, we see that $m\ge 3$.
By Lemma~\ref{lem_tr_rat}, no pair~$(l,r)$ can be non-special. Hence, for each $r\ne l$, either $z_l\in\C(z_r)$ or $z_r\in\C(z_l)$. In the latter case, we would obtain that $|\C(z_j)/\C(z_r)|=4$, which is impossible. Therefore $z_l\in\C(z_r)$ for all $r\ne l$.
Then 
$$
z_l=\left(c_{r}z_r+\frac{c'_r}{z_r}\right)^{\varepsilon_r}
$$
for some non-zero complex numbers $c_r$ and~$c'_r$, and $\varepsilon_r=\pm 1$. 
For any $r_1,r_2\ne l$, $r_1\ne r_2$, we obtain that $|\C(\Gamma)/\C(z_{r_1})|=|\C(\Gamma)/\C(z_{r_2})|$. Therefore the pair $(r_1,r_2)$ is non-special. Lemma~\ref{lem_tr_rat} implies that $|\C(\Gamma)/\C(z_{r})|=2$ for all $r\ne l$. Hence $|\C(\Gamma)/\C(z_{l})|=4$.

Suppose that there exist indices $r_1,r_2\ne l$, $r_1\ne r_2$, such that $\varepsilon_{r_1}=\varepsilon_{r_2}$. Then 
\begin{equation*}%\label{eq_rel_r1r2}
%c_{r_1}t_{r_1}+\frac{c'_{r_1}}{t_{r_1}}=c_{r_2}t_{r_2}+\frac{c'_{r_2}}{t_{r_2}}\,,\\
c_{r_1}z_{r_1}^2z_{r_2} - c_{r_2}z_{r_1}z^2_{r_2} -c'_{r_2}z_{r_1}+ c'_{r_1}z_{r_2}=0.
\end{equation*}
Since the pair $(r_1,r_2)$ is non-special, this relation must be proportional to the relation $R_{r_1r_2}(z_{r_1},z_{r_2})=0$, which is impossible. Hence $\varepsilon_{r_1}=-\varepsilon_{r_2}$ for any $r_1,r_2\ne l$, $r_1\ne r_2$. Therefore $m=3$, $\varepsilon_{r_1}=1$ and $\varepsilon_{r_2}=-1$ for the two distinct elements~$r_1$ and~$r_2$ of the set $\{1,2,3\}\setminus\{l\}$. Thus, assertion~(3) holds.
\end{proof}

\begin{lem}\label{lem_xieta}
Let $\F$ be a field, let $x$ and~$y$ be elements of~$\F$ such that $\xi=\sqrt{x}$ and~$\eta=\sqrt{y}$ do not   belong to~$\F$. Suppose that $\eta=\varphi(\xi)$ for a rational function~$\varphi$ with coefficients in~$\F$. Then $\eta/\xi$ belongs to~$\F$.  
\end{lem}
\begin{proof}
Since $\xi^2=x$, we see that $\varphi(\xi)=a\xi+b$ for some $a,b\in\F$.
Since $\eta\notin\F$, we obtain that $a\ne 0$. We have $y=a^2x+b^2+2ab\xi$. Since $\xi\notin\F$, we obtain that $b=0$. Hence $\eta/\xi=a$ is an element of~$\F$.
\end{proof}

The following lemma is (up to terminology) due to Bricard~\cite{Bri97}. However, in his proof one important case was missed. Hence we give here a complete proof. 

\begin{lem}\label{lem_proportion}
Suppose that the pairs $(j,l)$ and~$(j,r)$ are non-special. Then the polynomials
\begin{align*}
F_{jl}(z_j)&=-a_{jl}b_{jl}z_j^4+(1-a_{jl}e_{jl}-b_{jl}b_{lj})z_j^2-b_{lj}e_{jl},\\
F_{jr}(z_j)&=-a_{jr}b_{jr}z_j^4+(1-a_{jr}e_{jr}-b_{jr}b_{rj})z_j^2-b_{rj}e_{jr}
\end{align*}
coincide up to multiplication by a non-zero constant.
\end{lem}
\begin{proof}
We have
\begin{equation*}
z_l=\frac{z_j+\xi}{a_{jl}z_j^2+b_{lj}}\,,\qquad
z_r=\frac{z_j+\eta}{a_{jr}z_j^2+b_{rj}}
\end{equation*}
for  some values of the square roots $\xi=\sqrt{F_{jl}(z_j)}$ and $\eta=\sqrt{F_{jr}(z_j)}$.
Since the pairs $(j,l)$ and~$(j,r)$ are non-special, we obtain that  $\xi$ and~$\eta$  do not belong to the field~$\C(z_j)$.
By Lemma~\ref{lem_tr_rat}, $z_r$ is a rational function in~$z_j$ and~$z_l$. Hence, $\eta$ is a rational function in~$z_j$ and~$\xi$, that is, a rational function in~$\xi$ with coefficients in~$\C(z_j)$. Applying Lemma~\ref{lem_xieta} to the field $\C(z_j)$ and the elements~$\xi$ and~$\eta$, we obtain that 
$F_{jr}(z_j)=Q^2(z_j)F_{jl}(z_j)$
 for a rational function~$Q\in\C(z_j)$. 
 
Since the polynomial~$F_{jl}$ is even, its roots are symmetric with respect to~$0$. Since $\sqrt{F_{jl}(z_j)}\notin\C(z_j)$, there are $3$ possibilities: (1) $F_{jl}$ has four non-zero pairwise distinct roots, (2) $F_{jl}$ has root~$0$ of multiplicity two, and two distinct non-zero  roots, (3) $F_{jl}$ has degree two, and has two distinct non-zero  roots. The same holds for~$F_{jr}$. It follows easily that the equality $F_{jr}(z_j)=Q^2(z_j)F_{jl}(z_j)$ may hold only if $Q(z_j)$ is either a non-zero constant or has the form $cz_j^{\pm1}$, $c\ne 0$. (The latter case  was missed by Bricard.) If $Q(z_j)=c$, the lemma follows. 

Suppose, $Q(z_j)=cz_j$. (The case $cz_j^{-1}$ is similar.) Then $a_{jl}b_{jl}=0$ and $b_{rj}e_{jr}=0$. 
The polynomial~$F_{jl}$ will not change if we replace the coordinate~$z_l$ with $-z_l^{-1}$. This replacement turns~$a_{jl}$, $b_{jl}$, $b_{lj}$, and~$e_{jl}$ into $-b_{jl}$, $-a_{jl}$, $-b_{lj}$, and~$-e_{jl}$ respectively. Hence, without loss of generality, we may assume that $a_{jl}=0$. Similarly, we may assume that $b_{rj}=0$. Since the pairs~$(j,l)$ and~$(j,r)$ are non-special,  all other coefficients $b_{jl}$, $b_{lj}$, $e_{jl}$, $a_{jr}$, $b_{jr}$, and~$e_{jr}$ are non-zero. We have
$$
z_l=\frac{z_j+\xi}{b_{lj}}\,,\qquad z_r=\frac{1+c\xi}{a_{jr}z_j}\,. 
$$
Excluding~$\xi$ and solving the obtained equation in the variable~$z_j$, we obtain
$$
z_j=\frac{cb_{lj}z_l+1}{a_{jr}z_r+c}\,.
$$
Substituting this to the relation $R_{jl}(z_j,z_l)=0$,
we obtain
\begin{equation}\label{eq_rel_strange}
b_{jl}(cb_{lj}z_l+1)^2-2z_l(cb_{lj}z_l+1)(a_{jr}z_r+c)+(b_{lj}z_l^2+e_{jl})(a_{jr}z_r+c)^2=0.
\end{equation}
This relation has degree~$2$ with respect to either of the variables~$z_l$ and~$z_r$, since the coefficient of~$z_l^2z_r^2$ is $b_{lj}a_{jr}^2\ne 0$.
Since the pairs $(j,l)$ and $(j,r)$ are non-special, the pair $(l,r)$ is also non-special. Hence relation~\eqref{eq_rel_strange} must coincide with the relation $R_{lr}(z_l,z_r)=0$
up to multiplication by a non-zero constant. Nevertheless, this is not true, since the coefficient of~$z_r$ in~\eqref{eq_rel_strange} is $2ce_{jl}a_{jr}\ne 0$. This contradiction proves that the case $Q=cz_j$ is impossible, which completes the proof of the lemma.
\end{proof}
 
 Now, we proceed with the proof of Theorem~\ref{theorem_sol_C}. Consider three cases corresponding to the three possibilities in Lemma~\ref{lem_3_cases}:
 
 \textbf{Case 1.} Assertion~(1) of Lemma~\ref{lem_3_cases} holds. Without loss of generality we may assume that $\C(z_1)=\C(\Gamma)$. Then, for each $l>1$, we have
\begin{equation}
z_l=\left(c_lz_1+\frac{c_l'}{z_1}\right)^{\varepsilon_l},\qquad c_l,c_l'\ne 0,\ \ \varepsilon_l=\pm 1.
\end{equation}
By the transformation $(z_1,z_2,\ldots,z_m)\mapsto \left(z_1,\frac{z_2^{\varepsilon_2}}{c_2},\ldots,\frac{z_m^{\varepsilon_m}}{c_m}\right)$, we may obtain that $z_l=z_1+\frac{c_l'}{c_lz_1}$ for all $l>1$. Putting $x=z_1$, $\mu_1=0$, and $\mu_l=\frac{c_l'}{c_l}$ for all $l>1$, we obtain the  rational EPBQ-curve described in Subsection~\ref{subsec_rational}.  
 
\textbf{Case 2.} Assertion~(2) of Lemma~\ref{lem_3_cases} holds, i.\,e., all pairs $(j,l)$ are non-special. Then Lemma~\ref{lem_proportion} implies that the number
\begin{equation}\label{eq_kappa}
\kappa=\frac{(1-a_{jl}e_{jl}-b_{jl}b_{lj})^2-2a_{jl}b_{jl}b_{lj}e_{jl}}{a_{jl}b_{jl}b_{lj}e_{jl}}
\end{equation}
is the same for all pairs~$(j,l)$. By~\eqref{eq_ineq_cond4}, the numerator and the denominator in this formula cannot vanish simultaneously. Hence either $\kappa$ is a well-defined complex number or $\kappa=\infty$. We introduce the elliptic moduli~$k$ and~$k'$ by
\begin{equation}
k'^2+\frac{1}{k'^2}=\kappa,\qquad
k=\sqrt{1-k'^2}\,.\label{eq_k}
\end{equation}
(We take for $k$ and~$k'$ any solutions of these equations.) In particular, if $\kappa=\infty$, then we take $k'=0$, $k=1$. By~\eqref{eq_ineq_cond4}, $\kappa\ne 2$. Hence $k'^2\ne 1$ and $k\ne 0$.

If $\kappa=\infty$, then, for each pair $(j,l)$, exactly one of the four coefficients~$a_{jl}$, $b_{jl}$, $b_{lj}$, and~$e_{jl}$ equals zero. Replacing, if necessary, the coordinate~$z_1$ with~$z_1^{-1}$, we may achieve that $b_{21}e_{12}=0$. Then it follows from Lemma~\ref{lem_proportion} that $b_{l1}e_{1l}=0$ for all $l\ne 1$. For each~$l\ne 1$, replacing, if necessary, $z_l$ with~$z_l^{-1}$, we achieve that $e_{1l}=0$. Then, by Lemma~\ref{lem_proportion}, $e_{jl}=0$ for all $j\ne l$. Therefore all coefficients $a_{jl}$ and~$b_{jl}$ are non-zero. 

If $\kappa\ne\infty$, then all coefficients $a_{jl}$, $b_{jl}$, and~$e_{jl}$ are non-zero. 

Let us show that $\Gamma$ has a parametrization of the form
\begin{equation}\label{eq_param_Gamma}
z_j(u)=\nu_j\dn(u-\sigma_j),\qquad j=1,\ldots,m.
\end{equation}
(All  elliptic functions correspond to the modulus~$k$.) It follows easily from \eqref{eq_A-E_ell} that, if $\Gamma$ had such parametrization, then the coefficients $a_{jl}$, $b_{jl}$, and~$e_{jl}$ would be given by 
\begin{equation}\label{eq_syst_nu_sigma}
a_{jl}=\frac{\sn^2(\sigma_j-\sigma_l)}{\nu_j\nu_l\dn(\sigma_j-\sigma_l)},\ \ 
b_{jl}=\frac{\nu_l\cn^2(\sigma_j-\sigma_l)}{\nu_j\dn(\sigma_j-\sigma_l)},\ \ 
e_{jl}=\frac{k'^2\nu_j\nu_l\sn^2(\sigma_j-\sigma_l)}{\dn(\sigma_j-\sigma_l)}.
\end{equation} 
Consider these formulae as the system of equations in the variables $\nu_1,\ldots,\nu_m$, $\sigma_1,\ldots,\sigma_m$. Our goal is to show that this system of equations has a solution, and this solution actually provides a parametrization for~$\Gamma$ by~\eqref{eq_param_Gamma}.
It is easy to see that equations~\eqref{eq_syst_nu_sigma} imply that 
\begin{equation*}
\nu_j=\sqrt{\frac{1-a_{jl}e_{jl}-b_{jl}b_{lj}}{(1+k'^2)a_{jl}b_{jl}}}\ ,\qquad l\ne j.
\end{equation*}
By Lemma~\ref{lem_proportion}, the right-hand side is independent of the choice of $l\ne j$. Hence we can actually determine~$\nu_j$ by this formula. (A branch of the square root is chosen arbitrarily.) 

Without loss of generality, we may put
$\sigma_1=0$. Then system of equations~\eqref{eq_syst_nu_sigma} for a pair~$(1,l)$, $l\ne1$, takes the form
\begin{equation}\label{eq_syst_q}
\left\{
\begin{aligned}
\frac{\sn^2\sigma_l}{\dn\sigma_l}&=\xi_{l},&\qquad&&\xi_{l}&=\frac{1-a_{1l}e_{1l}-b_{1l}b_{l1}}{(1+k'^2)\sqrt{b_{1l}b_{l1}}}\,,\\
\frac{\cn^2\sigma_l}{\dn\sigma_l}&=\eta_{l},&\qquad&&\eta_{l}&=\sqrt{b_{1l}b_{l1}}\,.
\end{aligned}
\right.
\end{equation}
Using~\eqref{eq_kappa}, we obtain that $k'^2\xi_{l}^2+(1+k'^2)\xi_{l}\eta_{l}+\eta_{l}^2=1$, which easily implies that the system of equations~\eqref{eq_syst_q} has a solution~$\sigma_l$. 

Let $\gamma$ be a generic point of~$\Gamma$ in the following sense. Firstly, $\gamma$ is a smooth point of~$\Gamma$. Secondly, $\gamma$ is a regular point for every function~$z_j$, $j=1,\ldots,m$. Thirdly, $\nu_1^{-1}z_1(\gamma)$ is a regular value of the function~$\dn$, i.\,e., is not equal to~$\pm 1$, $\pm k'$. Then we can introduce a local coordinate~$u$ on~$\Gamma$ such that $z_1(u)=\nu_1\dn u$ in a neighborhood of~$\gamma$. If~$\sigma_l$, $l\ne 1$, is a solution of the system of equations~\eqref{eq_syst_q}, i.\,e., equations~\eqref{eq_syst_nu_sigma}  for the pair~$(1,l)$, then the equation
$$
a_{1l}z_1^2z_l^2+b_{1l}z_1^2-2z_1z_l+b_{l1}z_l^2+e_{1l}=0,\qquad z_1=\nu_1\dn u,
$$
has two solutions $z_l=\nu_l\dn(u\pm\sigma_l)$. Reversing, if necessary, the sign of~$\sigma_l$, we achieve that $z_l(u)=\nu_l\dn(u-\sigma_l)$ in a neighborhood of~$\gamma$. Formulae~\eqref{eq_param_Gamma} parametrize an irreducible algebraic curve that coincides with~$\Gamma$ in a neighborhood of~$\gamma$. Hence these formulae parametrize $\Gamma$. (Notice that this immediately implies that the obtained coefficients $\nu_1,\ldots,\nu_m$, $\sigma_1,\ldots,\sigma_m$ actually yield a solution of the whole system of equations~\eqref{eq_syst_nu_sigma}.) Therefore $\Gamma$ is equivalent to an EPBQ-curve constructed in Subsection~\ref{subsec_elliptic}.

\textbf{Case 3.} Assertion~(3) of Lemma~\ref{lem_3_cases} holds. Then $(1,2)$ is the only non-special pair. We put
\begin{equation*}
\kappa=\frac{(1-a_{12}e_{12}-b_{12}b_{21})^2-2a_{12}b_{12}b_{21}e_{12}}{a_{12}b_{12}b_{21}e_{12}}
\end{equation*}
and define the elliptic moduli~$k$ and~$k'$ by the same formulae~\eqref{eq_k} as in the previous case. Then $k'^2\ne 1$ and $k\ne 0$. Also, as in the previous case, we obtain that~$\Gamma$ has a parametrization of the form
\begin{align}
z_j(u)&=\nu_j\dn(u-\sigma_j),\qquad j=1,2,\\
z_3(u)&=c_1z_1(u)+\frac{c_1'}{z_1(u)}=\left(c_2z_2(u)+\frac{c_2'}{z_2(u)}\right)^{-1},\label{eq_t3}
\end{align}
where 
\begin{equation*}
\nu_1=\sqrt{\frac{1-a_{12}e_{12}-b_{12}b_{21}}{(1+k'^2)a_{12}b_{12}}}\ ,\qquad 
\nu_2=\sqrt{\frac{1-a_{12}e_{12}-b_{12}b_{21}}{(1+k'^2)a_{12}b_{21}}}\ ,
\end{equation*}
$\sigma_1=0$, and $\sigma_2$ is one of the solutions of the system of equations~\eqref{eq_syst_q} for $l=2$. Formula~\eqref{eq_t3} shows that 
$$a_{12}z_1^2z_2^2+b_{12}z_1^2+b_{21}z_2^2+e_{12}=2(c_1z_1^2+c_1')(c_2z_2^2+c_2').$$
 Therefore $a_{12}e_{12}=b_{12}b_{21}$. Then~\eqref{eq_syst_nu_sigma} implies that $k'^2\sn^4\sigma_2=\cn^4\sigma_2$. Since $k'^2\ne 1$, this is equivalent to $\dn^4\sigma_2=k'^2$. 
 
If $\dn^2\sigma_2=k'$, we obtain that $\sigma_2=\frac{K}{2}+Kq +2iK'q'$ for some $q,q'\in\Z$. Since $\dn (v+2iK')=-\dn v$ and $\dn(v+K)=-k'\dn^{-1}v$, replacing $z_1$ by $\pm z_1^{\pm 1}$, we may achieve that $\sigma_2=\frac{K}{2}$. Thus we obtain the first exotic elliptic EPBQ-curve described in Subsection~\ref{subsec_exotic}. Similarly, if $\dn^2\sigma_2=-k'$, we obtain the second exotic elliptic EPBQ-curve described in Subsection~\ref{subsec_exotic}.
\end{proof}

\section{Classification of real EPBQ-curves}\label{section_ap_R}

In this section we give a list of real EPBQ-curves $\Gamma\subset(\RP^1)^m$, $m\ge 2$. As in the previous section we provide explicit parametrizations for these curves. Further, we prove that any \textit{realisable\/} EPBQ-curve is equivalent to a EPBQ-curve in our list. 

\subsection{Rational EPBQ-curves}\label{subsec_rational_R} 
For  a set $\bmu=(\mu_1,\ldots,\mu_m)$ of real numbers with pairwise distinct absolute values $|\mu_1|,\ldots,|\mu_m|$, let $\Gamma^{rat}(\bmu)\subset(\RP^1)^m$ be the EPBQ-curve parametrized by 
\begin{equation*}
z_j(x)=x+\frac{\mu_j}{x}\,,\qquad j=1,\ldots,m,
\end{equation*}
where $x$ runs over~$\R$. The coefficients~$a_{jl}$, $b_{jl}$, and~$e_{jl}$ are given by~\eqref{eq_rat_ABE}.

\subsection{Elliptic EPBQ-curves of the first kind} Let $0<k<1$ be an elliptic modulus. Let $\bsigma=(\sigma_1,\ldots,\sigma_m)$ be a set of real numbers pairwise distinct modulo $K\Z$. Let $m'$ be an arbitrary integer in~$[0,m]$. Let $\Gamma^{ell}_{1,m'}(k,\bsigma)\subset(\RP^1)^m$ be the curve 
parametrized by
\begin{equation*}
z_j(u)=\left\{
\begin{aligned}
&\dn(u-\sigma_{j})&&\text{if  $j\le m'$,}\\
&\frac{\cn(u-\sigma_{j})}{\sn(u-\sigma_{j})}&&\text{if  $j>m'$,}
\end{aligned}
\right.
\end{equation*}
where $u$ runs over~$\R$. To write formulae for the coefficients~$a_{jl}$, $b_{jl}$, and~$e_{jl}$, we shall conveniently introduce the signs $\varepsilon_j\in\{-1,1\}$ such that $\varepsilon_j=1$ for $j\le m'$ and  $\varepsilon_j=-1$ for $j> m'$. If $\varepsilon_j=\varepsilon_l=\varepsilon$, then 
\begin{equation}\label{eq_a-e_1-f}
\begin{aligned}
a_{jl}&=\frac{\varepsilon\sn^2(\sigma_j-\sigma_l)}{\dn(\sigma_j-\sigma_l)}\,,&
b_{jl}&=\frac{\cn^2(\sigma_j-\sigma_l)}{\dn(\sigma_j-\sigma_l)}\,,&
e_{jl}&=\frac{\varepsilon k'^2\sn^2(\sigma_j-\sigma_l)}{\dn(\sigma_j-\sigma_l)}\,.
\end{aligned}
\end{equation}
If $\varepsilon_j=-\varepsilon_l=\varepsilon$, then 
\begin{gather}
a_{jl}=\frac{\varepsilon}{k^2\sn(\sigma_j-\sigma_l)\cn(\sigma_j-\sigma_l)}\,,\qquad
b_{jl}=\frac{\dn^2(\sigma_j-\sigma_l)}{k^2\sn(\sigma_j-\sigma_l)\cn(\sigma_j-\sigma_l)}\,,\\
\label{eq_a-e_1-l}
e_{jl}=-\frac{\varepsilon k'^2}{k^2\sn(\sigma_j-\sigma_l)\cn(\sigma_j-\sigma_l)}\,.
\end{gather}

\subsection{Elliptic EPBQ-curves of the second kind} Let $k$, $\bsigma$, and $m'$ be as in the previous case. Let $\Gamma^{ell}_{2,m'}(k,\bsigma)\subset(\RP^1)^m$ be the curve 
parametrized by
\begin{equation*}
z_j(u)=\left\{
\begin{aligned}
&\cn(u-\sigma_{j})&&\text{if  $j\le m'$,}\\
&\frac{\dn(u-\sigma_{j})}{k\sn(u-\sigma_{j})}&&\text{if  $j>m'$,}
\end{aligned}
\right.
\end{equation*}
where $u$ runs over~$\R$. 
If $\varepsilon_j=\varepsilon_l=\varepsilon$, then 
\begin{equation}
\begin{aligned}\label{eq_a-e_2-f}
a_{jl}&=\frac{\varepsilon k^2\sn^2(\sigma_j-\sigma_l)}{\cn(\sigma_j-\sigma_l)}\,,&
b_{jl}&=\frac{\dn^2(\sigma_j-\sigma_l)}{\cn(\sigma_j-\sigma_l)}\,,&
e_{jl}&=-\frac{\varepsilon k'^2\sn^2(\sigma_j-\sigma_l)}{\cn(\sigma_j-\sigma_l)}\,.
\end{aligned}
\end{equation}
If $\varepsilon_j=-\varepsilon_l=\varepsilon$, then 
\begin{gather}
a_{jl}=\frac{\varepsilon k}{\sn(\sigma_j-\sigma_l)\dn(\sigma_j-\sigma_l)}\,,\qquad
b_{jl}=\frac{k\cn^2(\sigma_j-\sigma_l)}{\sn(\sigma_j-\sigma_l)\dn(\sigma_j-\sigma_l)}\,,\\
e_{jl}=\frac{\varepsilon k'^2}{k\sn(\sigma_j-\sigma_l)\dn(\sigma_j-\sigma_l)}\,.\label{eq_a-e_2-l}
\end{gather}

\subsection{Exotic elliptic EPBQ-curves}\label{subsec_exotic_R}
Let $0<k<1$ be an elliptic modulus. Let $\Gamma_{\alpha}^{ex}(k)$, $\alpha=1,2,3$,  be the curves in $(\RP^1)^3$ parametrized by
\begin{align*}
\Gamma_1^{ex}(k)&\colon&z_1(u)&=\dn u,&z_2(u)&=\dn\left(u-\frac{K}{2}\right),&z_3(u)&=\dn u+\frac{k'}{\dn u}\,,\\
\Gamma_2^{ex}(k)&\colon&z_1(u)&=\dn u,&z_2(u)&=\frac{\cn\left(u-\frac{K}{2}\right)}{\sn\left(u-\frac{K}{2}\right)},&z_3(u)&=\dn u-\frac{k'}{\dn u}\,,\\
\Gamma_3^{ex}(k)&\colon&z_1(u)&=\frac{\cn u}{\sn u},&z_2(u)&=\frac{\cn\left(u-\frac{K}{2}\right)}{\sn\left(u-\frac{K}{2}\right)},&z_3(u)&=\frac{\cn u}{\sn u}-\frac{k'\sn u}{\cn u}\,,
\end{align*}
where $u$ runs over~$\R$.
These curves will be called \textit{exotic EPBQ-curves\/} of the \textit{first kind\/}, of the \textit{second kind\/}, and of the \textit{third kind\/} respectively. The corresponding matrices of coefficients are given by
\begin{gather}\label{eq_a-e_ex-f}
\CA=\frac{\varepsilon_2}{\sqrt{k'}(1+\varepsilon_1 k')}
\begin{pmatrix}
*&1&0\\1&*&1\\0&1&*
\end{pmatrix}\qquad\quad
\CB=
\begin{pmatrix}
*&\frac{\sqrt{k'}}{1+\varepsilon_1 k'}&2\\\frac{\varepsilon_1\sqrt{k'}}{1+\varepsilon_1 k'}&*&0\\0&\frac{\sqrt{k'}}{1+\varepsilon_1k'}&*
\end{pmatrix}\\
\label{eq_a-e_ex-l}
\CE=\varepsilon_1\varepsilon_2
\begin{pmatrix}
*&\frac{k'\sqrt{k'}}{1+\varepsilon_1k'}&2k'\\ \frac{k'\sqrt{k'}}{1+\varepsilon_1k'}&*&0\\2k'\vphantom{\frac{k'\sqrt{k'}}{1+k'}}&0&*
\end{pmatrix}
\end{gather}
where $(\varepsilon_1,\varepsilon_2)=(1,1)$ if $\alpha=1$, $(\varepsilon_1,\varepsilon_2)=(-1,1)$ if $\alpha=2$, and $(\varepsilon_1,\varepsilon_2)=(1,-1)$ if $\alpha=3$.

\subsection{Classification}

\begin{theorem}\label{theorem_sol_R}
All curves constructed in Subsections~\ref{subsec_rational_R}--\ref{subsec_exotic_R} are real EPBQ-curves. Any  realisable real EPBQ-curve is equivalent \textnormal{(}over~$\R$\textnormal{)} to one of the curves constructed in Subsections~\ref{subsec_rational_R}--\ref{subsec_exotic_R}.
\end{theorem}

\begin{proof}
It is easy to see that, for each of the curves~$\Gamma$ constructed in Subsections~\ref{subsec_rational_R}--\ref{subsec_exotic_R}, the complexification~$\Gamma_{\C}$ is equivalent over~$\C$ to one of the complex EPBQ-curves constructed in Subsections~\ref{subsec_elliptic}--\ref{subsec_exotic}. Hence all such curves~$\Gamma$ are real EPBQ-curves. Formulae~\eqref{eq_a-e_1-f}--\eqref{eq_a-e_ex-l} for the coefficients $a_{jl}$, $b_{jl}$, and~$e_{jl}$ are obtained from the corresponding formulae~\eqref{eq_A-E_ell}, \eqref{eq_C_ABE1},  \eqref{eq_C_ABE2} in the complex case by the standard formulae for the change of the variable by quarter- and half-periods, and the transformation formulae for Jacobi's elliptic functions, see~\cite[pp.~350, 369]{BaEr55}.

Now, let $\Gamma\subset(\RP^1)^m$, $m\ge 2$, be an arbitrary realisable real EPBQ-curve, and let $\Gamma_{\C}\subset(\CP^1)^m$
be its complexification. Then the coefficients~$a_{jl}$, $b_{jl}$, and~$e_{jl}$ are real. Moreover, since $\Gamma$ is realisable, Corollary~\ref{cor_ineq_main_abe}
implies that inequalities~\eqref{eq_ineq_cond4} are satisfied. Therefore, by Theorem~\ref{theorem_sol_C}, the curve~$\Gamma_{\C}$ is equivalent over~$\C$ to one of the EPBQ-curves in Subsections~\ref{subsec_elliptic}--\ref{subsec_exotic}.  

First, suppose that $\Gamma_{\C}$ is equivalent to a rational EPBQ-curve. Then $\Gamma_{\C}$ admits a parametrization $z_j(x)=\nu_j\left(x+\frac{\mu_j}{x}\right)^{\pm 1}$, $j=1,\ldots,m$, where $\nu_j$ are non-zero complex numbers, $\mu_j$ are complex numbers such that $\mu_j\ne\pm\mu_l$ unless $j=l$, and $x$ runs over~$\C$. Replacing~$z_j$ by~$z_j^{-1}$, we replace the given curve~$\Gamma$ by an equivalent over~$\R$ EPBQ-curve. Therefore we may assume that the parametrization of~$\Gamma_{\C}$ is given by $z_j(x)=\nu_j\left(x+\frac{\mu_j}{x}\right)$, $j=1,\ldots,m$. Assume that one of the numbers~$\mu_j$ is equal to~$0$, say $\mu_1=0$. (If none of the numbers~$\mu_j$ is equal to~$0$, then $\Gamma_{\C}$ is equivalent to an EPBQ-curve that admits parametrization~\eqref{eq_ell1} with the elliptic modulus $k=1$. This case will be considered when we consider the general elliptic case.)  Let $\varphi$ be the argument of~$\nu_1$. The number~$z_1(x)$ is real if and only if the argument of $x$ belongs to $-\varphi+\pi\Z$. Since the curve $\Gamma_{\C}$ contains infinitely many real points, we see that the numbers $\tilde\nu_j=e^{-i\varphi}\nu_j$ and $\tilde\mu_j=e^{2i\varphi}\mu_j$ are real for $j=1,\ldots,m$. Now a reparametrization $\tilde{x}=e^{i\varphi}x$ turns the coefficients~$\nu_j$ and~$\mu_j$ to the real coefficients $\tilde\nu_j$ and $\tilde\mu_j$. Thus the curve $\Gamma$ is equivalent over~$\R$ to a rational EPBQ-curve~$\Gamma^{rat}(\tilde\bmu)$.

Second, suppose that $\Gamma_{\C}$ is equivalent to an EPBQ-curve in Subsection~\ref{subsec_elliptic}  (possibly with $k=1$). As in the previous case, we may assume that $\Gamma_{\C}$ has a parametrization of the form $z_j(u)=\nu_j\dn(u-\sigma_j,k)$, $j=1,\ldots,m$, for some complex coefficients $k$, $\nu_j$, and~$\sigma_j$. With some abuse of notation, we further omit~$k$ in notation for elliptic functions. Formulae~\eqref{eq_A-E_ell} imply that, for the curve~$\Gamma_{\C}$, we have
\begin{equation}\label{eq_abe_nu_proof}
a_{jl}=\frac{\sn^2(\sigma_j-\sigma_l)}{\nu_j\nu_l\dn(\sigma_j-\sigma_l)}\,,\ \ 
b_{jl}=\frac{\nu_l\cn^2(\sigma_j-\sigma_l)}{\nu_j\dn(\sigma_j-\sigma_l)}\,,\ \ 
e_{jl}=\frac{\nu_j\nu_lk'^2\sn^2(\sigma_j-\sigma_l)}{\dn(\sigma_j-\sigma_l)}\,.
\end{equation} 
This easily implies that
\begin{gather*}
a_{jl}b_{jl}b_{lj}e_{jl}=\frac{k'^2\sn^4(\sigma_j-\sigma_l)\cn^4(\sigma_j-\sigma_l)}{\dn^4(\sigma_j-\sigma_l)}\,,\\
1-a_{jl}e_{jl}-b_{jl}b_{lj}=\frac{(1+k'^2)\sn^2(\sigma_j-\sigma_l)\cn^2(\sigma_j-\sigma_l)}{\dn^2(\sigma_j-\sigma_l)}\,.
\end{gather*}
Hence
\begin{equation}\label{eq_k'}
k'^2+\frac{1}{k'^2}=\frac{(1-a_{jl}e_{jl}-b_{jl}b_{lj})^2-2a_{jl}b_{jl}b_{lj}e_{jl}}{a_{jl}b_{jl}b_{lj}e_{jl}}\,.
\end{equation}
Since the coefficients $a_{jl}$, $b_{jl}$, and~$e_{jl}$ are real and satisfy inequality~\eqref{eq_ineq_main_abe}, we obtain that $k'^2+\frac{1}{k'^2}$  belongs to $(-\infty,-2]\cup(2,+\infty)\cup\{\infty\}$.  Therefore $k'^2$ is real, and $k'^2\ne 1$. Hence $k^2$ is real and $k\ne 0$. Besides, the sign of~$k'^2$ coincides with the sign of~$a_{jl}b_{jl}b_{lj}e_{jl}$.

The case $k^2\in(-\infty,0)\cup(2,+\infty)$ can be reduced to the case $k^2\in(0,2)$ by passing to a new elliptic modulus~$\tilde{k}=\frac{ik}{k'}$. Indeed, it is easy to see that $\tilde{k}^2\in(0,2)$ whenever $k^2\in(-\infty,0)\cup(2,+\infty)$. The standard transformation formula yields
$\dn(u,k)=\dn^{-1}(k'u,\tilde{k})$. Hence $\Gamma_{\C}$ is equivalent to an elliptic EPBQ-curve with the new elliptic modulus $\tilde{k}$ and new phases $k'\sigma_1,\ldots,k'\sigma_m$. 
Thus we may always assume that $k^2\in (0,2]$. Since the sign of~$k$ is irrelevant, we  assume that $k\in\bigl(0,\sqrt{2}\,\bigr]$.

Let $L\subset\C$ be the subset consisting of all~$u$ such that $z_j(u)\in\R\cup\{\infty\}$ for all~$j$.
Since the real part~$\Gamma$ of~$\Gamma_{\C}$ contains infinitely many points, the set~$L$ must contain infinitely many points modulo the lattice $K\Z+iK'\Z$. It follows from~\eqref{eq_abe_nu_proof} that
\begin{equation}\label{eq_nuj}
\nu_j^4=\frac{b_{lj}e_{jl}}{k'^2a_{jl}b_{jl}}\,.
\end{equation}
Since the sign of~$k'^2$ coincides with the sign of~$a_{jl}b_{jl}b_{lj}e_{jl}$, we obtain that $\nu_j^4>0$, $j=1,\ldots,m$. Hence, every~$\nu_j$ is either real or purely imaginary. 
Therefore, the set $L$ consists of all~$u\in\C$ such that $\dn(u-\sigma_j)\in\R\cup\{\infty\}$ whenever $\nu_j\in\R$ and $\dn(u-\sigma_j)\in i\R\cup\{\infty\}$ whenever $\nu_j\in i\R$.
Consider $2$ cases:

\textbf{Case 1: $\boldsymbol{0<k\le 1}$.} Then $K$ and~$K'$ are real (except for $K=\infty$ if $k=1$).
The function~$\dn u$ takes real values on the lines $\Re u=qK$ and  $\Im u=2qK'$, $q\in\Z$, and takes purely imaginary values on the lines $\Im u=(2q+1)K'$, $q\in\Z$, see~\cite[pp.~351, 352]{BaEr55}. We denote the union of all lines $\Re u=qK$ and $\Im u=qK'$, $q\in \Z$, by~$X$. Then, for each $j$, the set $L$ is contained in the set~$X_j=X+\sigma_j$.  Therefore, the intersection of the sets~$X_j$, $j=1,\ldots,m$, must contain a line. 
Consequently, either $\Re(\sigma_j-\sigma_l)\in K\Z$ for all $j$ and~$l$ or  $\Im(\sigma_j-\sigma_l)\in K'\Z$ for all $j$ and~$l$.

Suppose that $\Re(\sigma_j-\sigma_l)\in K\Z$ for all~$j$ and~$l$. A simultaneous shift of all phases~$\sigma_j$ does not change the curve~$\Gamma_{\C}$. Hence we may assume that $\sigma_1=0$. Then $\Re\sigma_j\in K\Z$ for all~$j$, and $L$ is  the union of the vertical lines $\Re u=qK$, $q\in\Z$. Since the function~$\dn$ is real on these vertical lines, we obtain that all~$\nu_j$ are real. Put $s=\sn\sigma_2$, $c=\cn\sigma_2$, and $d=\dn\sigma_2$. If $\Re\sigma_2=2qK$, $q\in\Z$, then $s$ is purely imaginary, i.\,e., $s^2<0$. Since $c^2=1-s^2$ and $d^2=1-k^2s^2$, we obtain that $c^2>d^2>1$, hence, $c^2>|d|$. Therefore the numbers $a_{12}$, $-b_{12}$, $-b_{21}$, and~$e_{12}$ are negative, and $b_{12}b_{21}> 1$. If $\Re\sigma_p=(2q+1)K$, $q\in\Z$, then $c$ is purely imaginary, i.\,e., $c^2<0$. Hence $s^2>1$ and $|d|<k'$. Therefore $a_{12}$, $-b_{12}$, $-b_{21}$, and~$e_{12}$ are positive, and $a_{12}e_{12}>1$. By Lemma~\ref{lem_realise}, in both cases the curve~$\Gamma$ is not realisable, which contradicts our assumption.

Consequently   $\Im(\sigma_j-\sigma_l)\in K'\Z$ for all~$j$. Since $\dn u$ is real on the lines $\Im u=2qK'$ and is purely imaginary on the lines $\Im u=(2q+1)K'$ we obtain that $L$ is a union of infinitely many horizontal lines with the distance $2K'$ between the consecutive lines. By a common shift of all~$\sigma_j$, we may achieve that $L$ is the union of the lines $\Im u=2qK'$, $q\in\Z$. Then $\Im\sigma_j\in K'\Z$ for all~$j$. Besides,
 $\nu_j\in\R$ whenever $\Im\sigma_j\in 2K'\Z$ and $\nu_j\in i\R$ whenever $\Im\sigma_j\in K'+2K'\Z$.
The function~$\dn$ has periods~$2K$ and~$4iK'$. Besides, $\dn(u+2iK')=-\dn u$.  Hence, adding $2qK+2q'K'$ to any~$\sigma_j$  and simultaneously multiplying~$\nu_j$ by~$(-1)^{q'}$, we do not change the curve~$\Gamma_{\C}$.  Thus, we may assume  that $0\le \Re\sigma_j<2K$ and $\Im\sigma_j$ is either~$0$ or~$K'$ for all~$j$. Permuting the coordinates~$z_j$, we may achieve that $\Im\sigma_j=0$ whenever $j\le m'$ and $\Im\sigma_j=K'$ whenever $j> m'$ for certain $m'$ such that $0\le m'\le m$.  Then $\nu_j\in\R$ whenever $j\le m'$ and $\nu_j\in i\R$ whenever $j>m'$. For $j>m'$, by
 the standard formula $\dn (v-iK')=\frac{i\cn v}{\sn v}$,
we obtain that 
$
z_j(u)=\tilde\nu_j\,\frac{\cn(u-\tilde\sigma_j)}{\sn(u-\tilde\sigma_j)}
$
for the real coefficients $\tilde\sigma_j=\sigma_j-iK'$, $\tilde\nu_j=i\nu_j$. Thus  $\Gamma$ is equivalent over~$\R$ to the curve~$\Gamma^{ell}_{1,m'}$ corresponding to the elliptic modulus~$k$ and the phases $\sigma_1,\ldots,\sigma_{m'}$, $\tilde\sigma_{m'+1},\ldots,\tilde\sigma_m$.  If $k=1$, then this curve degenerates to an EPBQ-curve equivalent over~$\R$ to the curve~$\Gamma^{rat}(\bmu)$, where $\mu_j=e^{2\sigma_j}$ whenever $j\le m'$ and $\mu_j=-e^{2\tilde\sigma_j}$ whenever $j> m'$.

\textbf{Case 2: $\boldsymbol{1<k\le\sqrt{2}}$.}  Then $K'$ and $K+iK'$ are real, hence, $\Re K=K+iK'$. Besides, $k'=i\sqrt{k^2-1}$ is purely imaginary.
Then the function $\dn$ takes real values on the lines $\Re u=2q\Re K$ and $\Im u=2qK'$, $q\in\Z$, and takes purely imaginary values on the lines $\Re u=(2q+1)\Re K$ and $\Im u=(2q+1)K'$, $q\in\Z$. As in the previous case, we obtain that either $\Re(\sigma_j-\sigma_l)\in (\Re K)\Z$ for all $j$ and~$l$ or $\Im(\sigma_j-\sigma_l)\in K'\Z$ for all $j$ and~$l$. Hence $L$ is either the union of infinitely many horizontal lines with the distance $2K'$ between the consecutive lines or the union of infinitely many vertical lines with the distance $2\Re K$ between the consecutive lines.  By a common shift of all~$\sigma_j$ we achieve that  $L$ is either the union of the lines $\Im u=2qK'$, $q\in\Z$, or the union of the lines $\Re u=2q\Re K$, $q\in\Z$. In the first case, we put
$$
\tilde u=ku,\qquad \tilde{k}=\frac{1}{k}\,,\qquad \tilde\sigma_j=k\sigma_j,\qquad \tilde\nu_j=\nu_j.
$$
In the second case, we put 
$$
\tilde u=-iku,\qquad \tilde{k}=\frac{k'}{ik}\,,\qquad \tilde\sigma_j=-ik(\sigma_j+K),\qquad \tilde\nu_j=k'\nu_j.
$$
Using the standard transformation formulae for elliptic functions,
we obtain that in both cases the curve~$\Gamma_{\C}$ is parametrized by
\begin{equation}\label{eq_ell3}
z_j(\tilde{u})=\tilde\nu_j\cn(\tilde{u}-\tilde\sigma_j,\tilde k),\qquad j=1,\ldots,m.
\end{equation}
The new elliptic modulus~$\tilde{k}$ satisfies $0<\tilde{k}<1$, and the corresponding quarter-periods are given by $\widetilde{K}=k\Re K$ and $\widetilde{K}'=kK'$ in the first case, and $\widetilde{K}=kK'$ and $\widetilde{K}'=k\Re K$ in the second case. Therefore in both cases the set $\widetilde{L}$ consisting of all~$\tilde{u}$ such that $z_j(\tilde{u})\in\R\cup\{\infty\}$ for all~$j$ is the union of lines $\Im\tilde u=2q\widetilde{K}'$, $q\in\Z$. Hence $\Im\tilde\sigma_j\in \widetilde{K}'\Z$ for all~$j$. As in the previous case, we may assume that $\Im\tilde\sigma_j=0$ whenever $j\le m'$ and $\Im\tilde\sigma_j=\widetilde{K}'$ whenever $j> m'$ for certain $m'$ such that $0\le m'\le m$.  Then $\tilde\nu_j\in\R$ whenever $j\le m'$ and $\tilde\nu_j\in i\R$ whenever $j>m'$. For $j>m'$, we obtain that 
$$
z_j(\tilde{u})=\hat\nu_j\,\frac{\dn(u-\hat\sigma_j,\tilde{k})}{\sn(u-\hat\sigma_j,\tilde{k})}
$$
for the real coefficients $\hat\sigma_j=\tilde\sigma_j-i\widetilde{K}'$, $\hat\nu_j=-i\tilde\nu_j/\tilde{k}$. Thus  $\Gamma$ is equivalent over~$\R$ to the curve~$\Gamma^{ell}_{2,m'}$ corresponding to the elliptic modulus~$\tilde{k}$ and the phases $\tilde\sigma_1,\ldots,\tilde\sigma_{m'}$, $\hat\sigma_{m'+1},\ldots,\hat\sigma_m$.  

Finally, suppose that $m=3$, and $\Gamma_{\C}$ is equivalent to one of the exotic elliptic curves in Subsection~\ref{subsec_exotic}. As in the previous cases, we may assume that $\Gamma_{\C}$ has a parametrization of the form 
$$
z_1(u)=\nu_1\dn (u,k),\quad  z_2(u)=\nu_2\dn (u-\sigma,k),\quad z_3(u)=
\nu_3\left(\dn(u,k)+\frac{\varepsilon_1k'}{\dn(u,k)}\right),
$$
where $k\ne1$ and $\nu_1,\nu_2,\nu_3\ne 0$ are complex numbers, $\sigma$ is either~$\frac{K}2$ or $\frac{K}{2}+iK'$, and $\varepsilon_1=1$ if $\sigma=\frac{K}2$ and $\varepsilon_1=-1$ if $\sigma=\frac{K}2+iK'$. Formulae~\eqref{eq_k'} and~\eqref{eq_nuj} hold for the pairs $(j,l)=(1,2)$ and~$(2,1)$. Hence we may   again obtain that $0<k\le\sqrt{2}$, $\nu_1$ and~$\nu_2$ belong to $\R\cup i\R$, and either $\Re\sigma\in (\Re K)\Z$ or $\Im\sigma\in K'\Z$. But $\Re\sigma=\Re K/2$ does not belong to $(\Re K)\Z$. Hence $\Im\sigma\in K'\Z$. If $k>1$, then $\Im\frac{K}{2}=-\frac{K'}{2}\notin K'\Z$. Therefore $0<k<1$.
The subset~$L\subset\C$ consisting of all~$u$ such that $z_1(u)$, $z_2(u)$, and~$z_3(u)$ belong to $\R\cup\{\infty\}$ is either the union of the lines $\Im\sigma=2qK'$, $q\in\Z$, or the union of the lines $\Im\sigma=(2q+1)K'$, $q\in\Z$.

If $L$ is the union of the lines $\Im\sigma=2qK'$, $q\in\Z$, and $\sigma=\frac{K}{2}$, then $\nu_1$, $\nu_2$, and~$\nu_3$ are real. Hence $\Gamma$ is equivalent to~$\Gamma^{ex}_1(k)$. 

If $L$ is the union of the lines $\Im\sigma=2qK'$, $q\in\Z$, and $\sigma=\frac{K}{2}+iK'$, then $\nu_1$ and~$\nu_3$ are real, and $\nu_2$ is purely imaginary. Hence $\Gamma$ is equivalent to~$\Gamma^{ex}_2(k)$. 

If $L$ is the union of the lines $\Im\sigma=(2q+1)K'$, $q\in\Z$, and $\sigma=\frac{K}{2}$, then $\nu_1$, $\nu_2$, and~$\nu_3$ are purely imaginary. After the reparametrization $\tilde u=u+iK'$, we easily obtain that $\Gamma$ is equivalent to~$\Gamma^{ex}_3(k)$. 

Suppose that $L$ is the union of the lines $\Im\sigma=(2q+1)K'$, $q\in\Z$, and $\sigma=\frac{K}{2}+iK'$. Then $\nu_2$ is real and $\nu_1$ and~$\nu_3$ are purely imaginary. Putting $\tilde{u}=u-\frac{K}{2}-iK'$, we obtain the following parametrization for~$\Gamma$:
\begin{gather*}
z_1(\tilde u)=\frac{\tilde\nu_1\sn \bigl(\tilde u-\frac{K}{2},k\bigr)}
{\cn \bigl(\tilde u-\frac{K}{2},k\bigr)}\,,\qquad   z_2(u)=\nu_2\dn (\tilde u,k),\\  z_3(\tilde u)=
\tilde\nu_3\left(\dn(\tilde u,k)-\frac{k'}{\dn(\tilde u,k)}\right)^{-1},
\end{gather*}
where $\tilde\nu_1=ik'\nu_1$ and~$\tilde\nu_3=-2i\sqrt{k'}(1-k')\nu_3$. Inverting~$z_1$ and~$z_3$ and then permuting the coordinates~$z_1$ and~$z_2$, we obtain that $\Gamma$ is equivalent to~$\Gamma^{ex}_2(k)$.
\end{proof}

\section{Classification of flexible cross-polytopes}\label{section_classify}

Our aim is to classify all irreducible essential flexes of cross-polytopes, i.\,e., all pairs $(\ell,\Sigma)$ such that  $\Sigma$ is an irreducible component of~$\Xi^{ess}(\ell)$, see Section~\ref{section_Bri}. In this classification, flexible cross-polytopes obtained from each other by renumbering their vertices are regarded as distinct.

\begin{theorem}[Classification Theorem]\label{theorem_classify}
Suppose that $\X^n=\E^n$, $\bS^n$, or~$\Lambda^n$, $n\ge 3$.  Then in~$\X^n$ there exist the following families of essential flexible cross-polytopes (more precisely, of irreducible essential flexes of cross-polytopes):
\begin{enumerate}
\item An $\left(\frac{n(n+1)}{2}-1\right)$-parametric family~$P^{simple}(\blambda,G)$ of  flexible cross-polytopes of the simplest type  described in Section~\ref{section_simplest}.
\item For each decomposition $\I\colon [n]=I_1\sqcup\ldots\sqcup I_m$, where $m\ge 2$ and $n_j=|I_j|\ge 1$\textnormal{:}
\begin{itemize}
\item An $\left(m+n-1+\sum_{j=1}^m{n_j\choose 2}\right)$-parametric family $P^{rat}_{\I}(\bmu,\blambda,\mathbf{g})$ of flexible cross-polytopes  obtained by applying Construction~\ref{constr} to the rational EPBQ-curves $\Gamma^{rat}(\bmu)$. 
\item $\left(m+n+\sum_{j=1}^m{n_j\choose 2}\right)$-parametric families $P^{ell}_{\I,\alpha,m'}(k,\bsigma,\blambda,\mathbf{g})$ of flexible cross-polytopes obtained by applying Construction~\ref{constr} to the elliptic EPBQ-curves $\Gamma^{ell}_{\alpha,m'}(k,\bsigma)$ for all pairs $(\alpha,m')$,  $\alpha=1,2$, $0\le m'\le m$.
\end{itemize}
\item If~$\X^n=\bS^n$, then, in addition, for each decomposition $\I\colon [n]=I_1\sqcup I_2\sqcup I_3$, where $n_j=|I_j|\ge 1$,  there exist three $\left(4+{n_1\choose 2}+{n_2\choose 2}+{n_3\choose 2}\right)$-parametric families~$P^{ex}_{\I,\alpha}(k,\blambda,\mathbf{g})$ of flexible cross-polytopes  obtained by applying Construction~\ref{constr} to the exotic EPBQ-curves~$\Gamma_{\alpha}^{ex}(k)$, $\alpha=1,2,3$.
\end{enumerate}
Each of the  families described is non-empty. The families described exhaust all flexible cross-polytopes in~$\X^n$. In particular, there exist no flexible cross-polytopes corresponding to exotic EPBQ-curves in~$\E^n$ and~$\Lambda^n$, $n\ge 3$.
\end{theorem}

\begin{remark}
In this theorem we mean that the specified numbers of parameters for the families of flexible cross-polytopes cannot be decreased. 
\end{remark}

\begin{cor}
For each partition $(n_1,\ldots,n_m)$ of a number $n\ge 3$, there exist flexible cross-polytopes of type~$(n_1,\ldots,n_m)$ in each of the spaces~$\E^n$, $\bS^n$, and~$\Lambda^n$.
\end{cor}

\begin{proof}[Proof of Theorem~\ref{theorem_classify}]
Proposition~\ref{propos_constr} and Theorem~\ref{theorem_sol_R} imply that all flexible cross-polytopes in spaces~$\E^n$, $\bS^n$, and $\Lambda^n$ can be obtained by applying Construction~\ref{constr} either to the simplest EPBQ-curve~$\RP^1$ or to the families of EPBQ-curves $\Gamma^{rat}(\bmu)$, $\Gamma^{ell}_{\alpha,m'}(k,\bsigma)$, $\alpha=1,2$, $\Gamma^{ex}_{\alpha}(k)$, $\alpha=1,2,3$, and to all decompositions~$\I$.
In each of these cases, we shall denote by~$\btheta$ the set of all continuous parameters on which the flexible cross-polytope obtained depends. So  $\btheta=(\blambda,G)$ for the simplest family of flexible cross-polytopes, $\btheta=(\bmu,\blambda,\mathbf{g})$ for  rational families, $\btheta=(k,\bsigma,\blambda,\mathbf{g})$ for  elliptic families, and $\btheta=(k,\blambda,\mathbf{g})$ for exotic families. Then $\btheta\in\R^{\rho}$, where 
\begin{equation}\label{eq_rho}
\rho=\left\{
\begin{aligned}
&\textstyle n(n+1)/2&&\text{for the simplest family,}\\
&\textstyle m+n+\sum_{j=1}^n{n_j\choose2}&&\text{for rational families,}\\
&\textstyle 1+m+n+\sum_{j=1}^n{n_j\choose2}&&\text{for elliptic families,}\\
&\textstyle4+\sum_{j=1}^3{n_j\choose2}&&\text{for exotic families.}
\end{aligned}
\right.
\end{equation}
We denote by $\Theta^{simple}(\X^n)$, $\Theta^{rat}_{\I}(\X^n)$, $\Theta^{ell}_{\I,\alpha,m'}(\X^n)$, $\Theta^{ex}_{\I,\alpha}(\X^n)$ the subsets of the corresponding spaces~$\R^{\rho}$ consisting of all~$\btheta$ such that the corresponding pair~$(G(\btheta),H(\btheta))$ belongs to~$\Psi(\X^n)$, i.\,e., the corresponding flexible cross-polytope in~$\X^n$ is well defined. The entries of the matrices~$G(\btheta)=(g_{pq})$ and $H(\btheta)=(h_{pq})$ for $j(p)\ne j(q)$ are obtained by substituting expressions~\eqref{eq_a-e_1-f}--\eqref{eq_a-e_ex-l} for the coefficients~$a_{jl}$, $b_{jl}$, and~$e_{jl}$ to formulae~\eqref{eq_g_ne} and~\eqref{eq_h_ne}. For $j(p)=j(q)$ and $p\ne q$, the matrix entries~$g_{pq}$ enter the set of parameters~$\mathbf{g}$, and the matrix entries~$h_{pq}$  are computed by~\eqref{eq_h_e}. The subsets  $\Theta^{simple}(\X^n)$, $\Theta^{rat}_{\I}(\X^n)$, $\Theta^{ell}_{\I,\alpha,m'}(\X^n)$, $\Theta^{ex}_{\I,\alpha}(\X^n)$ are semi-analytic sets given by the inequalities that are obtained by substituting the explicit expressions for~$g_{pq}$ and~$h_{pq}$ to the algebraic inequalities describing the subsets~$\Psi(\X^n)$. Thus we obtain explicitly the system of inequalities that gives each of the subsets $\Theta^{simple}(\X^n)$, $\Theta^{rat}_{\I}(\X^n)$, $\Theta^{ell}_{\I,\alpha,m'}(\X^n)$, $\Theta^{ex}_{\I,\alpha}(\X^n)$. Unfortunately, these inequalities do not allow us to obtain a reasonable description of the geometry and the topology of these sets. 

We formulate two lemmas whose proofs will be postponed to Section~\ref{section_exist}.

\begin{lem}\label{lem_exist}
Each of the sets\/~$\Theta^{simple}(\E^n)$, $\Theta^{rat}_{\I}(\E^n)$ and\/~$\Theta^{ell}_{\I,\alpha,m'}(\E^n)$ contains a point\/~$\btheta^0$ such that the gradient\/ $\nabla\det G(\btheta)$ is non-zero at\/ $\btheta^0$.
\end{lem}

\begin{lem}\label{lem_exotic_empty}
All sets $\Theta^{ex}_{\I,\alpha}(\bS^n)$ are non-empty, and all sets $\Theta^{ex}_{\I,\alpha}(\E^n)$ and $\Theta^{ex}_{\I,\alpha}(\Lambda^n)$ are empty.
\end{lem}

Lemma~\ref{lem_exist} implies that all sets $\Theta^{simple}(\X^n)$, $\Theta^{rat}_{\I}(\X^n)$, $\Theta^{ell}_{\I,\alpha,m'}(\X^n)$ are non-empty for all three spaces~$\X^n$. Indeed, by Corollary~\ref{cor_but}, the set $\Psi(n)=\Psi(\E^n)\sqcup\Psi(\bS^n)\sqcup\Psi(\Lambda^n)$ is open. Hence, for all~$\btheta$ in some neighborhood of~$\btheta^0$, the pair $(G(\btheta),H(\btheta))$ belongs to~$\Psi(n)$.  Besides, since $\det G(\btheta^0)=0$ and $\nabla\det G(\btheta)|_{\btheta=\btheta^0}\ne0$, we see that any neighborhood of~$\btheta^0$ contains both points~$\btheta$ with positive~$\det G(\btheta)$ and points~$\btheta$ with negative~$\det G(\btheta)$, i.\,e., points~$\btheta$ such that the pairs $(G(\btheta),H(\btheta))$ belong to either of the sets~$\Psi(\bS^n)$ and~$\Psi(\Lambda^n)$.

Thus, Lemmas~\ref{lem_exist} and~\ref{lem_exotic_empty} imply that all families of flexible cross-polytopes listed in Theorem~\ref{theorem_classify} are non-empty, and there are no other flexible cross-polytopes. It remains to calculate the dimensions of all these families of flexible cross-polytopes. Let $\Theta(\X^n)$ be one of the sets $\Theta^{simple}(\X^n)$, $\Theta^{rat}_{\I}(\X^n)$, $\Theta^{ell}_{\I,\alpha,m'}(\X^n)$, $\Theta^{ex}_{\I,\alpha}(\bS^n)$, and let $\rho$ be the dimension of the corresponding space of parameters~$\btheta$ given by~\eqref{eq_rho}. The EPBQ-curve~$\Gamma$ corresponding to a flexible cross-polytope is defined uniquely up to equivalence. It is not hard to see that, once a representative~$\Gamma$ of this equivalence class is chosen, the pair~$(\blambda,\mathbf{g})$ can take only finitely many values for each flexible cross-polytope. Therefore the set of parameters~$\btheta$ corresponding to the given flexible cross-polytope can take finitely many values for exotic families, finitely many values modulo transformations $\lambda_p\to y\lambda_p$ for the simplest family, finitely many values modulo transformations $\lambda_p\to y\lambda_p$, $\mu_j\to y^{-2}\mu_j$ for rational families, and finitely many values modulo transformations  $\sigma_j\to \sigma_j+y$ for elliptic families. (Here in each case we mean that the number~$y$ is the same for all~$p$ and~$j$. Besides, in the elliptic case the numbers~$\sigma_j$ are regarded as elements of~$\R/(4K\Z)$.)

If $\X^n$ is either $\bS^n$ or~$\Lambda^n$, then the corresponding set~$\Theta(\X^n)\subset\R^{\rho}$ is open, since it is given by strict inequalities. Hence the number of parameters on which the flexible cross-polytope in~$\X^n$ depends is equal to~$\rho-1$ for all families except for the exotic families, and is equal to~$\rho$ for the exotic families, which yields exactly the required numbers of parameters.

The set $\Theta(\E^n)\subset\R^{\rho}$ is an open semi-algebraic subset of the closed affine variety given by  $\det G(\btheta)=0$. It follows from Lemma~\ref{lem_exist} that  $\dim\Theta(\E^n)=\rho-1$. One degree of freedom is again killed by the one-parametric transformation groups described above. However, in the Euclidean case we have an additional parameter, namely, the scale, since the triple~$(\Gamma,\blambda,\mathbf{g})$ determines the flexible cross-polytope up to similarity, while we classify flexible cross-polytopes up to isometry. Thus we again obtain the $(\rho-1)$-parametric family of flexible cross-polytopes.
\end{proof}

\begin{remark}
Since the EPBQ-curve corresponding to an irreducible essential flex of a cross-polytope is defined uniquely up to equivalence, it follows that the different families of flexible cross-polytopes described in Theorem~\ref{theorem_classify} do not intersect, except for the fact that all elliptic families degenerate to the corresponding rational families for $k=1$. 
\end{remark}

\begin{remark} We have four kinds of flexible cross-polytopes, simplest, rational, elliptic, and exotic.
It is natural to ask whether the kind and the type of a flexible cross-polytope depend on which of its facets is taken for the butterfly's body~$\Delta$. Independence of the kind and the type of the choice of~$\Delta$ can be proved in the following way. Notice that the kind and the type of a flexible cross-polytope are completely determined by the set~$\mathcal{R}$ of all ordered pairs of indices $(p,q)$ such that $t_p$ is a rational function in~$t_q$. The form of relations~\eqref{eq_main_rel} easily yields that $t_p^2$ is a rational function in~$t_q^2$ if and only if $t_p$ is a rational function in~$t_q$. We denote by $d_p$ the \textit{square\/} of the length of the diagonal~$[\ba_p\bb_p]$ if $\X^n=\E^n$, the \textit{cosine\/} of the length of the diagonal~$[\ba_p\bb_p]$ if $\X^n=\bS^n$, and the \textit{hyperbolic cosine\/} of the length of the diagonal~$[\ba_p\bb_p]$ if $\X^n=\Lambda^n$. It is easy to see that in each case~$d_p$ is a linear fractional  function in~$t_p^2$. Hence, $(p,q)\in\mathcal{R}$ if and only if $d_p$ is a rational function in~$d_q$. Since this description of the set~$\mathcal{R}$ is independent of the choice of~$\Delta$, we obtain that the kind and the type of a flexible cross-polytope are independent of the choice of~$\Delta$. Moreover, for elliptic families of flexible cross-polytope, we can immediately show that the parameter~$\kappa$ and, hence, the elliptic modulus~$k$ are independent of the choice of~$\Delta$.  Using expressions~\eqref{eq_A'B'C'E'}
for the coefficients of relations~\eqref{eq_main_Bri_rel}, we can easily rewrite~\eqref{eq_kappa} as
$$
\kappa-2=\textstyle\frac{\strut\sin^2\alpha\,\sin^2\beta\,\sin^2\gamma\,\sin^2\delta}
{\strut\sin\frac{\alpha+\beta+\gamma+\delta}{2}
\sin\frac{\alpha+\beta+\gamma-\delta}{2}
\sin\frac{\alpha+\beta-\gamma+\delta}{2}
\sin\frac{\alpha+\beta-\gamma-\delta}{2}
\sin\frac{\alpha-\beta+\gamma+\delta}{2}
\sin\frac{\alpha-\beta+\gamma-\delta}{2}
\sin\frac{\alpha-\beta-\gamma+\delta}{2}
\sin\frac{\alpha-\beta-\gamma-\delta}{2}
}
$$
where  $\alpha=\alpha_{pq}$, $\beta=\beta_{pq}$, $\gamma=\gamma_{pq}$, and $\delta=\beta_{qp}$ for some pair~$(p,q)$ such that $j(p)\ne j(q)$. Since the expression obtained is a symmetric function in variables~$\alpha$, $\beta$, $\gamma$, and~$\delta$, it follows that $\kappa$ will not change if we take for the butterfly's body the facet~$\Delta_p$ instead of~$\Delta$. Thus we easily obtain that $\kappa$ is the same for all facets of the cross-polytope.
\end{remark}

\section{Existence of flexible cross-polytopes}\label{section_exist}

In Subsections~\ref{subsec_exist_rat1}--\ref{subsec_exist_ell} we prove Lemma~\ref{lem_exist}. For~$\Theta^{simple}(\E^n)$, the assertion of this lemma follows immediately from the fact that, for each matrix~$G$ satisfying condition~$(\E1)$, the pairs $(\blambda,G)$ belong to~$\Theta^{simple}(\E^n)$ for all~$\blambda$ in a non-empty Zariski open subset of~$\R^n$,  see Section~\ref{section_simplest}. Therefore we need to prove Lemma~\ref{lem_exist} for the rational and the elliptic families of flexible cross-polytopes only.
In Subsection~\ref{subsec_exist_ex} we prove Lemma~\ref{lem_exotic_empty}.

\subsection{Rational families of types $\boldsymbol{(1,\ldots,1)}$ ($\boldsymbol{n}$ units)}\label{subsec_exist_rat1}

In this case $\btheta=(\bmu,\blambda)$, where $\bmu=(\mu_1,\ldots,\mu_n)$ and $\blambda=(\lambda_1,\ldots,\lambda_n)$. By~\eqref{eq_g_ne}, \eqref{eq_h_ne}, and~\eqref{eq_rat_ABE}, the matrices $G(\bmu,\blambda)$ and~$H(\bmu,\blambda)$ are given by
\begin{align}
\label{eq_g_rat}
g_{pq}&=\frac{\lambda_p^2\mu_p+\lambda_q^2\mu_q-\lambda_p^2\lambda_q^2(\mu_p-\mu_q)^2}{\lambda_p\lambda_q(\mu_p+\mu_q)}\,,\\
\label{eq_h_rat}h_{pq}&=\frac{2\lambda_p\bigl(
\mu_p-\lambda_q^2(\mu_p-\mu_q)^2
\bigr)}{\lambda_q(\mu_p+\mu_q)}\,.
\end{align}

\begin{lem}\label{lem_rat_S}
For each $n\ge 2$ and for each row of signs $(\varepsilon_1,\ldots,\varepsilon_n)\in\{-1,1\}^n$, there exist rows of parameters 
$$\bmu^*=(\mu^*_1,\ldots,\mu^*_n),\qquad \blambda^+=(\lambda^+_1,\ldots,\lambda^+_n),\qquad  \blambda^-=(\lambda^-_1,\ldots,\lambda^-_n)$$ such that $\sign\mu^*_p=\varepsilon_p$, $p=1,\ldots,n$, the matrix $G(\bmu^*,\blambda^+)$ is arbitrarily close to the unit matrix, hence, positive definite, and $\det G(\bmu^*,\blambda^-)<0$. 
\end{lem}
\begin{proof}
Let $0<\delta\ll 1$. We put $\mu_p^*=\varepsilon_p\delta^{2p}$, $\lambda^+_p=\delta^{n-p}$ for $p=1,\ldots,n$. Then  $g_{pq}(\bmu^*,\blambda^+)=O(\delta^{|q-p|})$ as $\delta\to 0$. Therefore, for sufficiently small~$\delta$, the matrix~$G(\bmu^*,\blambda^+)$ is arbitrarily close to the unit matrix.

Now, we put $\lambda_1^-=2\delta^{n-2}$ and $\lambda_p^-=\lambda_p^+=\delta^{n-p}$ for $p=2,\ldots,n$. Then  $g_{pq}(\bmu^*,\blambda^-)\to 0$ as $\delta\to 0$ if $p\ne q$ and $\{p,q\}\ne \{1,2\}$, and $g_{12}(\bmu^*,\blambda^-)\to 2$ as $\delta\to 0$. Therefore, $\det G(\bmu^*,\blambda^-)<0$ for sufficiently small~$\delta$.
\end{proof}

The determinant $\det G(\bmu,\blambda)$ is a rational function in~$\bmu$ and~$\blambda$ with real coefficients. Moreover, it is easy to see that $\det G(\bmu,\blambda)$ is a rational function in $\mu_1,\ldots,\mu_n$ and $\lambda_1^2,\ldots,\lambda_n^2$. We put $\tau_p=\lambda_p^2$, $\btau=(\tau_1,\ldots,\tau_n)$, and $D(\bmu,\btau)=\det G(\bmu,\blambda)$.
The denominator of~$D(\bmu,\btau)$ is a product of terms of the form~$\tau_p$ and~$\mu_p+\mu_q$. Besides,  $D(\bmu,\btau)$ is invariant under the action of the symmetric group $\mathfrak{S}_n$ that permute simultaneously the variables $\mu_1,\ldots,\mu_n$ and the variables $\tau_1,\ldots,\tau_n$. 
Hence there exists a unique up to multiplications by constants decomposition of the form
\begin{equation}\label{eq_D_decomp}
D(\bmu,\btau)=(\tau_1\cdots\tau_n)^{-1}f(\bmu)P_1(\bmu,\btau)\cdots P_l(\bmu,\btau),
\end{equation}
where $f\in\R(\bmu)$ is a rational function, and $P_1,\ldots,P_l\in\R[\bmu,\btau]$ are  irreducible polynomials each of which depends essentially on~$\btau$. (Notice that~$P_j$ are supposed to be irreducible over~$\R$ rather than over~$\C$.) Let~$P$ be the product $P_1\cdots P_l$.

\begin{lem}\label{lem_no_tau}
None of the polynomials~$P_j(\bmu,\btau)$ has the form $c\tau_p$,   $c\in\R$.
\end{lem}

\begin{proof}
The  polynomial $P(\bmu,\btau)$ is  invariant under the action of~$\mathfrak{S}_n$ up to  sign. If $P_j(\bmu,\btau)=c\tau_p$, then $P(\bmu,\btau)$ would be divisible by~$\tau_1\cdots\tau_n$. Therefore $\det G(\bmu,\blambda)$ would be a polynomial in~$\blambda$ whose coefficients are rational functions in~$\bmu$. Then, for any fixed~$\bmu$ such that $\mu_p+\mu_q\ne 0$ for all $p,q$, the determinant $\det G(\bmu,\blambda)$ considered as a function of $\blambda$ would have a well-defined limit as $\blambda\to 0$.
{\sloppy

} 

We take $\bmu^*$ as in the proof of Lemma~\ref{lem_rat_S}, and put $\lambda_p^{\pm}(\gamma)=\gamma\lambda_p^{\pm}$ for $p=1,\ldots,n$. Then for all sufficiently small positive~$\delta$ and for all $\gamma\in(0,1)$,  the matrices $G(\bmu^*,\blambda^+(\gamma))$ are arbitrarily close to the unit matrix, and the matrices $G(\bmu^*,\blambda^-(\gamma))$ are arbitrarily close to the symmetric matrix with units on the diagonal and the only non-zero non-diagonal entries in positions~$(1,2)$ and~$(2,1)$ equal to~$2$. Let $\gamma$ tend to zero, while $\delta$ is constant. Then all $\mu_p^*$ do not change, and all~$\lambda^{\pm}_p(\gamma)$ tend to zero. Hence the determinants of the matrices $G(\bmu^*,\blambda^+(\gamma))$ and $G(\bmu^*,\blambda^-(\gamma))$ should tend to the same limit, which is impossible, since the determinants of these matrices are close to~$1$ and to~$-3$ respectively for all $\gamma\in(0,1)$.
\end{proof}

\begin{lem}\label{lem_P1P2}
For every~$n\ge 2$, the number~$l$ is either~$1$ or~$2$. If $l=1$, then every variable~$\tau_p$ enters non-trivially~$P(\bmu,\btau)$. If $l=2$, then the polynomials $P_1(\bmu,\btau)$ and $P_2(\bmu,\btau)$ are not proportional to each other, and every variable~$\tau_p$ enters non-trivially both~$P_1(\bmu,\btau)$ and $P_2(\bmu,\btau)$.
\end{lem}

\begin{proof}
For $n=2$ and $n=3$ the lemma is proved by a direct computation. For $n=2$, we obtain that $l=1$ and
$$
P_1=(\mu_1+\mu_2)^2\tau_1\tau_2-\bigl(
\mu_1\tau_1+\mu_2\tau_2-(\mu_1-\mu_2)^2\tau_1\tau_2
\bigr)^2.
$$
For $n=3$, we obtain that $l=2$ and
\begin{align*}
&P_1=2(\mu_1-\mu_2)(\mu_2-\mu_3)(\mu_3-\mu_1)\tau_1\tau_2\tau_3+
(\mu_1^2-\mu_2^2)\tau_1\tau_2+
(\mu_2^2-\mu_3^2)\tau_2\tau_3\\
{}&+(\mu_3^2-\mu_1^2)\tau_3\tau_1,\\
&P_2=(\mu_1^2-\mu_2^2)(\mu_2^2-\mu_3^2)(\mu_3^2-\mu_1^2)\tau_1\tau_2\tau_3
+2(\mu_1-\mu_2)(\mu_2+\mu_3)(\mu_3+\mu_1)\mu_1\mu_2\tau_1\tau_2
\\{}&+2(\mu_1+\mu_2)(\mu_2-\mu_3)(\mu_3+\mu_1)\mu_2\mu_3\tau_2\tau_3
+2(\mu_1+\mu_2)(\mu_2+\mu_3)(\mu_3-\mu_1)\mu_3\mu_1\tau_3\tau_1
\\{}&-(\mu_2^2-\mu_3^2)\mu_1^2\tau_1
-(\mu_3^2-\mu_1^2)\mu_2^2\tau_2
-(\mu_1^2-\mu_2^2)\mu_3^2\tau_3.
\end{align*}

Now, suppose that $n\ge 4$. For $j=1,\ldots,l$, let~$e_j$ be the subset of~$[n]$  consisting of all~$p$ such that the variable $\tau_p$ enters into~$P_j$.  A variable~$\lambda_p$ enters only the entries of~$G(\bmu,\blambda)$ that are either in the $p$th row or in the $p$th column. This easily implies that the degree of any variable~$\tau_p$ in the polynomial $P$ does not exceed~$2$. Hence every~$p$ enters at most two of the subsets~$e_j$. On the other hand, the system of subsets~$e_j$ is $\mathfrak{S}_n$-invariant. Hence with any subset it necessarily contains all subsets of the same cardinality. For $n\ge 4$, this proves that the cardinality of each~$e_j$ is either~$1$ or~$n$.

Assume that the cardinality of at least one subset~$e_j$ is equal to~$1$. Then all one-element subsets of~$[n]$ appear among the subsets~$e_j$. We may assume that $e_p=\{p\}$ for $p=1,\ldots,n$. Then
$$P_p(\bmu,\btau)=a_p(\bmu)\tau_p^2+ b_p(\bmu)\tau_p+c_p(\bmu).$$
By Lemma~\ref{lem_no_tau},  $c_p(\bmu)$ is not identically zero. Besides, at least one of the two polynomials~$a_p(\bmu)$ and~$b_p(\bmu)$ is not identically zero. Hence the degree of every variable~$\tau_p$ in the polynomial $\widetilde{P}(\bmu,\btau)=P_{n+1}(\bmu,\btau)\cdots P_{l}(\bmu,\btau)$ does not exceed~$1$.  On the other hand, the total degree of any monomial of~$P(\bmu,\btau)$ with respect to  $\tau_1,\ldots,\tau_n$ is at least~$n$. Since $c_p(\bmu)$ is non-zero for $p=1,\ldots,n$, it follows that the total degree of any monomial of~$\widetilde{P}(\bmu,\btau)$ with respect to $\tau_1,\ldots,\tau_n$ is at least~$n$. Hence $\widetilde{P}(\bmu,\btau)=d(\bmu)\tau_1\cdots\tau_n$, which is impossible by Lemma~\ref{lem_no_tau}.

Thus the cardinality of every~$e_j$ is equal to~$n$. This means that every variable~$\tau_p$ enters every polynomial~$P_j(\bmu,\btau)$. Since the degree of~$P$ with respect to every~$\tau_p$ does not exceed~$2$, we see that $l\le 2$. 

Now, we need to prove that the polynomials~$P_1$ and~$P_2$ are not proportional to each other if $l=2$. Assume that they are proportional. Then we may assume that $P_1=P_2$. Hence
$$
\det G(\bmu,\blambda)=(\lambda_1\cdots\lambda_n)^{-2}f(\bmu)P_1(\bmu,\lambda_1^2,\ldots,\lambda_n^2)^2.
$$
This yields a contradiction, since by Lemma~\ref{lem_rat_S} there exist $\bmu^*$, $\blambda^+$, and~$\blambda^-$ such that  $\det G(\bmu^*,\blambda^+)>0$ and $\det G(\bmu^*,\blambda^-)<0$.  Hence $P_1$ and~$P_2$ are not proportional to each other. 
\end{proof}

\begin{lem}\label{lem_exist_11}
For each $n\ge 2$ and for each row of signs $(\varepsilon_1,\ldots,\varepsilon_n)\in\{-1,1\}^n$, there exist rows of parameters 
$\bmu^0=(\mu_1^0,\ldots,\mu_n^0)$ and~$\blambda^0=(\lambda_1^0,\ldots,\lambda_n^0)$ such that $\sign\mu^0_p=\varepsilon_p$, $p=1,\ldots,n$, the pair $\bigl(G(\bmu^0,\blambda^0),H(\bmu^0,\blambda^0)\bigr)$ belongs to~$\Psi(\E^n)$, and the gradient $\nabla\det G(\bmu,\blambda)|_{(\bmu^0,\blambda^0)}$ is non-zero. 
\end{lem}

\begin{proof}
Let $\bmu^*$, $\blambda^+$, and~$\blambda^-$ be the rows of parameters constructed in the proof of Lemma~\ref{lem_rat_S}, and let the rows~$\btau^{\pm}$ be defined by~$\tau_p^{\pm}=(\lambda_p^{\pm})^2$. 
Consider the space~$\R^{2n}$ with coordinates $\mu_1,\ldots,\mu_n$, $\tau_1,\ldots,\tau_n$. 
 Let~$T_p\subset\R^{2n}$ be the hyperplane given by $\tau_p=0$, $p=1,\ldots,n$. Represent the rational function~$f(\bmu)$ in~\eqref{eq_D_decomp} by the irreducible fraction $f_1(\bmu)/f_2(\bmu)$, where $f_1,f_2\in\R[\bmu]$. Let~$F_1,F_2\subset\R^{2n}$ be  the affine varieties given by the equations $f_1(\bmu)=0$ and $f_2(\bmu)=0$ respectively. Let  $X\subset\R^{2n}$ be the affine variety given by the equation $P(\bmu,\btau)=0$. We denote by~$S$ the union of the varieties $T_1,\ldots,T_n$, $F_1$, $F_2$, and~$X$. Then the function~$D(\bmu,\btau)$ is defined and is non-zero in the open set~$\R^{2n}\setminus S$. 
{\sloppy

}

Let $U\subset\R^{2n}$ be the connected component of~$\R^{2n}\setminus S$ containing the point $(\bmu^*,\btau^+)$. Since $D(\bmu^*,\btau^+)>0$, we have $D(\bmu,\btau)>0$ in~$U$. Therefore the point $(\bmu^*,\btau^-)$ does not belong to~$U$. The boundary~$\partial U$ of~$U$ is a piecewise smooth surface contained in~$S$. Now, notice that the union of the hyperplanes~$T_p$, $p=1,\ldots,n$, and the varieties~$F_1$ and~$F_2$ does not separate the points $(\bmu^*,\btau^+)$ and $(\bmu^*,\btau^-)$ from each other. Hence the surface $\partial U$ contains a smooth $(2n-1)$-dimensional disc~$\mathbb{D}$ that is contained in~$X$. We may assume that~$\mathbb{D}$ is contained in the smooth part of the irreducible (over~$\R$) component $X_1$ of~$X$ given by the equation $P_1(\bmu,\btau)=0$, and $\mathbb{D}$ does not intersect other irreducible components of~$S$. By Lemma~\ref{lem_P1P2},  $P(\bmu,\btau)$ is not divisible by the square of $P_1(\bmu,\btau)$. Therefore the gradient of~$D(\bmu,\btau)$ is non-zero in~$\mathbb{D}$. 
%(Here and further a ``generic point'' means a point in the intersection of~$\mathbb{D}$ with some Zariski open subset of~$X_1$.) 

We shall denote by~$\sqrt{\btau}$ the row $(\sqrt{\tau_1},\ldots,\sqrt{\tau_p})$. Then $\sqrt{\btau^{\pm}}=\blambda^{\pm}$. Since $G(\bmu^*,\blambda^+)$ is positive definite and $U$ is connected, we obtain that $G(\bmu,\sqrt{\btau})$ is positive definite for any $(\bmu,\btau)\in U$. Hence  $G(\bmu,\sqrt{\btau})$ is degenerate positive semidefinite for any $(\bmu,\btau)\in \mathbb{D}$. Let $D_p(\bmu,\btau)$ be the minor of $G(\bmu,\sqrt{\btau})$ of size $(n-1)\times(n-1)$ obtained by deleting the $p$th row and the $p$th column. Then $D_p(\bmu,\btau)$ is a rational function in~$\bmu$ and~$\btau$ independent of the two variables~$\mu_p$ and~$\tau_p$.
By Lemma~\ref{lem_P1P2},  $P_1(\bmu,\btau)$ depends essentially on every variable~$\tau_p$, $p=1,\ldots,n$. Hence, for each $p$, neither  the numerator  nor the denominator of $D_p(\bmu,\btau)$ can be divisible by $P_1(\bmu,\btau)$. Therefore $D_p(\bmu,\btau)$ is defined and is non-zero in a dense Zariski open subset $W_p\subset X_1$. Let $W=W_1\cap\dots\cap W_n$. Then  the matrix $G(\bmu,\sqrt{\btau})$ has positive principal minors of size $(n-1)\times (n-1)$ for all $(\bmu,\btau)\in\mathbb{D}\cap W$.

For each $p$, we denote by~$D'_p(\bmu,\btau)$ the determinant of the matrix obtained from $G(\bmu,\sqrt{\btau})$ by replacing its $p$th row with the $p$th row of~$H(\bmu,\sqrt{\btau})$. Then $D'_p$ is a rational function in~$\bmu$ and~$\btau$. Assume that there exists~$p$ such that~$D'_p$ is identically zero in~$X_1$. The action of the group~$\mathfrak{S}_n$ permutes  the rational functions $D'_1,\ldots,D'_n$. If $l=1$, then the variety $X_1=X$ is $\mathfrak{S}_n$-invariant. Hence we obtain that all functions $D'_1,\ldots,D'_n$ are identically zero in~$X_1$. If $l=2$, then the variety $X=X_1\cup X_2$ is $\mathfrak{S}_n$-invariant. Hence the variety $X_1$ is invariant under the action of the alternating subgroup $\mathfrak{A}_n\subset \mathfrak{S}_n$ that acts transitively on~$[n]$. (Recall that $l=1$ if $n=2$, hence, $n\ge 3$ in the case under consideration.) Thus we again obtain that all functions $D'_1,\ldots,D'_n$ are identically zero in~$X_1$. 

The rows of $G(\bmu,\sqrt{\btau})$ are linearly dependent for all $(\bmu,\btau)\in X_1$. In the dense Zariski open subset $W\subset X_1$, any $n-1$ rows of $G(\bmu,\sqrt{\btau})$ are linearly independent. Since all $D_p'$ are identically zero in~$X_1$, it follows that all rows of $H(\bmu,\sqrt{\btau})$ belong to the span of the rows of~$G(\bmu,\sqrt{\btau})$. Now, we introduce the matrix $R(\bmu,\sqrt{\btau})=G(\bmu,\sqrt{\btau})-\frac12H(\bmu,\sqrt{\btau})$. Then $\det R(\bmu,\sqrt{\btau})$ is identically zero in~$W$, hence, in~$X_1$. The entries of $R(\bmu,\sqrt{\btau})$ are given by $r_{pq}=\frac{\lambda_q\mu_q}{\lambda_p(\mu_p+\mu_q)}$. Hence $\det R(\bmu,\sqrt{\btau}) =\det R^0(\bmu)$, where $r^0_{pq}=\frac{\mu_q}{\mu_p+\mu_q}$. It is easy to check that the determinant $\det R^0(\bmu)$ considered as a function in independent variables $\mu_1,\ldots,\mu_n$ is not identically zero. Hence $\det R^0(\bmu)$ is not identically zero in~$X_1$. Thus none of the functions~$D'_p(\bmu,\btau)$ is identically zero in~$X_1$. Therefore there is a dense Zariski open subset $W'\subset X_1$ such that no row of~$H(\bmu,\sqrt{\btau})$ is a linear combination of the rows of~$G(\bmu,\sqrt{\btau})$ for any $(\bmu,\btau)\in W'$. Then the pair $(G(\bmu,\sqrt{\btau}),H(\bmu,\sqrt{\btau}))$ belongs to~$\Psi(\E^n)$ for any $(\bmu,\btau)\in \mathbb{D}\cap W\cap W'$. Hence the corresponding pair~$(\bmu,\sqrt{\btau})$ can be taken for $(\bmu^0,\blambda^0)$.
\end{proof}

\subsection{Rational families of arbitrary types $(n_1,\ldots,n_m)$}
For an arbitrary decomposition $\I\colon [n]=I_1\sqcup\cdots\sqcup I_m$ such that $|I_j|=n_j$, $m\ge 2$, we prove the following lemma, which strengthens Lemma~\ref{lem_exist}. For  $(n_1,\ldots,n_m)=(1,\ldots,1)$ this lemma coincides with Lemma~\ref{lem_exist_11}.

\begin{lem}\label{lem_exist_rat}
For  each partition $(n_1,\ldots,n_m)$ of\/~$n$ such that $m\ge 2$,  and for each row of signs $(\varepsilon_1,\ldots,\varepsilon_m)\in\{-1,1\}^n$, there exist sets of parameters 
$$\bmu^0=(\mu_1^0,\ldots,\mu_m^0),\qquad\blambda^0=(\lambda_1^0,\ldots,\lambda_n^0),\qquad\mathbf{g}^0=(g_{pq}^0\mid j(p)=j(q))$$ such that $\sign\mu^0_j=\varepsilon_j$, $j=1,\ldots,m$, the pair $\bigl(G(\bmu^0,\blambda^0,\mathbf{g}^0),H(\bmu^0,\blambda^0,\mathbf{g}^0)\bigr)$ belongs to~$\Psi(\E^n)$, and the gradient $\nabla\det G(\bmu,\blambda,\mathbf{g})|_{(\bmu^0,\blambda^0,\mathbf{g}^0)}$ is non-zero. 
\end{lem}

In fact, the required triple $(\bmu^0,\blambda^0,\mathbf{g}^0)$ will be constructed so that all $g_{pq}^0$ such that $j(p)=j(q)$ and $p\ne q$ are equal to each other and close to~$1$. We put $g_{pq}(x)=1-x$ whenever  $j(p)=j(q)$ and $p\ne q$, and introduce the notation
$$G(x,\bmu,\blambda)=G(\bmu,\blambda,\mathbf{g}(x)),\qquad H(x,\bmu,\blambda)=H(\bmu,\blambda,\mathbf{g}(x)).$$

The proof of the following lemma is straightforward.

\begin{lem}\label{lem_det_G}
Let $\underline{G}$ be an arbitrary matrix of size $m\times m$ with units on the diagonal. Let $\underline{D}=\det \underline{G}$. For $j=1,\ldots, m$, let $\underline{D}_j$ be the minor of~$\underline{G}$ of size $(m-1)\times (m-1)$ obtained by deleting the $j$th row and the $j$th column.
Let $G=G(x)$ be the matrix of size $n\times n$ such that $g_{pp}=1$ for all\/~$p$, $g_{pq}=1-x$ if\/ $j(p)=j(q)$ and\/ $p\ne q$, and\/ $g_{pq}=\underline{g}\vphantom{g}_{j(p)j(q)}$ if\/ $j(p)\ne j(q)$. Then the derivatives of\/ $\det G(x)$ at\/~$0$ are given by
\begin{equation*}
(\det G)^{(k)}(0)=
\left\{
\begin{aligned}
&0&&\text{if\/ $0\le k<n-m$,}\\
&(n-m)!\,n_1\cdots n_m\underline{D}&&\text{if\/ $k=n-m$,}\\
&(n-m+1)!\,n_1\cdots n_m\left(\frac{\underline{D}_1}{n_1}+\cdots+\frac{\underline{D}_m}{n_m}
\right)
&&\text{if\/ $k=n-m+1$.}
\end{aligned}\right.
\end{equation*}
\end{lem}

For each pair $(\underline{\bmu},\underline{\blambda})\in\R^{2m}$, let $\underline{G}(\underline{\bmu},\underline{\blambda})$ and $\underline{H}(\underline{\bmu},\underline{\blambda})$ be the matrices of sizes $m\times m$  given by formulae~\eqref{eq_g_rat} and~\eqref{eq_h_rat} respectively with $n$ replaced by~$m$, and all~$\mu_p$ and~$\lambda_p$ replaced by $\underline{\mu}\vphantom{\underline{\lambda}}_p$ and~$\underline{\lambda}_p$ respectively. Let $(\underline{\bmu}^0,\underline{\blambda}^0)\in\R^{2m}$ be the point  in Lemma~\ref{lem_exist_11}. Since $\nabla\det\underline{G}$ is non-zero at $(\underline{\bmu}^0,\underline{\blambda}^0)$, there exists a smooth curve $(\underline{\bmu}(y),\underline{\blambda}(y))$ in~$\R^{2m}$ such that $\underline{\bmu}(0)=\underline{\bmu}^0$, $\underline{\blambda}(0)=\underline{\blambda}^0$, and $\frac{d}{dy}\det\underline{G}\bigl(\underline{\bmu}(y),\underline{\blambda}(y)\bigr)|_{y=0}>0$. We denote the matrix $\underline{G}\bigl(\underline{\bmu}(y),\underline{\blambda}(y)\bigr)$ by~$\underline{G}(y)$. Now, we define the matrix $G(x,y)$ of size $n\times n$ by
\begin{equation*}
g_{pq}(x,y)=\left\{
\begin{aligned}
&1&&\text{if\/ $p=q$,}\\
&1-x&&\text{if\/ $p\ne q$ and\/ $j(p)=j(q)$,}\\
&\underline{g}_{j(p)j(q)}(y)&&\text{if\/ $j(p)\ne j(q)$.}
\end{aligned}
\right.
\end{equation*}
Put $D(x,y)=\det G(x,y)$. Lemma~\ref{lem_det_G} implies that \begin{equation*}
D_{x^k}(0,y)=\left\{
\begin{aligned}
&0&&\text{if\/ $0\le k<n-m$,}\\
&(n-m)!\,n_1\cdots n_m\,\underline{D}(y)&&\text{if\/ $k=n-m$,}\\
&(n-m+1)!\,n_1\cdots n_m\left(\frac{\underline{D}_1(y)}{n_1}+\cdots+\frac{\underline{D}_m(y)}{n_m}
\right)
&&\text{if\/ $k=n-m+1$,}
\end{aligned}
\right.
\end{equation*}
where the subscript~$x^k$ denotes the partial derivative.
Since $\underline{D}(0)=0$ and $s=\frac{d\underline{D}}{dy}(0)> 0$, we obtain that
all partial derivatives of~$D$ up to order~$n-m$ vanish at~$(0,0)$, 
\begin{align}\label{eq_partial1}
D_{x^{n-m+1}}(0,0)&=(n-m+1)!\,n_1\cdots n_m\left(\frac{\underline{D}_1(0)}{n_1}+\cdots+\frac{\underline{D}_m(0)}{n_m}\right),\\
D_{x^{n-m}y}(0,0)&=(n-m)!\,n_1\cdots n_m\,s,\label{eq_partial2}
\end{align}
and all other partial derivatives of~$D$ of order~$n-m+1$ vanish at~$(0,0)$. Therefore in a neighborhood of~$(0,0)$ the equation $D(x,y)=0$ has a smooth solution $y=y(x)$ such that 
\begin{equation*}
y'(0)=-\frac{1}{s}\left(\frac{\underline{D}_1(0)}{n_1}+\cdots+\frac{\underline{D}_m(0)}{n_m}\right).
\end{equation*}

\begin{lem}\label{lem_x_minors}
There exists a number $x_0>0$ such that, for all $x\in (0,x_0)$, all proper principal minors of $G(x,y(x))$ are positive, and $D_x(x,y(x))>0$. 
\end{lem}
\begin{proof}
Let $\widetilde{D}(x,y)$ be a principal minor of~$G(x,y)$ of size $\tilde{n}\times\tilde{n}$ formed by the rows and the columns of numbers $p_1,\ldots,p_{\tilde n}$, where $\tilde n<n$. For each $j$, we denote by~$\tilde n_j$ the number of~$p_i$ such that $j(p_i)=j$. Then $\tilde n_j\le n_j$ for all~$j$, and at least one of these inequalities is strict. We consider  two cases:

1. Assume that all  $\tilde n_j$ are  positive. Applying Lemma~\ref{lem_det_G} to the minor  $\widetilde{D}(x,y)$, we obtain that all partial derivatives of $\widetilde{D}(x,y)$ up to order $\tilde n-m$ vanish at~$(0,0)$, 
\begin{align*}
D_{x^{\tilde{n}-m+1}}(0,0)&=(\tilde n-m+1)!\,\tilde n_1\cdots \tilde n_m\left(\frac{\underline{D}_1(0)}{\tilde{n}_1}+\cdots+\frac{\underline{D}_m(0)}{\tilde{n}_m}\right),\\
D_{x^{\tilde{n}-m}y}(0,0)&=(\tilde n-m)!\,\tilde n_1\cdots \tilde n_m\,s,
\end{align*}
and all other partial derivatives of $\widetilde{D}(x,y)$ of order $\tilde{n}-m+1$ vanish at~$(0,0)$. Hence $\frac{d^k}{dx^k}(\widetilde{D}(x,y(x)))|_{x=0}=0$ for $k=0,\ldots,\tilde n-m$, and
$$
\left.\frac{d^{\tilde n-m+1}}{dx^{\tilde n-m+1}}\bigl(\widetilde{D}(x,y(x))\bigr)\right|_{x=0}=(\tilde n-m+1)!\,\tilde n_1\cdots \tilde n_m
\sum_{j=1}^l\left(\frac{1}{\tilde n_j}-\frac{1}{n_j}\right)\underline{D}_j(0)>0.
$$
Therefore $\widetilde{D}(x,y(x))>0$ for sufficiently small positive $x$.

2. Assume that at least one of the numbers~$\tilde n_j$ is equal to zero. Let $J\subset [m]$ be the subset consisting of all~$j$ such that $\tilde n_j>0$, and let $\tilde{m}=|J|$. Let $\underline{\widetilde{D}}(y)$ be the principal minor of size $\tilde{m}\times\tilde{m}$ of~$\underline{G}(y)$ formed by the rows and the columns of numbers in~$J$. Then $\underline{\widetilde{D}}(0)>0$. Applying Lemma~\ref{lem_det_G} to the minor  $\widetilde{D}(x,y)$, we obtain that all partial derivatives of $\widetilde{D}(x,y)$ up to order $\tilde{n}-\tilde{m}$ vanish at~$(0,0)$, except the derivative
$$
\widetilde{D}_{x^{\tilde{n}-\tilde{m}}}(0,0)=(\tilde{n}-\tilde{m})\Bigl(\prod_{j\in J}\tilde{n}_j\Bigl)\underline{\widetilde{D}}(0)>0.
$$
Hence $\widetilde{D}(x,y(x))>0$ for sufficiently small positive $x$.

Now, let us consider the function $\Phi(x)=D_{x}(x,y(x))$. Then $\Phi^{(k)}(0)=0$ for $k\!=0,\ldots,n-m-1$. By~\eqref{eq_partial1}, \eqref{eq_partial2}, we obtain
\begin{multline*}
\Phi^{(n-m)}(0)=\left(\left(\partial_{x}+y'(x)\partial_{y}\right)^{n-m}\partial_x D
\right)(0,0)=D_{x^{n-m+1}}(0,0)\\{}+(n-m)y'(0)D_{x^{n-m} y}(0,0)=(n-m)!\,n_1\cdots n_m\left(\frac{\underline{D}_1(0)}{n_1}+\cdots+\frac{\underline{D}_m(0)}{n_m}\right)>0.
\end{multline*}
Hence $\Phi(x)>0$ for sufficiently small positive $x$.
\end{proof}
 
In the space~$ \R^{1+m+n}$ with coordinates $x$, $\mu_1,\ldots,\mu_m$, $\lambda_1,\ldots,\lambda_n$, we consider the curve $(x,\bmu(x),\blambda(x))$, $x\in (0,x_0)$, where $\mu_j(x)=\underline{\mu}\vphantom{\underline{\lambda}}_j(y(x))$, $j=1,\ldots,m$, and $\lambda_p(x)=\underline{\lambda}_{j(p)}(y(x))$, $p=1,\ldots,n$. This curve is contained in the affine variety~$X$ given by the equation $\det G(x,\bmu,\blambda)=0$, since $G(x,\bmu(x),\blambda(x))=G(x,y(x))$. Since $(\det G)_{x}\ne0$ at the points $(x,\bmu(x),\blambda(x))$, $x\in (0,x_0)$, these points are smooth points of~$X$. Let $X_1$ be the  irreducible component of~$X$ that contains these points.  

\begin{lem}\label{lem_x_H}
In a  dense   open subset of~$X_1$ \textnormal{(}in the Zariski topology\/\textnormal{)}, no row of $H(x,\bmu,\blambda)$ is a linear combination of the rows of $G(x,\bmu,\blambda)$.
\end{lem} 
\begin{proof}
Let $\mathbf{h}_p(x,\bmu,\blambda)$ be the $p$th row of $H(x,\bmu,\blambda)$.  We consider two cases:

1. Assume that $n_{j(p)}\ge 2$. Let $q\in[n]$ be a number such that $j(q)=j(p)$ and $q\ne p$. Since $(\det G)_x\ne 0$ at the points $(x,\bmu(x),\blambda(x))$ for $x\in(0,x_0)$, we see that $\lambda_p-\lambda_q$ is not identically zero in~$X_1$, and  $(x,\bmu(x),\blambda(x))$ are smooth points of the subvariety $Y_{pq}\subset X_1$ given by~$\lambda_p=\lambda_q$. Let $R(x,\bmu,\blambda)$ be the matrix obtained from $G(x,\bmu,\blambda)$ by replacing its $q$th row  by the row $\mathbf{h}_p(x,\bmu,\blambda)$. By~\eqref{eq_h_e}, we have 
$$
r_{qq}=h_{pq}=\frac{2\lambda_p}{\lambda_p+\lambda_q}-\frac{2\lambda_p^2x}{\lambda_p^2-\lambda_q^2}\,,
$$ 
and the factor $\lambda_p-\lambda_q$ does not enter the denominators of all other entries of~$R(x,\bmu,\blambda)$. Hence the residue of the rational function $\det R|_{X_1}$ at the subvariety $Y_{pq}$ is equal to $-\frac{2\lambda_p^2xD_q}{\lambda_p+\lambda_q}$, where $D_q$ is the principal minor of~$G$ of size $(n-1)
\times (n-1)$ obtained by deleting the $q$th row and the $q$th column. This residue is non-zero in a neighborhood of $(x,\bmu(x),\blambda(x))$ for $x\in(0,x_0)$. Hence $\det R$ is not identically zero in~$X_1$. Since $X_1$ is irreducible, we obtain that  $\det R$ is defined and is non-zero in a dense  open subset of~$X_1$. Lemma~\ref{lem_x_minors} implies that the $q$th row of $G$ is a linear combination of all other rows of~$G$ in a dense open subset of~$X_1$. Thus  $\mathbf{h}_p$ is not a linear combination of the rows of~$G$ in a dense open subset of~$X_1$. 

2. Assume that $n_{j(p)}=1$. Then the row $\mathbf{h}_p$ is defined at $(x,\bmu(x),\blambda(x))$ for small positive $x$. Since $(\underline{\bmu},\underline{\blambda})$ is the point in Lemma~\ref{lem_exist_11}, the matrix obtained from $\underline{G}(y)$ by replacing its $j(p)$th row with the $j(p)$th row of~$\underline{H}(y)$ is non-degenerate for $y$ in a neighborhood of~$0$. Similar to  Case 2 in the proof of Lemma~\ref{lem_x_minors}, one can show  that the matrix obtained from~$G$ by replacing its $p$th row with the row $\mathbf{h}_p$ is non-degenerate at $(x,\bmu(x),\blambda(x))$ for small positive~$x$. Hence $\mathbf{h}_p$ is not a linear combination of the rows of~$G$ in a dense open subset of~$X_1$.
\end{proof}
\begin{proof}[Proof of Lemma~\ref{lem_exist_rat}]
It follows from Lemmas~\ref{lem_x_minors} and~\ref{lem_x_H} that for the required triple $(\bmu^0,\blambda^0,\mathbf{g}^0)$ one can take a triple $(\bmu,\blambda,\mathbf{g}(x))$ such that $(x,\bmu,\blambda)$ is a generic point of~$X_1$ close to a point $(x,\bmu(x),\blambda(x))$ for an $x\in (0,x_0)$. 
\end{proof} 

\subsection{Elliptic families}\label{subsec_exist_ell}  Now, let us prove Lemma~\ref{lem_exist} for $\Theta^{ell}_{\I,\alpha,m'}$. In this case, we have $\btheta=(k,\bsigma,\blambda,\mathbf{g})$. It is easy to see that for $k=1$, the matrices $G(\btheta)$ and~$H(\btheta)$ coincide up to elementary reversions with the matrices $G(\bmu,\tilde{\blambda},\mathbf{g})$ and~$H(\bmu,\tilde{\blambda},\mathbf{g})$ respectively corresponding to the rational family subject to the same decomposition~$\I$. Here the parameters $\bmu=(\mu_1,\ldots,\mu_m)$ and $\tilde{\blambda}=(\tilde{\lambda}_1,\ldots,\tilde{\lambda}_n)$ are given by 
$$
\mu_j=\varepsilon_j\exp(2\sigma_j),\qquad 
\tilde{\lambda}_p=\frac12\lambda_p^{-1}\exp(-\sigma_{j(p)}),
$$
where $\varepsilon_j=1$ for $j\le m'$ and $\varepsilon_j=-1$ for $j>m'$. Let $(\bmu^*,\tilde{\blambda}^*,\mathbf{g}^*)$ be the point constructed in Lemma~\ref{lem_exist_rat} corresponding to this choice of signs~$\varepsilon_j$. Consider the corresponding point $\btheta^*=(1,\bsigma^*,\blambda^*,\mathbf{g}^*)$, where $\sigma^*_j=\frac12\log(\varepsilon_j\mu^*_j)$, $\lambda^*_p=\frac12(\tilde{\lambda}^*_p)^{-1}\bigl(\varepsilon_{j(p)}\mu^*_{j(p)}\bigr)^{-\frac12}$\,. Then $(G(\btheta^*),H(\btheta^*))\in\Psi(\E^n)$ and the gradient of $\det G(\btheta)$ with  respect to the variables~$\bsigma$, $\blambda$, and~$\mathbf{g}$  is non-zero at~$\btheta^*$. Hence the hypersurface $\det G(\btheta)=0$ is smooth in a neighborhood of~$\btheta^*$, and contains points~$\btheta$ with $k<1$. Any such point~$\btheta$ sufficiently close to~$\btheta^*$ can be taken for~$\btheta^0$.
  
\subsection{Exotic families: Proof of Lemma~\ref{lem_exotic_empty}}\label{subsec_exist_ex}

Let us prove that any set $\Theta_{\I,\alpha}^{ex}(\bS^n)$ is non-empty. 
Let $0<k'\ll 1$. Choose  coefficients $\lambda_p$ such that $k'^{-\frac12}\le |\lambda_p|\le 2k'^{-\frac12}$ whenever $j(p)$ is either $1$ or~$2$, and $k'^{-\frac14}\le | \lambda_p|\le 2k'^{-\frac14}$ whenever $j(p)=3$. We put $g_{pq}=0$ whenever $j(p)=j(q)$ and $p\ne q$. The entries~$g_{pq}$ such that $j(p)\ne j(q)$ can be computed by substituting~\eqref{eq_a-e_ex-f}, \eqref{eq_a-e_ex-l} to~\eqref{eq_g_ne}. As $k'\to 0$, we have
$$
g_{pq}=\left\{
\begin{aligned}
&O(k'^{\frac12})&&\text{if $(j(p),j(q))=(1,2)$,}\\
&O(k'^{\frac14})&&\text{if $(j(p),j(q))=(1,3)$ or $(j(p),j(q))=(2,3)$.}
\end{aligned}
\right.
$$
Thus, for sufficiently small~$k'$, $G$ is arbitrarily close to the unit matrix. Hence the constructed point $(k,\blambda,\mathbf{g})$  belongs to~$\Theta_{\I,\alpha}^{ex}(\bS^n)$.

The proof of the fact that all sets $\Theta_{\I,\alpha}^{ex}(\E^n)$ and $\Theta_{\I,\alpha}^{ex}(\Lambda^n)$ are empty is based on the following strange fact, which can be easily obtained by substituting~\eqref{eq_a-e_ex-f}, \eqref{eq_a-e_ex-l} to~\eqref{eq_g_ne}:
\begin{lem}\label{lem_strange}
Choose any $p\in I_1$, $q\in I_2$, and $r\in I_3$. Then, for each $\btheta=(k,\blambda,\mathbf{g})$, we have $g_{pq}(\btheta)=g_{pr}(\btheta)g_{qr}(\btheta)$.
\end{lem}
 Assume that $(G(\btheta),H(\btheta))$ belongs either to~$\Psi(\E^n)$ or to~$\Psi(\Lambda^n)$. Realise the matrix $G=G(\btheta)$ as the Gram matrix of $n$ vectors $\be_1,\ldots,\be_n$ in a pseudo-Euclidean space~$U$. For each finite set of vectors~$E$, we denote by~$g(E)$ the determinant of the Gram matrix of the vectors in~$E$. Let $E=\{\be_1,\ldots,\be_n\}$. Then $g(E)\le 0$ and $g(E')>0$ for each proper subset~$E'$ of~$E$. Let $E_j$ be the set of all vectors~$\be_p$ such that $j(p)=j$, and let $L$ be the span of all vectors in~$E_3$. Since $g(E_3)>0$, we see that $U=L\oplus L^{\bot}$. For each $\be_p$ let $\be_p^{\bot}$ be the projection of~$\be_p$ to~$L^{\bot}$ along~$L$.  It follows from Lemma~\ref{lem_strange}  that  $(\be_p^{\bot},\be_q^{\bot})=0$ whenever $p\in I_1$ and $q\in I_2$. For $j=1,2$, let $E_j^{\bot}$ be the set of all vectors~$\be_p^{\bot}$ such that $j(p)=j$. Then we have 
\begin{equation*}
g(E)=g(E_1^{\bot},E_2^{\bot})g(E_3)=g(E_1^{\bot})g(E_2^{\bot})g(E_3)=\frac{g(E_1,E_3)g(E_2,E_3)}{g(E_3)}>0.
\end{equation*}
The  contradiction obtained completes the proof of the lemma.
 
\section{Conclusion}\label{section_concl}

Let us make some conclusive remarks. First, it  follows easily from the results by Bricard~\cite{Bri97} that his three types of flexible octahedra in~$\E^3$, namely, \textit{line-symmetric flexible octahedra\/}, \textit{plane-symmetric flexible octahedra\/}, and  \textit{skew flexible octahedra\/} correspond in our terminology to flexible octahedra of types $(1,1,1)$, $(2,1)$, and~$(3)$ respectively. The octahedra of the first two types generically admit elliptic parametrization, but for some special values of edge lengths the elliptic curve degenerates to a rational curve. The skew flexible octahedra admit the rational parametrization described in Section~\ref{section_simplest}. Second, it is not hard to see that Stachel's examples of flexible cross-polytopes in~$\E^4$ correspond in our classification to  rational cross-polytopes of type~$(2,2)$. Hence all other examples of flexible cross-polytopes in~$\E^4$ constructed in the present paper are new.

Now, let us formulate two problems that remain open and seem to be interesting. First, for each of the spaces~$\X^n$, and for each partition $(n_1,\ldots,n_m)$  of~$n$, one can construct the \textit{moduli space\/}~$\mathcal{M}_{n_1,\ldots,n_m}(\X^n)$ of  flexible cross-polytopes in~$\X^n$ of type $(n_1,\ldots,n_m)$, i.\,e., the set of all isometry classes of such flexible cross-polytopes with the natural topology. The structure of this space seems to be highly non-trivial: It consists of strata~$\Theta_{\I,\alpha,m'}^{ell}(\X^n)$ corresponding to different elliptic families of type $(n_1,\ldots,n_m)$ that are attached  to the rational stratum~$\Theta_{\I}^{rat}$. In the spherical case we also have exotic strata. The union of all these strata should by factorized by elementary reversions.

\begin{problem}
Describe the topology of the moduli space~$\mathcal{M}_{n_1,\ldots,n_m}(\X^n)$.
\end{problem}  

\begin{problem}
Does there exist a flexible polyhedron in~$\E^n,$ $\bS^n,$ or~$\Lambda^n$ such that its configuration space is a curve of genus greater than~$1$?
\end{problem}


\begin{thebibliography}{99}

\bibitem{Ale85} R.\,Alexander, \textit{Lipschitzian mappings and total mean curvature of polyhedral surfaces\textnormal{,} I\/}, Trans. Amer. Math. Soc. \textbf{288} (1985), 661--678.


\bibitem{Ale95} V.\,A.\,Alexandrov, \textit{A new example of a flexible polyhedron\/}, Sib. Mat. Zhurnal \textbf{36}:6 (1995), 1215--1224 (in Russian); Siberian Math. J. \textbf{36}:6 (1995), 1049--1057 (English translation).

\bibitem{BaEr55} H.\,Bateman, A.\,Erd\'elyi, \textit{Higher Transcendental Functions\/}, Volume~II, New York, Toronto, London: McGraw--Hill Book Company, 1953. 


\bibitem{Ben12} G.\,T.\,Bennett, \textit{Deformable octahedra\/}, Proc. London Math. Soc. \textbf{10} (1912), 309--343.

\bibitem{Bri97} R.\,Bricard, \textit{M\'emoire sur la th\'eorie de l'octa\`edre articul\'e\/}, J. Math. Pures Appl. \textbf{5}:3 (1897), 113--148.



\bibitem{Cau13} A.\,Cauchy, \textit{Deuxi\`eme m\'emoire sur les polygones et poly\`edres\/}, J. Ecole Polytechnique \textbf{19} (1813), 87--98.

\bibitem{Con80} R.\,Connelly, \textit{An Attack on Rigidity\/}, Preprint, 1974; in: \textit{Research on the metric theory of surfaces\/}, Moscow: Mir, 1980,  164--209 (Russian translation).

\bibitem{Con77} R.\,Connelly, \textit{A counterexample to the rigidity conjecture for polyhedra\/}, Inst. Hautes \'Etudes Sci. Publ. Math. \textbf{47} (1977), 333--338.

\bibitem{Con78} R. Connelly, \textit{Conjectures and open questions in rigidity\/}, Proc. Internat. Congress Math. (Helsinki, 1978), Acad. Sci. Fennica, Helsinki, 1980, 407--414.


\bibitem{CSW97} R.\,Connelly, I.\,Sabitov, A.\,Walz, \textit{The Bellows Conjecture\/}, Beitr. Algebra Geom. \textbf{38}:1 (1997), 1--10.


\bibitem{Gai11} A.\,A.\,Gaifullin, \textit{Sabitov polynomials for volumes of polyhedra in four dimensions\/}, Adv. Math. \textbf{252} (2014), 586--611, arXiv: 1108.6014.

\bibitem{Gai12} A.\,A.\,Gaifullin, \textit{Generalization of Sabitov's theorem to polyhedra of arbitrary dimensions}, arXiv: 1210.5408.



\bibitem{Izm} I.\,V.\,Izmestiev, \textit{Deformation of quadrilaterals and addition on elliptic curves\/} (to appear).



\bibitem{Sab96} I.\,Kh.\,Sabitov, \textit{Volume of a polyhedron as a function of its metric\/}, Fundamental and Applied Math. \textbf{2}:4 (1996), 1235--1246 (in Russian). 

\bibitem{Sab98a} I.\,Kh.\,Sabitov, \textit{A generalized Heron--Tartaglia formula and some of its consequences\/}, Mat. Sb. \textbf{189}:10 (1998), 105--134 (in Russian); Sb. Math. \textbf{189}:10 (1998), 1533--1561 (English translation).

\bibitem{Sab98b} I.\,Kh.\,Sabitov, \textit{The volume  as a  metric invariant of polyhedra\/}, Discrete Comput. Geom. \textbf{20}:4 (1998), 405--425. 

\bibitem{Sto13} M.\,Shtogrin, \textit{A flexible disk with a handle\/}, Uspekhi Mat. Nauk \textbf{68}:5(413) (2013), 177--178 (in Russian); Russian Math. Surveys, \textbf{68}:5 (2013), 951--953 (English translation).


\bibitem{Sta00} H.\,Stachel, \textit{Flexible cross-polytopes in the Euclidean $4$-space\/}, J. Geom. Graph. \textbf{4}:2 (2000), 159--167.

\bibitem{Sta06} H.\,Stachel, \textit{Flexible octahedra in the hyperbolic space\/}, In the book: \textit{Non-Euclidean geometries. J\'anos Bolyai memorial volume\/} (Eds. A. Pr\'ekopa et al.). New York: Springer. Mathematics and its Applications (Springer) \textbf{581}, 209--225 (2006).










\bibitem{WhWa27} E.\,T.\,Whittaker, G.\,N.\,Watson, \textit{A course of modern analysis}, 3rd edn., Cambridge: University Press, 1920. 




\end{thebibliography}
\end{document}